\documentclass[reqno,12pt,letterpaper]{amsart}
\usepackage[proof]{sdmacros}

\newcommand{\per}{\mathrm{per}}

\def\SD#1{\textcolor{blue}{#1}}
\DeclareMathOperator{\dvol}{dvol}
\graphicspath{{fig/}}
\title{Trace formula for billiard flow on surfaces with concave boundary}
\author{Alain Kangabire}
\address{Department of Mathematics \\ Massachusetts Institute of Technology \\ \newline 77 Massachusetts Ave, Cambridge, MA 02139}
\email{kanga27@mit.edu}

\begin{document}

\begin{abstract}
We study grazing closed trajectories of the billiard flow $\varphi^t$ on a compact surface $(\Sigma,g)$ with concave boundary. Under a non-degeneracy condition, which is satisfied for example when $(\Sigma,g)$ has negative sectional curvature, we compute the trace of a regularized pullback $E_{\epsilon}(\varphi^t)^* E_{\epsilon}$ near a closed trajectory grazing once, which is an analogue of the Atiyah--Bott--Guillemin trace formula for closed billiard trajectories.   
\end{abstract}

\maketitle

%%%%%%%%%%%%%%%%%%%%%%%%%%%%%%%%%%%%%%%%%%%%%%%%%%%%%%%%%%%%%%%%%%%%%%%%%%%%%%%%
%%%%%%%%%%%%%%%%%%%%%%%%%%%%%%%%%%%%%%%%%%%%%%%%%%%%%%%%%%%%%%%%%%%%%%%%%%%%%%%%

\section{Introduction}
\label{sec: intro}

In this paper we prove a trace formula for a class of dispersive billiards. Dispersive billiards are
a classical setting for strongly chaotic / hyperbolic dynamical systems and trace formulas are a powerful tool used for example in understanding the dynamical zeta functions and length spectrum asymptotics, see~ \cite{S67}, \cite{M69}, \cite{G77}, \cite{GLP13}, \cite{DZ16}, and \cite{DZ17}.

A standard setting of strongly chaotic dynamics is given by geodesic flows on compact negatively curved manifolds without boundary. Let $(\Sigma,g)$ be such a manifold, $M=S\Sigma$ its unit tangent bundle, and denote the geodesic flow by
$$
\varphi^t:M\to M.
$$
Consider the pullback operator $(\varphi^t)^*:C^\infty(M)\to C^\infty(M)$. It is not trace class (say, on $L^2(M)$) but one can define its \emph{flat trace} by means of a regularization procedure:
$$
\tr^{\flat} (\varphi^t)^* = \lim_{\epsilon \rightarrow 0} \tr E_{\epsilon} (\varphi^t)^* E_{\epsilon}, 
$$
where the smoothing operators $E_{\epsilon}: \mathcal D'(M) \rightarrow C^{\infty}(M)$ are defined in ~\cite[\S2.4]{DZ16} and $\tr$ denotes the trace for trace class operators.  
The Atiyah--Bott--Guillemin trace formula~(\cite{AB67}, \cite{G77}) states that, in the sense of distributions
in $\mathcal D'((0,\infty))$,
\begin{equation}
    \label{eqn: trace formula boundaryless case}
    \tr^\flat \varphi_t^*=\sum_{\gamma} \frac{T_{\gamma}^{\#} \delta(t - T_{\gamma})}{|\det(I - \mathcal P_{\gamma})|}.
\end{equation}
The sum is over closed trajectories $\gamma$ of $\varphi^t$; and $T_{\gamma}$ is the period of $\gamma$, $T^{\#}_{\gamma}$ the primitive period, and $\mathcal P_{\gamma}$ the linearized Poincar\'e map of the closed trajectory $\gamma$.

The setting of manifolds with boundary, which is the topic of this paper, presents considerable
new difficulties compared to the boundaryless case; see the works of Sinai ~\cite{S70} on mixing for billiard flows, the work of Young \cite{Y98} on exponential mixing for billiard maps
and the work of Baladi--Demers--Liverani \cite{BDL18} on exponential mixing for billiard flows. The question of trace formulas and dynamical zeta functions on closed billiards, to the best of our knowledge, is not addressed in the literature. However, there is recent work by Chaubet--Petkov \cite{CP22} and Delarue--Sch\"utte--Weich \cite{DSW24} on open billiards, as well as work by Baladi--Demers \cite{BD20} which studies closed trajectories on closed billiards without trace formulas.

The main difficulty in the case of closed billiards (typically avoided for open billiards by imposing a non-eclipse condition) is the presence of \emph{grazing rays}, which intersect the boundary tangentially.
This is particularly difficult to analyze when the length of the trajectory tends to infinity, however it is already challenging for trajectories of fixed length.

%%%%%%%%%%%%%%%%%%%%%%%%%%%%%%%%%%%%%%%%%%%%%%%%%%%%%%%%%%%%%%%%%%%%%%%%%%%%%%%%
\begin{figure}
\includegraphics[scale=0.8]{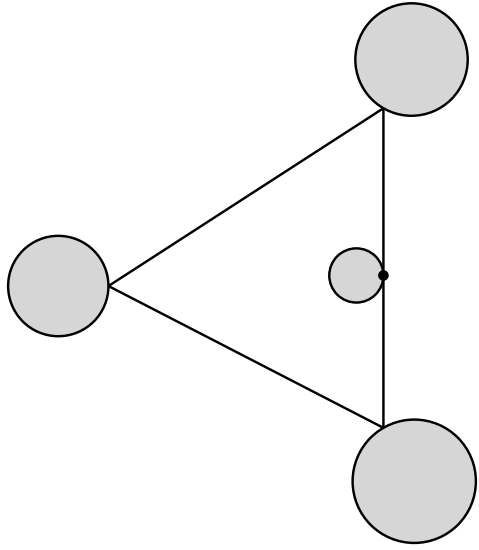}
\caption{A closed trajectory with a single grazing point.}
\label{f:grazing-intro}
\end{figure}
%%%%%%%%%%%%%%%%%%%%%%%%%%%%%%%%%%%%%%%%%%%%%%%%%%%%%%%%%%%%%%%%%%%%%%%%%%%%%%%%

In the present paper, we prove a trace formula near a fixed closed trajectory which has one grazing point. (Much of the proof applies to multiple grazing points, however we were not able to handle this case to some possibly technical difficulties, see~\S \ref{sec: difficulties for mult glancing points} for more details). 

We now describe the setting. Let $(\Sigma,g)$ be a compact manifold with concave boundary (see~\S \ref{sec: setup} for precise definitions) and $M = S \Sigma $ its sphere bundle. Let 
$$
\varphi^t: M \rightarrow M \quad (t \geq 0)
$$
be the billiard flow on $M$, meaning $\varphi^t$ is the geodesic flow in the interior of $M$ and obeys the law of reflection (the angle of incidence is equal to the angle of reflection) at the boundary $\partial M$ (refer to~\S \ref{sec: billiard flow and glan hypersurf} for formal definitions). 

The pullback operator $(\varphi^t)^*: C^{\infty}(M) \rightarrow \mathcal D'(M)$ has worse mapping properties than in the boundaryless case; but nonetheless we would like to make sense of its \textit{flat trace} as 
$$
\tr^{\flat}(\varphi^t)^* = \lim_{\epsilon \rightarrow 0} \tr E_{\epsilon} (\varphi^t)^* E_{\epsilon},
$$
where $E_{\epsilon}: \mathcal D'(M) \rightarrow C^{\infty}(M)$ are smoothing operators with similar properties as the smoothing operators in \cite[\S 2.4]{DZ16} (see \S \ref{sec: regularized flat trace} for an explicit definition of $E_{\epsilon}$). The operator $E_{\epsilon} (\varphi^t)^* E_{\epsilon}$ has a smooth Schwartz kernel; so it is indeed trace class and $\tr E_{\epsilon}(\varphi^t)^*E_{\epsilon}$ is well defined for each fixed $\epsilon > 0$. Under suitable assumptions (to be made precise below), we show that the limit above exists, as a distribution, and compute explicitly what it converges to. 

Let $\gamma$ be a closed billiard trajectory with exactly one grazing point in the sense of Lemma \ref{lemma: Poinc return map for closed traj with m glancing pts} (notice that $\gamma$ must then be primitive) and 
$$
P_{\gamma}: U_{\gamma} \rightarrow U_{\gamma}, \quad P_{\gamma}(x_0) = x_0
$$
its associated Poincar\'e return map. Here we choose the Poincar\'e section $U_{\gamma}$ in the interior of $M$. Unlike in the boundaryless case, the differential of $P_{\gamma}$ is not defined at $x_0$ or any other point $x \in U_{\gamma}$ where the billiard flow starting at $x$ intersects the boundary tangentially before the first return time.   

Because $\gamma$ has exactly one grazing point, the set of points $x \in U_{\gamma}$ where the differential $dP_{\gamma}$ is not defined is a smooth hypersurface in $U_{\gamma}$ containing $x_0$ and which splits $U_{\gamma}$ into two disjoint connected sets $U_{\gamma, 0}$ and $U_{\gamma, 1}$ (the sets $U_{\gamma, 0}$ and $U_{\gamma,1}$ are defined in Equation \eqref{eqn: partition of U based on how many boundary hitting points}). Figure \ref{fig: splitting of Poinc section} illustrates $U_{\gamma,0}$ and $U_{\gamma,1}$.
\begin{figure}[ht]
    \centering
    \includegraphics[width=0.50\linewidth]{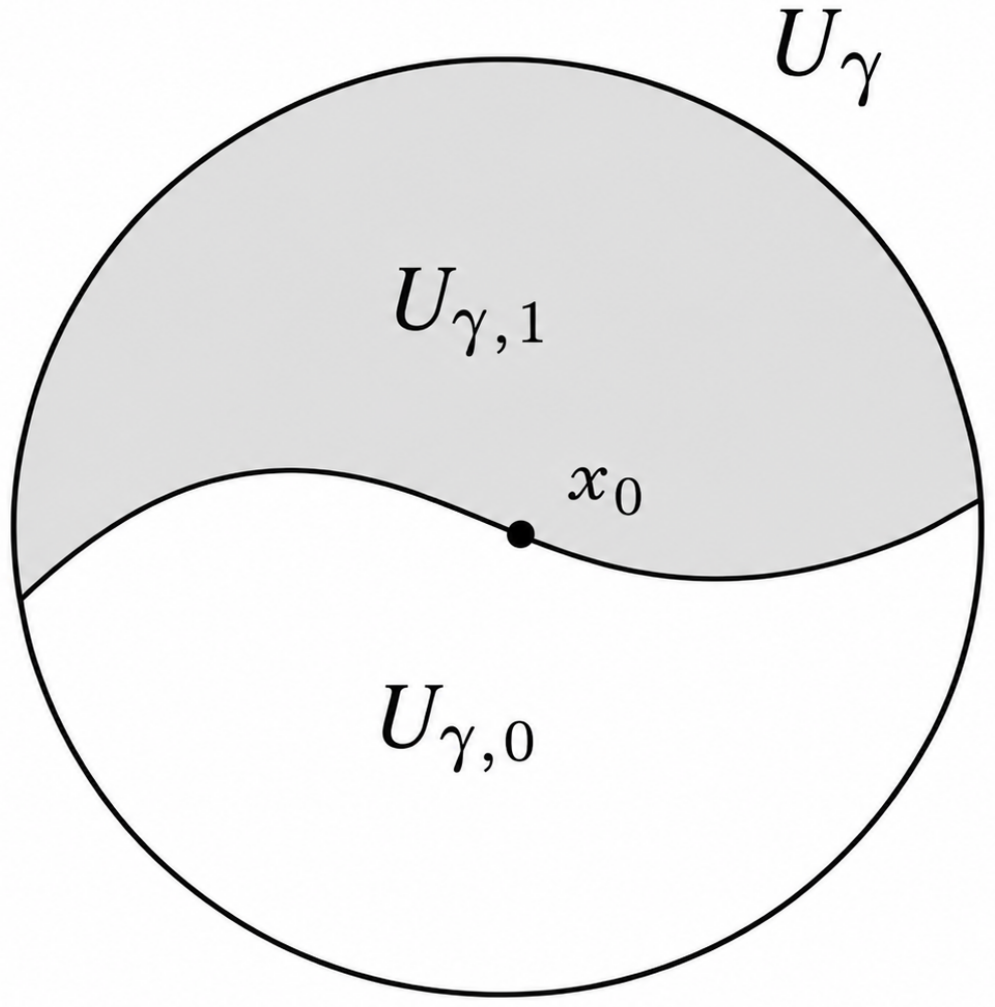}
    \caption{The Poincar\'e section $U_{\gamma}$ and the sets $U_{\gamma,0}$ and $U_{\gamma,1}$ on which the differential of $P_{\gamma}$ is defined.}
    \label{fig: splitting of Poinc section}
\end{figure}

These sets $U_{\gamma,1}$ and $U_{\gamma,0}$ have the property  that a billiard trajectory starting from a point $x \in U_{\gamma,1}$ (or $x \in U_{\gamma,0}$) intersects (or does not intersect) the boundary near the grazing point on $\gamma$ before the next intersection $P_{\gamma}(x)$ with the Poincar\'e section $U_{\gamma}$. 

Under the following \textit{non-degeneracy} assumption (made precise in Equation \eqref{eqn: curvature assumption})
\begin{equation}
\label{eqn: non-degeneracy assumption for one glanc point}
  \begin{minipage}{\textwidth}
    \begin{itemize}
        \item $(\Sigma,g)$ has dimension $n=2$ and 
        \item $(\Sigma,g)$ has negative sectional curvature or $(\Sigma,g)$ has nonpositive sectional curvature and  $\gamma$  has at least one strict reflection point on the boundary $\partial M$,
    \end{itemize}
  \end{minipage}
\end{equation}
 Proposition \ref{prop: expans of det of Poinc ret map in coord minus ident with non-zero coeff} implies that if $\psi$ is any coordinate function for $U_{\gamma}$ and $\widetilde P_{\gamma} = \psi \circ P_{\gamma} \circ \psi^{-1}$, then
 $$
 \det(I - d \widetilde P_{\gamma}) \quad \text{does not vanish on} \quad \psi(U_{\gamma,j}) \quad (j=0,1).
 $$
 
  While $|\det(I - d \widetilde P_{\gamma})|^{-1}$ is not defined at $\psi(x_0)$, yet Corollary \ref{corr: limits diff of Poinc retur map minus id based on reflect pts} shows that for all sequences $\{x_{n,j} \}_{n \in \mathbb N} \subset U_{\gamma,j}$ converging to $x_0$, the sequence ~$|\det(I - \widetilde P_{\gamma}(\psi(x_{n,j})))|^{-1}$ converges to the same limit point. So, let
 $$
 c_{\gamma,j} := \lim_{x_n \rightarrow \infty} |\det(I - d\widetilde P_{\gamma}(\psi(x_{n,j})) )|^{-1},
 $$
 which is shown in Corollary \ref{corr: limits diff of Poinc retur map minus id based on reflect pts} to be independent of the coordinate $\psi$, strictly positive if $j=0$ and zero if $j=1$. 
 
 For a closed billiard trajectory with one grazing point, these constants $c_{\gamma, 0}>0$ and $c_{\gamma, 1}=0$ are the correct replacement for the term $|\det(I-\mathcal P_{\gamma})|^{-1}$ in the boundaryless trace formula \eqref{eqn: trace formula boundaryless case}. Precisely, we state our main theorem.
\begin{theo}
    \label{thm: trace formula near fixed closed traj with one grazing point}
    Let $f \in C_c^{\infty}((0,\infty))$. If the non-degeneracy assumption in \eqref{eqn: non-degeneracy assumption for one glanc point} holds and all closed billiard trajectories $\gamma$ with period $T_{\gamma} \in \supp f$ have one grazing point (in the sense of Lemma \ref{lemma: Poinc return map for closed traj with m glancing pts}), then 
    \begin{equation}
        \label{eqn: trace formula billiard flow with one glanc point}
        \lim_{\epsilon \rightarrow 0} \int_0^{\infty} f(t) \left(\tr E_{\epsilon} (\varphi^t)^* E_{\epsilon} \right) dt = \frac{1}{2}\sum_{\gamma}  c_{\gamma, 0} \,T_{\gamma}f(T_{\gamma}).
    \end{equation}
\end{theo}

Intuitively, Theorem \ref{thm: trace formula near fixed closed traj with one grazing point} says that for each closed billiard trajectory $\gamma$ with one grazing point, any $x \in U_{\gamma,j}$ that is $\epsilon$ close to $P_{\gamma}(x)$ contributes roughly a multiple of $c_{\gamma,j}$ to the trace $\tr E_{\epsilon}(\varphi^t)^* E_{\epsilon}$. The Schwartz kernel of the regularizer $E_{\epsilon}$ depends more or less on the distance between its inputs, so this symmetry and $c_{\gamma,1} = 0$ explain the $\frac{1}{2}$ factor in the trace formula \eqref{eqn: trace formula billiard flow with one glanc point}. 

\bigskip
\noindent
\textbf{Organization of the paper.} \S \ref{sec: setup} introduces a manifold with concave boundary while \S \ref{sec: billiard flow and glan hypersurf} defines the billiard flow $\varphi^t$ and studies the Poincar\'e return map $P_{\gamma}$ of a closed trajectory $\gamma$ with $m$ ($m \geq 1$) grazing points (in the sense of Lemma \ref{lemma: Poinc return map for closed traj with m glancing pts}) using glancing hypersurfaces.

In  \S \ref{sec: non-deg assumption}, we discuss \textit{non-degeneracy} assumptions guaranteeing that the determinant $\det(I-d\widetilde P_{\gamma})$ does not vanish on a subset of full measure in $\psi(U_{\gamma})$ (we present the argument when $\gamma$ has $m$ grazing points). For a closed trajectory $\gamma$ with one grazing point, we prove that the map $\widetilde P_{\gamma} - I$ is injective, something we are unable to show in the case of multiple grazing points, therefore hindering us from proving Theorem \ref{thm: trace formula near fixed closed traj with one grazing point} when the closed trajectory $\gamma$ has more than one grazing point.  

We define in \S \ref{sec: regularized flat trace} the regularized flat trace of the billiard flow and prove an important lemma (Lemma \ref{lemma: local trace formula for regularized billiard flow on fns}) instrumental in the proof of Theorem \ref{thm: trace formula near fixed closed traj with one grazing point} presented in this same section. Finally, we end the paper by discussing in \S \ref{sec: difficulties for mult glancing points} the difficulties we face to upgrade Theorem \ref{thm: trace formula near fixed closed traj with one grazing point} to handle closed trajectories with multiple grazing points.

\noindent\textbf{Acknowledgements}. 
I would like to thank Semyon Dyatlov for invaluable discussions and Richard Melrose for discussions on glancing hypersurfaces. My deepest gratitude is also extended to Maciej Zworski for hosting me as a visiting researcher while working on some part of this project. This work was partially supported by the NSF CAREER grant DMS-1749858, NSF grant DMS-2400090, and Simons Foundation award SFI-MPS-SFM-00005584. 

\noindent
\textbf{Disclaimer on AI use:} Figures in this paper were created with the help of ChatGPT. 

\section{Setup}
\label{sec: setup}
Consider a smooth compact Riemannian manifold with boundary $(\Sigma,g)$. We can embed $(\Sigma,g)$ in a smooth Riemannian manifold $(\widetilde{\Sigma},g)$ without boundary, and we specify a boundary defining function $b \in C^{\infty}(\widetilde \Sigma)$ with the following properties:
\begin{equation}
    \label{eqn: bound def fn}
    \begin{split}
        \Sigma &= \{b \geq 0\}, \\
        b &> 0 \text{ on } \Sigma^{\circ},\\
        b &= 0 \text{ on } \partial \Sigma, \text{ and }\\
        db &\neq 0 \text{ on } \partial \Sigma.
    \end{split}
\end{equation}

Furthermore, we require the boundary $\partial \Sigma$ to be concave. In terms of the boundary defining function $b$, it means that the Hessian $\nabla^2 b$ satisfies for some constant $c>0$
\begin{equation}
    \label{eqn: concave boundary}
     \langle \nabla^2 b, v \otimes v \rangle_{g(p)} > c \| v \|^2_{g(p)} \text{ for all } (p,v) \in T(\partial \Sigma),\ v\neq 0
\end{equation}
where $\nabla$ is the Levi-Civita connection. 

\begin{Remark}
    In fact, because $\partial \Sigma$ is compact, one can show that there exists $\delta > 0$ such that  equation \eqref{eqn: concave boundary} holds for all $(p,v) \in T\widetilde{\Sigma}$ with $|b(p)| \leq \delta$ and $|db(v)| \leq \delta\|v \|_{g(p)}$.
\end{Remark}
 
The condition \eqref{eqn: concave boundary} implies that the second fundamental form
\begin{equation}
\label{eqn: 2nd fund form}
 \Pi (v,v)= -\frac{\langle \nabla^2 b, v \otimes v \rangle_{g(p)}}{\|\nabla b\|_{g(p)}^2} \nabla b 
\end{equation}
has a negative sign relative to $\nabla b$, which is inward pointing by Definition \eqref{eqn: bound def fn}. 

\section{Billiard Flow and Glancing Hypersurfaces}
\label{sec: billiard flow and glan hypersurf}
\subsection{Billiard Flow}
\label{subsec: sympl billiard flow}

Consider the tangent bundle  
\begin{equation}
\label{eqn: cotangent bundle with proj map}
    N = T \widetilde \Sigma \quad \text{with projection map} \quad \pi: N \rightarrow \widetilde \Sigma.
\end{equation}
We equip the tangent bundle $N$ defined by Equation \eqref{eqn: cotangent bundle with proj map} with the symplectic structure provided in Section \S1.3 of \cite{GP99}. Set 
\begin{equation}
    \label{eqn: f covector with norm 1}
    f(x) = \frac{1}{2} \left(\|  v \|_{g(p)}^2 -1\right), 
\end{equation}
where $x=(p,v) \in N$ and $g$ is the metric assigned to $\widetilde \Sigma$. Define the hypersurfaces 
\begin{equation}
    \label{eqn: hypersurfaces F and G}
    F = S\widetilde\Sigma = \{ f = 0 \} \quad \text{and} \quad G = \{ \pi^*b = 0 \},
\end{equation}
where
% $f$ is the function given by Equation \eqref{eqn: f covector with norm 1} and
$b$ is the boundary defining function for $\partial \Sigma$ given by Equation \eqref{eqn: bound def fn}.

\begin{lemm}
    \label{lemma: diff of def fn for F and G are lin. indep}
    % Suppose that the hypersurfaces $F = \{f = 0\}$ and $G = \{ \pi^*b =0 \}$ are defined by Equation \eqref{eqn: hypersurfaces F and G}. 
     If $x \in F \cap G$, then the differentials $df(x)$ and $d(\pi^*b)(x)$ are linearly independent.  
\end{lemm}
\begin{proof}
    To see this, assume the contrary, that is there exist $x =(p,v)\in F \cap G$ and constants $c_1$ and $c_2$ not both zero such that $c_1df(x) = c_2 d(\pi^*b)(x)$. 
    %If we use the coordinate $(p,v)$ for $T \widetilde \Sigma$ at $x$, where here $p \in \widetilde \Sigma$ and $v \in T_p \widetilde \Sigma$,
    Then $0 = c_2 d(\pi^* b)(\partial_{v_j}) = c_1 df(\partial_{v_j})$, but since $df(\partial_{v_j}) \neq 0$ for some $j$, we get that $c_1 = 0$.
    % since both $c_1$ and $c_2$ cannot be zero, we have that $c_2 \neq 0$, and 
    Then $c_2\neq 0$ and thus $(\pi^*db)(x) = 0$, which gives a contradiction to the assumption \eqref{eqn: bound def fn} which says that $db \neq 0$ when $b=0$.
\end{proof}

% If $f$ is the function defined by Equation \eqref{eqn: f covector with norm 1}, then 
The Hamiltonian vector field $H_f$ restricted to the hypersuface $F = S \widetilde \Sigma$
% defined by Equation \eqref{eqn: hypersurfaces F and G}
is the generator of the geodesic flow \cite[\S1.3]{GP99}. In the coordinates $(p,v)$ for $N = T \widetilde \Sigma$, one has that the Hamiltonian vector field 
\begin{equation}
    \label{eqn: hamil vect field of p}
    H_{\pi^* b} = - \nabla b \cdot \partial_v,
\end{equation}
and consequently one can compute that
\begin{equation}
    \label{eqn: poisson bracket of f and p}
    H_{\pi^*b}f = - \sum_{j,l} g_{jl} (\nabla b)_j v_l  = -\left \langle \nabla b, v \right \rangle_{g(p)}  
\end{equation}
and that 
\begin{equation}
    \label{eqn: 2nd der of f wrt hamilt flow of p}
    H_{\pi^*b}^2 f = \| \nabla b \|_{g(p)}^2.
\end{equation}

 Since $H_f\pi^*b = - H_{\pi^*b}f$, we also have that 
\begin{equation}
    \label{eqn: 2nd der of p wrt hamil flow of f}
    \begin{split}
        H_f^2 \pi^*b &=  H_f\left(\left \langle \nabla b, v \right \rangle_{g} \right)  = \left \langle \nabla^2 b, v \otimes v \right \rangle_{g(p)}
    \end{split},
\end{equation}
where the last equality follows because $H_f$ is the generator of the geodesic flow.

Since the boundary $\partial\Sigma$ is concave, we see that any geodesic
starting at the boundary which points tangent to the boundary or inside $\Sigma$
will be in the interior of $\Sigma$ for small positive times:
%We now show the following: 
\begin{lemm}
    \label{lemma: uniform eps where hamil of f inside F and has not hit G}
    Let the hypersurfaces $F = \left\{ f = 0 \right\}$ and $G = \left\{ \pi^*b = 0 \right\}$ be given by Equation \eqref{eqn: hypersurfaces F and G}. Then, there exists $\epsilon_{0} >0$  such that $\pi^*b\left(e^{tH_f}(x)\right) > 0$ for all $t \in (0,\epsilon_0]$ and $x \in F \cap G \cap \{ H_f \pi^*b \geq 0 \}$.
\end{lemm}
\begin{proof}
    Suppose that $x \in F \cap G \cap \{ H_f \pi^*b \geq 0 \}$. We have that 
    \begin{equation}
    \label{eqn: second order approx of function pi b}
        \pi^*b \left(e^{t H_f}(x)\right) = H_f \pi^*b (x) t  + t^2 \int^{1}_0 (1-s)\left(H_f^2 \pi^*b\right)\left(e^{st H_f}(x)\right) ds. 
    \end{equation}
    Now we proceed with cases. 
    
    \noindent \textbf{Case 1: } Assume that $H_f\pi^*b (x) > 0$.
    
     It follows that there exist an open set $U \subset F \cap G$ containing $x$ and a constant $c>0$ such that $H_f \pi^*b (y) >c$ for all $y \in U$. The compactness of $F \cap \{\pi^*b \geq 0 \} = S \Sigma$ gives that $|H_f^2 \pi^*b|$ is bounded on $F \cap \{ \pi^*b \geq 0\}$, so using Equation \eqref{eqn: second order approx of function pi b}, we get $\epsilon_0 > 0$ such that for all $y \in U$ and $t \in (0,\epsilon_0]$
    $$
    \pi^*b\left(e^{tH_f}(y)\right) > 0.
    $$
    
    \noindent \textbf{Case 2: } Suppose that $H_f \pi^*b(x) = 0$.
    
    Because $H_f \pi^*b (x) = 0$, it follows that $x \in S\partial \Sigma$, and Equation \eqref{eqn: concave boundary} and Equation \eqref{eqn: 2nd der of p wrt hamil flow of f} show that $H_f^2 \pi^*b (x) > 0$. So, there exist an open subset $U \subset F$ containing $x$ and a constant $c>0$ such that $H_f^2 \pi^*b (y) > c$ for all $y \in U$. We then choose an open subset $\widetilde U \subset   U \cap G$ containing $x$ and $\epsilon_0 > 0$ such that $e^{tH_f}(y) \in U$ for all $y \in \widetilde U$ and $t \in (0,\epsilon_0]$. As a result, Equation \eqref{eqn: second order approx of function pi b} gives that 
    $$
    \pi^*b\left(e^{tH_f}(y)\right) > 0 \quad \text{for all} \quad y \in \widetilde U \cap \{H_f \pi^*b \geq 0\} \quad \text{and} \quad t \in (0, \epsilon_0].
    $$

    We have shown that for each $x \in F \cap G \cap \{H_f \pi^*b \geq 0\}$, there exist an open subset $U \subset F \cap G$ containing $x$ and $\epsilon_0 > 0$ such that $e^{tH_f}(y) >0$ for all $y \in U \cap \{H_f \pi^*b \geq 0\}$ and $t \in (0,\epsilon_0]$. So, the compactness of $F \cap G \cap \{H_f \pi^*b \geq 0\}$ finishes the proof.  
\end{proof}

Consider the hypersurfaces $F,G \subset T \widetilde \Sigma$ defined by Equation \eqref{eqn: hypersurfaces F and G}. Suppose that we look at the Hamilonian flow of $f$ on $F \cap \{ \pi^*b \geq 0\}=S\Sigma$, which is a manifold with boundary $F \cap G = F \cap \{ \pi^*b = 0\}=S_{\partial\Sigma}\widetilde\Sigma$. If we define the first boundary hitting time by 
\begin{equation}
    \label{def: 1st bound hitt time sympl}
    \begin{split}
        &\tau(x) = \inf \{t>0: e^{tH_f}(x) \in G \} \quad (\inf \emptyset = +\infty),\\
        &\text{where} \quad x \in F \cap \{ \pi^*b >0 \} \quad \text{or} \quad x \in F \cap G \cap \{ H_f \pi^*b \geq 0 \},
    \end{split}
\end{equation}
then Lemma \ref{lemma: uniform eps where hamil of f inside F and has not hit G} implies that:
\begin{corr}
    \label{cor: 1st bound hit time symp lower bound start F inter G}
    There exists $\epsilon_0 >0$ such that 
    \begin{equation}
        \label{eqn: lower bound on 1st bound hit time symp}
        \tau(x) > \epsilon_0 \quad \text{for all} \quad x \in F \cap G \cap \{H_f \pi^*b \geq 0 \},
    \end{equation}
    where $\tau$ is defined by Equation \eqref{def: 1st bound hitt time sympl}.
\end{corr}

Now, notice that for small times the Hamiltonian flow $e^{tH_f}$ starting at $x \in F \cap G \cap \{ H_f \pi^*b < 0\}$ leaves the manifold $F \cap \{ \pi^*b \geq 0 \}$ and moves to $F \cap \{ \pi^*b < 0 \}$. To define the billiard flow on $F \cap \{\pi^*b \geq 0\}$ which on $F \cap \{ \pi^*b > 0 \}$ is given by the Hamiltonian flow $e^{tH_f}$, we need to identify each $x \in F \cap G \cap \{ H_f \pi^*b < 0\}$ with a unique point $y \in F \cap G \cap \{ H_f \pi^*b \geq 0 \}$ because the Hamiltonian flow $e^{tH_f}$ starting at such a point $y$ stays inside $F\cap \{ \pi^*b \geq 0 \}$ for small times. We do this identification using the following reflection operation: If $x \in F \cap G \cap \left\{ H_f \pi^*b < 0 \right\}$ and $x = (p,v)$, where $p \in \partial\Sigma$ and $v \in T_p\Sigma$, then 
\begin{equation}
    \label{eqn: reflection operator}
    R(x) = R\left( (p,v) \right) = \left(p, v - 2 \left\langle v , \frac{\nabla b}{\| \nabla b \|_g^2} \right\rangle_g \nabla b \right).
\end{equation}

\begin{Remark}
    \label{rmk: refl oper same as moving along Ham flow of pi b}
    Equation \eqref{eqn: hamil vect field of p} says that the Hamiltonian vector field $H_{\pi^*b} = -\nabla b \cdot \partial_v$, so for any $x \in F \cap G \cap \{ H_f \pi^*b < 0 \}$, the definition of $R(x)$ provided by Equation \eqref{eqn: reflection operator} implies that $R(x)$ is the first intersection point of the Hamiltonian flow line $e^{tH_{\pi^*b}}(x)$  starting at $x$ ($t<0$) and the hypersurface $F$. In fact this intersection point exists because by Equation \eqref{eqn: 2nd der of f wrt hamilt flow of p}, $H_{\pi^*b}^2 f(y) = \| \nabla b \|_{g(\pi(y))}^2 > 0$ for all $y \in G$, where the last inequality follows since $db(\pi(y)) \neq 0$ as a consequence of Equation \eqref{eqn: bound def fn} which says that $b$ is a boundary definition function.   
\end{Remark}

We then define the map 
\begin{equation}
    \label{eqn: refl oper or identity}
     \begin{split}
        &S: F \cap \{ \pi^*b \geq 0 \} \rightarrow \left(F \cap \{ \pi^*b \geq 0 \}\right) \setminus \left( F \cap G \cap \{H_f \pi^*b < 0 \} \right) \quad \text{by}\\
        &S(x) = \begin{cases}
        R(x) &\text{if} \quad x \in F \cap G \cap \{ H_f \pi^*b < 0 \}\\
        x &\text{else} 
    \end{cases}
     \end{split}
\end{equation}
where $R$ is given by Equation \eqref{eqn: reflection operator}.
Finally, we can define the billiard flow on $S\Sigma$ 
\begin{equation}
    \label{def: sympl billiard flow for all positive times}
    \begin{split}
        &\varphi^t: F \cap \{\pi^*b \geq 0\} \rightarrow F \cap \{ \pi^*b \geq 0\} \quad \text{by}\\
        &\varphi^t(x) = \begin{cases}
            e^{tH_f}(S(x)) &\text{if} \quad 0 \leq t < \tau(S(x)) \\
            \varphi^{t-\tau(S(x))} 
            %\left( \lim_{\epsilon \rightarrow 0^+} \varphi^{\tau(S(x))}(x)\right)
            (e^{\tau(S(x)) H_f}(S(x)))
             &\text{if} \quad \tau(S(x)) \leq t < + \infty
        \end{cases}
    \end{split}
\end{equation}
where $S$ is defined by Equation \eqref{eqn: refl oper or identity} and $\tau$ by expression \eqref{def: 1st bound hitt time sympl}. See Figure \ref{fig: billiard traj} for an illustration of the billiard flow $\varphi^t$. 

\begin{Remark}
\label{rmk: sympl billiard flow}
    The flow $\varphi^t$ defined by Equation \eqref{def: sympl billiard flow for all positive times} is nothing but the geodesic flow  $e^{tH_f}$ for $x \in F \ \cap \{ \pi^*b > 0 \} = S \Sigma^{\circ}$, where $\Sigma^{\circ}$ is the interior of $\Sigma$. This flow $\varphi^t$ is discontinuous for any $\tilde t$ such that $\varphi^{\tilde t}(x) \in F \cap G \cap \{ H_{f} \pi^*b > 0 \}$, and the right limit $\lim_{\epsilon \rightarrow 0^+} \varphi^{\tilde t + \epsilon}(x) = \varphi^{\tilde t}(x)$ is uniquely determined from the left limit $\lim_{\epsilon \rightarrow 0^+} \varphi^{\tilde t - \epsilon}(x) \in F \cap G \cap \{ H_f \pi^*b < 0\}$ by following the Hamiltonian flow of $H_{\pi^*b}$ until we intersect again the hypersurface $F$ (refer to Remark \ref{rmk: refl oper same as moving along Ham flow of pi b} for more details). 
\end{Remark}
%\textcolor{red}{Maybe consider adding a figure of $\varphi^t$.}

\begin{figure}[ht]
    \centering
    \includegraphics[width=0.6\linewidth]{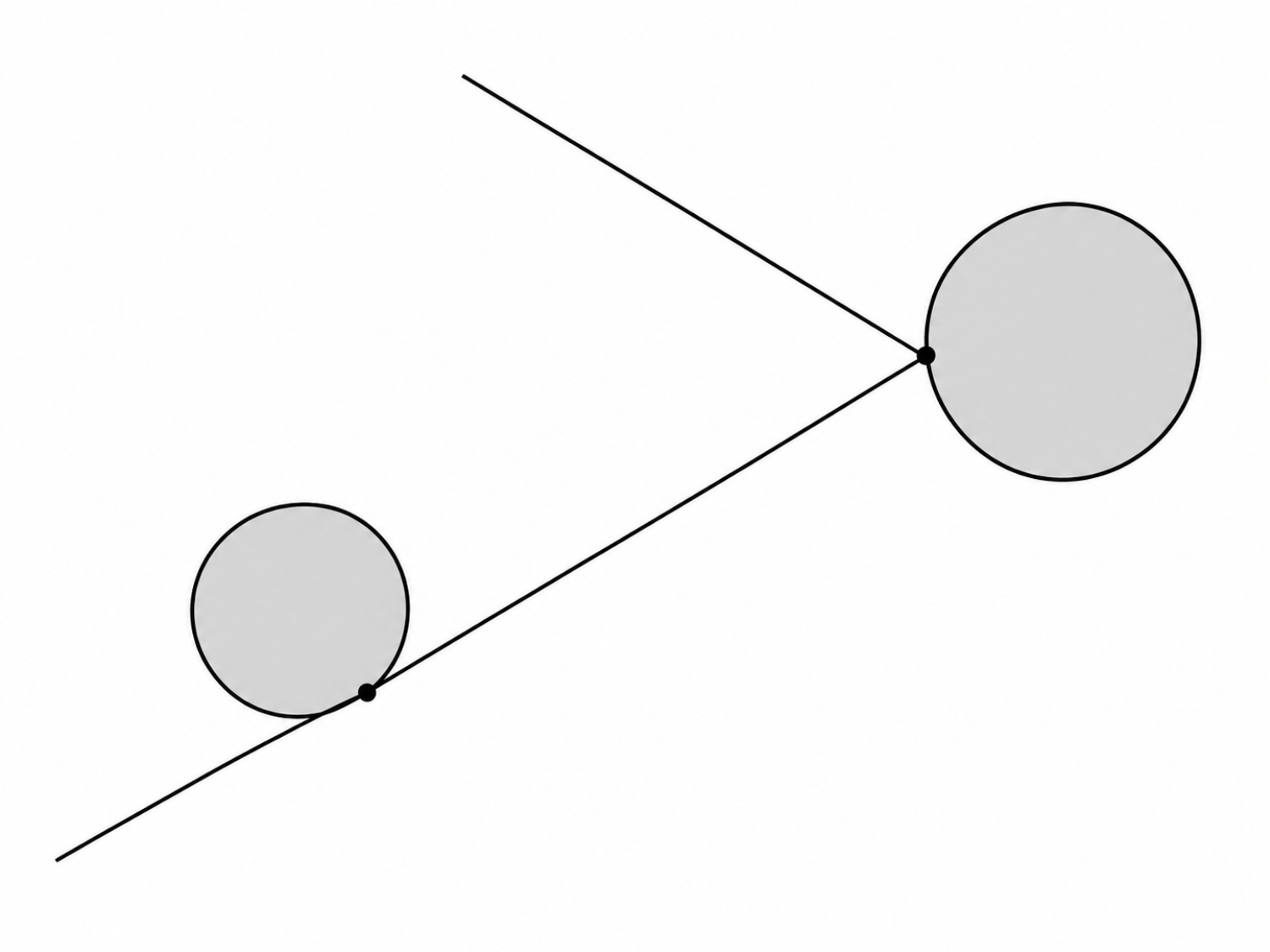}
    \caption{A billiard trajectory with a glancing and a transversal intersection.}
    \label{fig: billiard traj}
\end{figure}

\subsection{Glancing Hypersurfaces and Billiard Flow}
\label{subsec: glanc hypersurf and billiard flow}

Let $(N,\sigma)$ be a symplectic manifold of dimension $2n$, and $F$ and $G$ hypersurfaces in $N$. We recall the definition of a glancing hypersurface \cite[Definition~21.4.6]{H07}:
\begin{defi}
\label{def: glancing hypersurfaces}
    Two hypersurfarces $F , G \subset N$ are said to be glancing at $x \in F \cap G$ if for any defining functions $f$ and $g$ for $F$ and $G$, respectively, the following holds:
    \begin{enumerate}
        \item $df(x)$ and $dg(x)$ are linearly independent,
        \item $H_f g(x) = 0$, where $H_f$ is the Hamiltonian vector field of $f$, and
        \item $H_f^2 g(x) \neq 0$ and $H_g^2 f (x) \neq 0$, where $H_g$ is the Hamiltonian vector field of $g$.
    \end{enumerate}
\end{defi}

Melrose \cite{M76} shows that any two pairs of glancing hypersurfaces are locally equivalent in the following sense. 
\begin{theo}
    \label{thm: equiv of glancing hypersurfaces}
    Let $(N_1, \sigma_1)$ and $(N_2, \sigma_2)$ be symplectic manifolds of dimension $2n$, and $F_1, G_1 \subset N_1$ and $F_2, G_2 \subset N_2$ glancing hypersurfaces at $x_1 \in F_1 \cap G_1$ and $x_2 \in F_2 \cap G_2$, respectively. Then there exist open neighborhoods $U_1 \subset N_1$ and $U_2 \subset N_2$ of $x_1$ and $x_2$, respectively, and a symplectomorphism $\kappa: U_1 \rightarrow U_2$ such that $\kappa(x_1) = x_2$, $\kappa(F_1 \cap U_1) = F_2\cap U_2$, and $\kappa(G_1\cap U_1) = G_2 \cap U_2$.  
\end{theo}
As a consequence of Theorem \ref{thm: equiv of glancing hypersurfaces}, one has that given a symplectic manifold $(N, \sigma)$ of dimension $2n$ and glancing hypersurfaces $F$ and $G$ at $x \in F \cap G$, we obtain neighborhoods $U \subset N$ and $V \subset T^* \mathbb R^n$ of $x$ and $0$, respectively, and a symplectomorpshim  
\begin{equation}
    \label{eqn: friedlander model}
    \begin{split}
        &\kappa: U \rightarrow V \quad \text{such that}\\
        &\kappa(F \cap U) = \left\{ y_1 = 0 \right\} \cap V \quad \text{and} \quad \kappa(G\cap U) = \{ \eta_1^2 - y_1 - \eta_n = 0\} \cap V,
    \end{split}
\end{equation}
where $(y, \eta) \in T^* \mathbb R^n$. In fact, one can check that the hypersurfaces $\{y_1 = 0 \}$ and $\{ \eta_1^2 - y_1 -\eta_n = 0 \}$ are glancing at $0$.

On the manifold with boundary $\{y_1 = 0\} \cap \left\{ \eta_1^2 - y_1 -\eta_n \geq 0 \right\}$, we define the billiard flow $\Phi^t$ which is the Hamiltonian flow $e^{tH_{y_1}}$ in the interior $\{y_1 = 0\} \cap \{ \eta_1^2 - y_1 -\eta_n > 0 \}$ and obeys elastic reflections on the boundary $\{y_1 = 0\} \cap \left\{ \eta_1^2 - y_1 -\eta_n =0 \right\}$, that is we obtain the outgoing point from the incoming one by following the Hamiltonian flow of $\eta_1^2 - y_1 -\eta_n$ until we reach again the hypersurface $\{y_1 = 0 \}$. Explicitly, for $t\geq 0$ the billiard flow
\begin{equation}
    \label{def: symplectic billiard flow}
    \begin{split}
        &\Phi^t: \{y_1 = 0\} \cap \left\{ \eta_1^2 - y_1 -\eta_n \geq 0 \right\} \rightarrow \left\{y_1 = 0\right\} \cap \left\{ \eta_1^2 - y_1 -\eta_n \geq 0 \right\} \quad \text{is given by}\\
        &\Phi^t(y, \eta)\\
        &= \begin{cases}
        \left(0,y', y_n + 2 \sqrt{\eta_n}, \eta_1 - t - 2 \sqrt{\eta_n}, \eta', \eta_n\right) &\text{if} \quad \eta_n > 0 \quad \text{and} \quad t \geq  \eta_1 - \sqrt{\eta_n} \geq 0\\
        \left(0 , y', y_n, \eta_1 - t, \eta ', \eta_n\right) &\text{else} 
    \end{cases},
    \end{split}
\end{equation}
where $y = (y_1, y',y_n)$ and $\eta = (\eta_1, \eta', \eta_n)$. Figure \ref{fig:fried model} illustrates the flow $\Phi^t$. 
\begin{figure}[ht]
    \centering
    \includegraphics[width=0.75\linewidth]{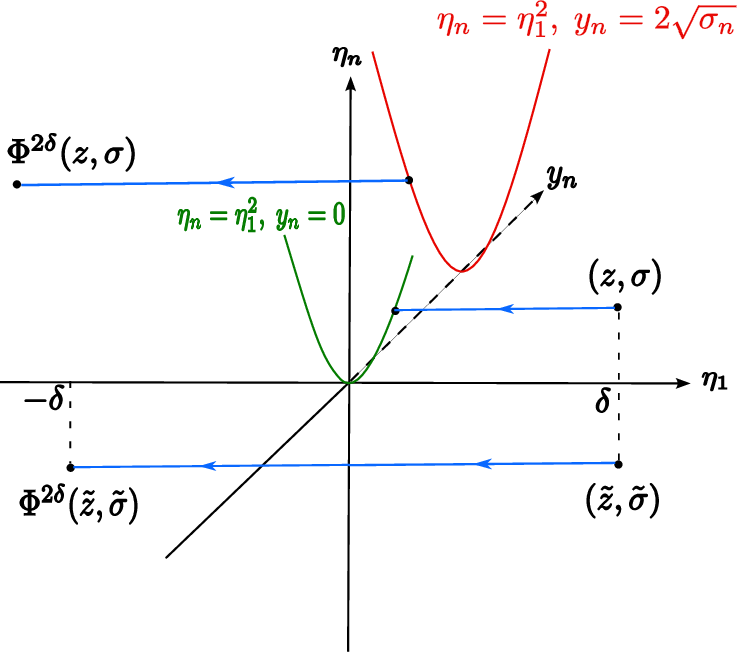}
    \caption{The flow $\Phi^t$ defined by Equation \eqref{def: symplectic billiard flow} on the manifold $\{y_1 = 0 \} \subset T^* \mathbb R^n$. Here, we only display the $\eta_1$, $\eta_n$, and $y_n$ axes.}
    \label{fig:fried model}
\end{figure}

\begin{Remark}
\label{rmk: symplectic billiard flow fried model}
    To digest further what this billiard flow $\Phi^t$ defined by Equation \eqref{def: symplectic billiard flow} is, notice that in the case when $\eta_n \leq 0$, the flow $\Phi^t$ is nothing but the Hamiltonial flow of $y_1$. When $\eta_n > 0$, the flow $\Phi^t$ is given by the Hamiltonian flow of $y_1$, but with a discontinuity. Precisely, if $t_0 = \eta_1 - \sqrt{\eta_n}$, then $\Phi^t(y,\eta)$ is the Hamiltonian flow of $y_1$ for times $t \in [0,t_0)$ and $t \geq t_0$, with a discontinuity at $t_0$. We obtain the right limit $ \Phi^{t_0}(y, \eta) \in \{y_1 = 0\} \cap \{\eta_1^2 - y_1 - \eta_n = 0 \}$ from the left limit $\lim_{\epsilon \rightarrow 0^+} \Phi^{t_0-\epsilon}(y,\eta) \in \{y_1 = 0\} \cap \{\eta_1^2 - y_1 - \eta_n = 0 \}$ by following the Hamiltonian flow of $\eta_1^2 - y_1 - \eta_n$ until we reach again the hypersurface $\{y_1=0\}$, which happens after time $\tau = -2 \sqrt{\eta_n}$. Moving along the Hamiltonian flow of $\eta_1^2 - y_1 - \eta_n$ results in shifting $y_n$ by $2 \sqrt{\eta_n}$ as seen in the definition of $\Phi^t$ provided in Equation \eqref{def: symplectic billiard flow}.
\end{Remark}

\begin{comment}
    One can see that 
\begin{equation}
    \label{eqn: symplectic flow for t= 2eps}
    \begin{split}
        \Phi^{2 \delta}&(x_2, \cdots, x_n,  \xi_1 = \delta, \cdots, \xi_n) \\
        &= \begin{cases}
        (x_2, \cdots, x_n, -\delta, \cdots ,\xi_n) &\text{if} \quad \xi_n \leq 0 \\
        e^{2 \sqrt{\xi_n} \left(-\partial_{\xi_1}\right)}\left(x_2, \cdots, x_n + 2\sqrt{\xi_n}, -\delta , \cdots, \xi_n\right) &\text{if} \quad 0 < \xi_n < \delta^2
    \end{cases},
    \end{split}
\end{equation}
where $e^{t \left(-\partial_{\xi_1}\right)}$ is the Hamiltonian flow of $x_1$. 
\end{comment}

Near glancing points of the hypersurfaces $F$ and $G$ from Section \S \ref{subsec: sympl billiard flow}, we will relate the billiard flow $\varphi^t$ given by Definition \ref{def: sympl billiard flow for all positive times} to the flow $\Phi^t$ . But before, we show the following: 
\begin{lemm}
    \label{lemma: cosphere bundle and boundary can be glanc}
    Consider the symplectic manifold $N = T \widetilde \Sigma$  and the hypersurfaces $F,G \subset N$ given by Equation \eqref{eqn: hypersurfaces F and G}. If $x_0 \in F \cap G \cap \{ H_f \pi^*b = 0\} = S\partial \Sigma$, then $F$ and $G$ are glancing at $x_0$, as defined by Definition \ref{def: glancing hypersurfaces}.
\end{lemm}
\begin{proof}
    Lemma \ref{lemma: diff of def fn for F and G are lin. indep} gives that $df(x_0)$ and $d(\pi^*b)(x_0)$ are linearly independent, and Equation \eqref{eqn: 2nd der of f wrt hamilt flow of p} implies that $H_{\pi^*b}^2 f(x_0) >0$ because $\nabla b (\pi(x_0)) \neq 0$ as a consequence of Equation \eqref{eqn: bound def fn} which says that $b$ is a boundary defining function for $\partial \Sigma$. Additionally, because $H_f \pi^*b(x_0) = 0$, equations \eqref{eqn: concave boundary} and \eqref{eqn: 2nd der of p wrt hamil flow of f} give that $H_f^2 \pi^*b(x_0) > 0$.
\end{proof}

We then use Theorem \ref{thm: equiv of glancing hypersurfaces} on the equivalence of glancing hypersurfaces to get that:
\begin{lemm}
    \label{lemma: relating billiard flow with flow from fried model}
    If $x_0 \in F \cap G \cap \{ H_f \pi^*b = 0 \}$, then there exists a symplectomorphism $\kappa_0: U \subset T \widetilde \Sigma \rightarrow V \subset T^* \mathbb R^{n} \quad (n=\dim \Sigma)$ specified by Equation \eqref{eqn: friedlander model} and $\delta > 0$ such that for all $(t,x) \in (0,\delta) \times \left(F \cap \{\pi^*b \geq 0\} \cap U\right)$
    \begin{equation}
        \label{eqn: relating billiard flow with flow from fried model}
        \begin{split}
            &\kappa_0^{-1} \circ e^{tH_{y_1}}\circ \kappa_0(x) = e^{\alpha(t,x)H_f}(x)  \quad \text{and} \quad \kappa_0^{-1} \circ \Phi^t \circ \kappa_0(x) = \varphi^{\beta(t,x)}(x) ,
        \end{split}
    \end{equation}
    where $\varphi^t$ is the billiard flow defined by Definition \ref{def: sympl billiard flow for all positive times}, 
    \begin{equation}
        \label{eqn: time shift when relating billiard flow and flow from fried model}
        \beta(t,x) = \begin{cases}
             \tau(x) + \alpha \left(t-t_0(x),\varphi^{\tau(x)}(x)\right)&\text{if} \quad \eta_n( \kappa_0(x)) > 0, \; t \geq  t_0(x) > 0 \\
             \alpha(t,\varphi^0(x)) &\text{if} \quad \eta_n( \kappa_0(x)) > 0, \; t \geq  t_0(x) = 0\\ 
            \alpha(t,x) &\text{else} 
        \end{cases},
    \end{equation}
    $t_0(x) = \eta_1( \kappa_0(x) ) - \sqrt{\eta_n (\kappa_0(x))}$,  $\alpha(t,x)$ is a smooth function
    such that $\alpha(0,x)=0$ and $\partial_t\alpha(t,x)>0$, and the first hitting time $\tau$ defined by Equation \eqref{def: 1st bound hitt time sympl} is given by $\tau(x)=\alpha(t_0(x),x)$ when $\eta_n(\kappa_0(x)) > 0$ and $t_0(x) > 0$. Moreover, if $R$ is the reflection operator defined in Equation \eqref{eqn: reflection operator}, then for any $(y,\eta) \in \kappa_0 (F \cap \{ \pi^*b = 0\} \cap \{ H_f \pi^*b < 0\} \cap U)$ we have
    \begin{equation}
        \label{eqn: reflec opera in sympl coord}
        \begin{split}
            &(y,\eta) = (0,y',y_n,\sqrt{\eta_n},\eta',\eta_n), \quad \text{with} \quad \eta_n > 0, \quad \text{and}\\
            &\kappa_0 \circ R \circ \kappa_0^{-1}(0,y',y_n,\sqrt{\eta_n},\eta',\eta_n) = (0,y',y_n + 2 \sqrt{\eta_n},-\sqrt{\eta_n},\eta',\eta_n).
        \end{split}
    \end{equation}
    %\textcolor{red}{(figure out if $\alpha>0$ or $\alpha < 0$ because the billiard flow $\varphi^t$ is defined for only positive times)}.
\end{lemm}
\begin{proof}
Let $q = \eta_1^2 - y_1 - \eta_n$. Since
$\kappa_0\left(\{f=0\} \cap U\right) = \left\{ y_1 = 0 \right\} \cap V$
and $\kappa_0\left(\{\pi^*b=0\}\cap U\right)=\{ q=0 \}\cap V$,
we have that
$$
y_1 \circ \kappa_0 = \lambda f,\quad
q \circ \kappa_0 = \gamma \pi^*b
$$
for some smooth nonvanishing functions $\lambda,\gamma$ on $U$.
Since $f(x_0)=\pi^*b(x_0)=H_f\pi^*b(x_0)=0$, we have
$$
\begin{aligned}
2
&=(H_{y_1}^2q)(0)
=H_{\lambda f}^2(\gamma\pi^*b)(x_0)=\lambda(x_0)^2\gamma(x_0)H_f^2(\pi^*b)(x_0),
\\
2
&=(H_q^2y_1)(0)
=H_{\gamma\pi^*b}^2(\lambda f)(x_0)=\lambda(x_0)\gamma(x_0)^2H_{\pi^*b}^2f(x_0).
\end{aligned}
$$
Since $H_f^2(\pi^*b)(x_0)$ and $H_{\pi^*b}^2f(x_0)$ are positive
(see Lemma~\ref{lemma: cosphere bundle and boundary can be glanc}),
we see that the functions $\gamma,\lambda$ are positive at $x_0$ and thus everywhere on~$U$.

Since $\kappa_0$ is a symplectomorphism, we have
$$
\begin{aligned}
\kappa^{-1}_0 \circ e^{tH_{y_1}} \circ \kappa_0(x)
&= e^{t H_{y_1 \circ \kappa_0}}(x)
= e^{\alpha(t,x)H_{f}}(x)\quad\text{on }F,
\\
\kappa^{-1}_0 \circ e^{tH_{q}} \circ \kappa_0 (x)
&= e^{tH_{q \circ \kappa_0}}(x)
= e^{\widetilde\alpha(t,x)H_{\pi^*b}}(x)
\quad\text{on }G
\end{aligned}
$$
where $\alpha,\widetilde\alpha$ are smooth and satisfy
$\alpha(0,x)=\widetilde\alpha(0,x)=0$,
 $\partial_t \alpha(t,x)|_{t=0} = \lambda(x)$, and
$\partial_t \widetilde\alpha(t,x)|_{t=0} = \gamma(x)$ (See \cite[Section \S 2.2]{HK12} for more details). 

%\textcolor{red}{(Not sure if we actually have $\lambda>0$ or $\gamma > 0$? Will think about it more.)}. 

    Let $x \in F \cap \{ \pi^*b > 0 \} = S \Sigma$. If we also have that $x \in U$, we see that $q(\kappa_0(x)) > 0$ 
    %\textcolor{red}{(If we assume that $\gamma>0$. Need to prove that this is true)}
    and that $y_1(\kappa_0(x))=0$; so from the definition of $\Phi^t$ provided in Equation \ref{def: symplectic billiard flow}, we have that  $\Phi^t(\kappa_0(x))$ makes sense. By making $U$ small if needed, we see that there exists $\delta > 0$ such that $\Phi^t(\kappa_0(x))$ is in the range of $\kappa_0$ for all $t \in (0,\delta)$ and $x \in F \cap \{ \pi^*b \geq 0 \} \cap U$.  Now pick $(t,x) \in (0, \delta) \times ( F \cap \{ \pi^*b \geq 0 \} \cap U)$, and we proceed with cases.

    \noindent
    \textbf{Case 1:} $\eta_n(k_0(x)) \leq 0$ or $\eta_n( \kappa_0(x)) > 0 \; \text{and} \;  t_0(x) > t \geq 0$.

    From the definition of $\Phi^t$, we have that $\Phi^t(\kappa_0(x)) = e^{tH_{y_1}}(\kappa_0(x)) \in \{ y_1 = 0\} \cap \{ q > 0\}$, which implies that $\kappa_0^{-1} \circ \Phi^t \circ \kappa_0 (x) = e^{\alpha(t,x)H_f}(x) \in F \cap \{ \pi^*b > 0 \} = S \Sigma^{\circ}$.
    %\textcolor{red}{(Here we use that $\gamma>0$, which we need to prove)}.
    Since from the definition of $\varphi^t$, we know that $\varphi^t$ is given by $e^{tH_f}$ on $F \cap \{ \pi^*b > 0 \}$, we get that $\kappa_0^{-1} \circ \Phi^t \circ \kappa_0 (x) = \varphi^{\alpha(t,x)}(x)$.

    \noindent
    \textbf{Case 2:} $\eta_n( \kappa_0(x)) > 0$ and $t \geq t_0(x) \geq 0$

    See that $\Phi^{t_0(x)}(\kappa_0(x)) \in \left\{ y_1 = 0 \right\} \cap \{ q = 0\}$, and furthermore Remark \ref{rmk: symplectic billiard flow fried model} says that the right limit $\Phi^{t_0(x)}(\kappa_0(x)) $ is obtained from the left limit $\lim_{\epsilon \rightarrow 0^+} \Phi^{t_0(x)-\epsilon}(\kappa_0(x)) \in \left\{ y_1 = 0 \right\} \cap \{ q = 0\}$ by following the Hamiltonian flow $e^{tH_q}$ until we reach again the hypersurface $\{ y_1 = 0\}$. Remark \ref{rmk: sympl billiard flow}  then implies that $\varphi^{\alpha(t_0(x),x)}(x) = \kappa_0^{-1} \circ \Phi^{t_0(x)} \circ \kappa_0 (x)$. We get Equation \eqref{eqn: relating billiard flow with flow from fried model} by noticing that $\Phi^{t-t_0(x)}(\kappa_0(x)) \in \left\{ y_1 = 0 \right\} \cap \{ q > 0 \}$ when $t > t_0(x)$ and that the first hitting time $\tau(x) = \alpha(t_0(x),x)$ when $t_0(x) > 0$.

    Finally, we prove Equation \eqref{eqn: reflec opera in sympl coord}. First notice that $\kappa_0(F \cap \{ \pi^*b = 0 \} \cap \{ H_f \pi^*b < 0 \} \cap U) = \left\{ y_1 =0 , \eta_1^2 = \eta_n, \eta_1 > 0 \right\} \cap V$ because for $x \in F \cap \{ \pi^*b =0 \} \cap \{ H_f \pi^*b < 0 \} \cap U$ we have that $- (2\eta_1 \circ \kappa_0)(x) = H_{y_1}q (\kappa_0(x)) = \gamma(x) \lambda(x) H_f \pi^*b (x) < 0$ since $\gamma , \lambda > 0$. As a result, if $(y, \eta) \in \kappa_0(F \cap \{ \pi^*b = 0 \} \cap \{ H_f \pi^*b < 0 \} \cap U)$, then 
    $$
     (y,\eta) = (0,y',y_n,\sqrt{\eta_n},\eta',\eta_n), \quad \text{with} \quad \eta_n > 0.
    $$
    
    Next, let $x \in F \cap \{ \pi^*b = 0\} \cap \{ H_f \pi^*b < 0 \} \cap U$. Recall that Remark \ref{rmk: refl oper same as moving along Ham flow of pi b} says that the point $R(x)$ can be obtained by starting at $x$ and then moving along the Hamiltonian flow of $\pi^*b$ until we reach again the hypersurface $F$. Therefore, one gets that the point $\kappa \circ R \circ \kappa_0^{-1}(\kappa_0(x))$ is obtained by starting at $\kappa_0(x)$ and following the Hamiltonian flow of $q$ until we reach again the hypersurface $\{ y_1 = 0 \}$. Concretely, if we set $\kappa_0(x) = (0,y',y_n,\sqrt{\eta_n}, \eta', \eta_n)$, where $\eta_n > 0$, one can check that 
    $$
    \kappa_0 \circ R \circ \kappa_0^{-1}(0,y',y_n,\sqrt{\eta_n}, \eta', \eta_n) = (0,y',y_n+2\sqrt{\eta_n},-\sqrt{\eta_n}, \eta', \eta_n).
    $$
\end{proof}

%%%%%%%%%%%%%%%%%%%%%%%%%%%%%%%%%%%%%%%%%%%%%%%%%%%%%%%%%%%%%%%%%%%%%%%%%%%%%%%%
\subsection{Closed billiard trajectory with multiple glancing points}

Here we describe the billiard flow near a trajectory with multiple glancing points,
using the concept of Poincar\'e maps.

%%%%%%%%%%%%%%%%%%%%%%%%%%%%%%%%%%%%%%%%%%%%%%%%%%%%%%%%%%%%%%%%%%%%%%%%%%%%%%%%
\subsubsection{Away from the boundary}

We construct first a Poincar\'e map between two Poincar\'e sections which have the property that the billiard flow $\varphi^t$ in between these two sections lies in the interior $S \Sigma^{\circ}$. One or both of the Poincar\'e sections could be a subset of the boundary $S_{\partial\Sigma}\Sigma$ not intersecting the glancing set $S\partial \Sigma$.
\begin{lemm}
    \label{lemma: billiard flow between two transv hypersurf}
    Let $F$ be the hypersurface defined by Equation \eqref{eqn: hypersurfaces F and G}. Suppose that $t_0 > 0$ and 
$$
\begin{aligned}
x_0 & \in \left(F \cap \{ \pi^*b \geq 0 \} \right)\setminus \{\pi^*b=0,\ H_f\pi^*b\leq 0\},
\\
x_1 & \in \left(F \cap \{ \pi^*b \geq 0 \} \right)\setminus \{\pi^*b=0,\ H_f\pi^*b\geq 0\}
\end{aligned}
%\setminus \{H_f \pi^*b = 0 \} = S \Sigma \setminus S\partial \Sigma
$$ 
    are such that $\varphi^{t_0 -0}(x_0) = x_1$
     and $\varphi^t(x_0) \in F \cap \{ \pi^*b > 0 \} = S \Sigma^{\circ}$ for all $t \in (0,t_0)$. Let $S_j \subset F$ ($j = 0, 1$) be hypersurfaces (without boundary) transversal to $H_f$ at $x_j$
     and contained in $\{\pi^*b\geq 0\}$. Then there exist open neighborhoods $U_j$ of $x_j$ in $S_j$ and a smooth function $T: U_0 \rightarrow \mathbb R_+$ such that 
    \begin{enumerate}
        \item $T(x_0) = t_0$,
        \item the Poincar\'e map $\varphi^{T(\bullet)-0}(\bullet): U_0 \rightarrow U_1$ is a diffeomorphism, and 
        \item $\varphi^t(x) \in F \cap \{ \pi^*b > 0\}$ for all $x \in U_0$ and $t \in (0,T(x))$.
    \end{enumerate}
\end{lemm}
\begin{proof}
    The assumption that $\varphi^t(x_0) \in F \cap \{ \pi^*b >0 \}$ for all $t \in (0,t_0)$ implies that $\varphi^t(x_0) = e^{tH_f}(x_0)$ for all $t \in (0,t_0)$ because $\varphi^t = e^{tH_f}$ on $F \cap \{ \pi^*b > 0 \}$, and as a result we have that $x_1 = \varphi^{t_0-0}(x_0) = e^{t_0H_f}(x_0)$.  Let $h$ be a defining function for the hypersurface $S_1$ near the point $x_1$. Since $S_1$ is transversal to the Hamiltonian flow $e^{tH_f}$ at $x_1 = e^{t_0H_f}(x_0)$, we get
    $$
    \partial_t|_{t=t_0}\left(h\left(e^{t H_f}(x_0\right)\right) \neq 0.
    $$
    Therefore, the implicit function theorem says that there exists an open neighborhood $U_0 \subset S_0$ of $x_0 \in S_0$ and a smooth function $T: U_0 \rightarrow \mathbb R$ such that $T(x_0) = t_0$ and $h\left( e^{T(x) H_f}(x)\right)=0$ for all $x \in U_0$, which implies that $e^{T(x) H_f}(x) \in S_1$. Since $T(x_0) = t_0 > 0$, we can take $U_0$ small enough so that $T(x) > 0$ for all $x \in U_0$. 

    Consider the function $e^{T(\bullet)H_f}(\bullet): U_0 \rightarrow S_1$. Notice that if $w \in T_{x_0} U_0$, then 
    $$
    d\left( e^{T(\bullet)H_f}(\bullet) \right)(w) = d\left(e^{tH_f}\right)|_{t=T(x_0)}(w) + dT(w) H_f(x_1), 
    $$
    which implies that $w+dT(w)H_f(x_0) = 0$ because $e^{tH_f}$ is a local diffeomorphism. But since $U_0$ is transversal to $H_f$ at $x_0$ and $w \in T_{x_0}U_0$, we must have that $w=0$, which gives that the differential $d\left( e^{T(\bullet)H_f}(\bullet) \right)$ is invertible at $x_0$. By making $U_0$ small if needed, the inverse function theorem says that there exists an open neighborhood $U_1 \subset S_1$ of $x_1$ such that 
    $$
    e^{T(\bullet)H_f}(\bullet): U_0 \rightarrow U_1 \quad \text{is a diffeomorphism.}
    $$

    Next, we show that $e^{tH_f}(x) \in F \cap \left\{ \pi^*b > 0 \right\}$ for all $x \in U_0$ and $t \in (0,T(x))$. Let $j=0 \; \text{or} \; 1$. If $x_j \in F \cap \{ \pi^*b > 0 \}$, we set $y_j = x_j$. 
    
    If $x_j \in \left(F \cap \{ \pi^*b = 0 \} \right)$, then we have that $(-1)^jH_f \pi^*b (x_j) > 0$. Thus, by making $U_j$ small enough, we can ensure that for some constant $c>0$, we have $(-1)^jH_f \pi^*b (x) > c$ for all $x \in U_j$, so there exists $\delta > 0$ such that $e^{ tH_f}(x) \in F \cap \{ \pi^*b >0 \}$ for all $\left((-1)^jt,x\right) \in (0,\delta] \times U_j$. Set $y_j = e^{\left( j t_0 + (-1)^j \frac{\delta}{2} \right) H_f}(x_j) \in F \cap \{\pi^*b > 0 \}$.

    By construction, we have that $y_0 , y_1 \in F \cap \{ \pi^*b > 0 \}$ with $y_j = e^{s_jH_f}(x_0)$, where $0 \leq s_0 \leq \frac{\delta}{2}$ and $t_0 - \frac{\delta}{2} \leq s_1 \leq t_0$. Additionally, we have that $e^{tH_f}(x_0) \in F \cap \{ \pi^*b > 0 \}$ for all $t \in [s_0, s_1]$. Thus, for each $s \in [s_0,s_1]$, there exist an open neighborhood $V_s \subset U_0$ of $x_0$ and $\epsilon_s > 0$ such that $e^{tH_f}(x) \in F \cap \{ \pi^*b > 0 \}$ for all  $(t,x) \in (s-\epsilon_s, s+ \epsilon_s) \times V_s$. Using the compactness of $[s_0,s_1]$, we get an open neighborhood $U_0 \subset S_0$ such that $e^{tH_f}(x) \in F \cap \{ \pi^*b > 0 \}$ for all $(t,x) \in [s_0,s_1] \times U_0$.  
     
    By making $U_0$ small enough if needed, we get that $e^{tH_f}(x) \in F \cap \{ \pi^*b > 0 \}$ for all $x \in U_0$ and $t \in (0,T(x))$. We finish the proof by noticing that the flow $\varphi^t$ on $F \cap \{ \pi^*b > 0 \}$ is given by $e^{tH_f}$.
    
\end{proof}

%%%%%%%%%%%%%%%%%%%%%%%%%%%%%%%%%%%%%%%%%%%%%%%%%%%%%%%%%%%%%%%%%%%%%%%%%%%%%%%%
\subsubsection{Transversal intersections}

Next, we construct a Poincar\'e map between two Poincar\'e sections where the billiard flow $\varphi^t$ in between these sections does not have any glancing points. See Figure \ref{fig:poinc map transversal reflections}
%\textcolor{red}{!! Add a figure to explain Lemma \ref{lemma: Poinc map with strict reflections between poinc secs} visually.}

\begin{figure}[ht]
    \centering
    \includegraphics[width=0.95\linewidth]{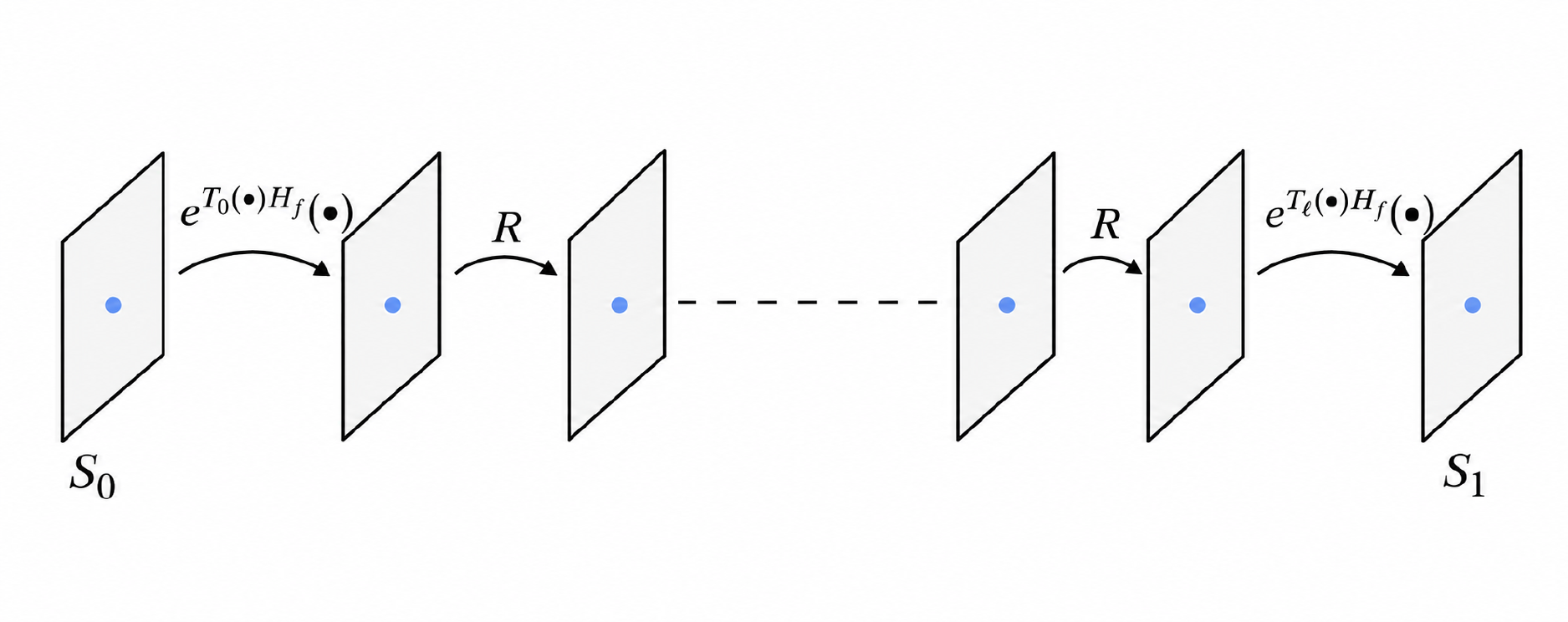}
    \caption{Structure of a Poincar\'e map $\varphi^{T(\bullet)}(\bullet): S_0 \rightarrow S_1$ with only transversal reflections between its Poincar\'e sections $S_0$ and $S_1$.}
    \label{fig:poinc map transversal reflections}
\end{figure}
\begin{lemm}
    \label{lemma: Poinc map with strict reflections between poinc secs}
       Let $F$ be the hypersurface defined by Equation \eqref{eqn: hypersurfaces F and G}. Suppose that $\tilde t > 0$ and $x_0 , x_1 \in F \cap \{ \pi^*b > 0 \} = S \Sigma^{\circ}$ are such that $\varphi^{\tilde t}(x_0) = x_1$. Assume that there exist $t_0 = 0 < t_1 < \cdots < t_l < t_{l+1} = \tilde t$, for some $l \geq 0$, such that $\varphi^{t_j}(x_0) \in F \cap \{ \pi^*b = 0 \} \setminus \{ H_f \pi^*b = 0 \} = S_{\partial \Sigma}\Sigma \setminus S \partial \Sigma$ for all $ 1 \leq j \leq l$ and $\varphi^t(x_0) \in F \cap \{ \pi^*b > 0 \}$ for all $t \in \left[0,\tilde t\right] \setminus \{ t_1, \dots, t_l \}$. If $S_j \subset F \cap \{ \pi^*b > 0 \}$, $j=0,1$, are hypersurfaces transversal to $H_f$ at $x_j$, then there exist open neighborhoods $U_j$ of $x_j$ in $S_j$ and a smooth function $T: U_0 \rightarrow \mathbb R_+$ such that 
    \begin{enumerate}
        \item $T(x_0) = \tilde t$, and 
        \item the Poincar\'e map $\varphi^{T(\bullet)}(\bullet): U_0 \rightarrow U_{1}$ is a diffeomorphism.
    \end{enumerate}
    Furthermore, the Poincar\'e map has the form
    $$
    \varphi^{T(\bullet)}(\bullet) = e^{T_l(\bullet)H_f}(\bullet) \circ R \circ \cdots  \circ e^{T_1(\bullet)H_f}(\bullet) \circ R \circ e^{T_0(\bullet)H_f}(\bullet), 
    $$
    where $R$ is the reflection operator defined by Equation \eqref{eqn: reflection operator}, $T_0: U_0 \rightarrow \mathbb R_+$ is smooth, $T_j: V_j \rightarrow \mathbb R_+$, $1 \leq j \leq l$, are smooth functions defined on open neighborhoods $V_j \subset F \cap \{\pi^*b = 0\} \cap \{ H_f \pi^*b > 0\}$ of $\varphi^{t_j}(x_0)$, and $T_j\left(\varphi^{t_j}(x_0)\right) = t_{j+1} - t_j$ for all $0 \leq j \leq l$.
\end{lemm}
\begin{proof}
    For $1 \leq j \leq l$, we define the hypersurfaces 
    $$
    G_j^{\pm} = F \cap \{ \pi^*b = 0\} \cap \{ \pm H_f \pi^*b > 0 \},
    $$
    which are transversal to $H_f$ at $\varphi^{t_j+0^{\pm}}(x_0)$. Set $G^+_0 = S_0$ and $G^-_{l+1} = S_{1}$.

    Let $0 \leq j \leq l$. The points $\varphi^{t_j+0}(x_0), \varphi^{t_{j+1}-0}(x_0)$ and hypersurfaces $G^+_j,G^-_{j+1}$ satisfy the assumptions in Lemma \ref{lemma: billiard flow between two transv hypersurf}, so there exist open neighborhoods $V_j^+ \subset G^+_j$ and $V^-_{j+1} \subset G^-_{j+1}$ of $\varphi^{t_j+0}(x_0)$ and $\varphi^{t_{j+1}-0}(x_0)$, respectively, and a smooth function $T_j: V^+_j \rightarrow \mathbb R_+$ such that 
    (1), (2), and (3) in Lemma \ref{lemma: billiard flow between two transv hypersurf} hold. Properties (2) and (3) imply that the Poincar\'e map $\varphi^{T_j(\bullet)-0}(\bullet): V_j^+ \rightarrow V_{j+1}^-$ is given by $\varphi^{T_j(\bullet)-0}(\bullet) = e^{T_j(\bullet)H_f}(\bullet)$.

    For each $1 \leq j \leq l$, notice that the reflection operator $R: G_j^- \rightarrow G_j^+$ is a diffeomorphism. As a result for any $0 \leq j < l$, we have that the Poincar\'e map 
    \begin{equation*}
        \begin{split}
            &\varphi^{T_j(\bullet)}(\bullet): V_j^+ \rightarrow V_{j+1}^+ \quad \text{given by} \\
            & \varphi^{T_j(\bullet)}(\bullet) = R \circ \varphi^{T_j(\bullet)-0}(\bullet) = R \circ e^{T_j(\bullet)H_f}(\bullet) \quad \text{is a diffeomorphism,}
        \end{split}
    \end{equation*}
    by modifying $V_j^+$ and $V_{j+1}^+$ if needed.
    
    Define 
    \begin{equation*}
        \begin{split}
            &T: V_0^+ \rightarrow \mathbb R_+ \quad \text{by}\\
            &T(x) = \sum_{j=0}^l \widehat T_j(x), \quad \text{where} \quad \widehat T_0(x) = T_0(x) \quad \text{and}\\
            &\widehat T_j(x) = \left(T_j \circ \varphi^{T_{j-1}(\bullet)}(\bullet)  \circ \cdots \circ \varphi^{T_{0}(\bullet)}(\bullet)\right)(x) \quad \text{if} \quad 1 \leq j \leq l.
        \end{split}
    \end{equation*}
    We see that the Poincar\'e map
    \begin{equation*}
        \begin{split}
            \varphi^{T(\bullet)}(\bullet) &: V_0^+ \rightarrow V^-_{l+1} \quad \text{given by} \\ 
            \varphi^{T(\bullet)}(\bullet) &= \varphi^{T_l(\bullet)}(\bullet) \circ \cdots \circ \varphi^{T_0(\bullet)}(\bullet) \\
            &= e^{T_l(\bullet)H_f}(\bullet) \circ R \circ e^{T_{l-1}(\bullet)H_f}(\bullet) \circ \cdots  \circ R \circ e^{T_0(\bullet)H_f}(\bullet) \quad \text{is a diffeomorphism,}
        \end{split}
    \end{equation*}
    by modifying the open sets $V_0^+$ and $V^-_{l+1}$ if needed. To see why the map $\varphi^{T(\bullet)}(\bullet): V_0^+ \rightarrow V^-_{l+1}$ is a diffeomorphism, notice that it is a composition of local diffeomorphisms. Setting $U_0$ and $U_1$ to be $V_0^+$ and $V^-_{l+1}$, respectively, and $V_j$ to be $V_j^+$ finishes the proof.
\end{proof}

%%%%%%%%%%%%%%%%%%%%%%%%%%%%%%%%%%%%%%%%%%%%%%%%%%%%%%%%%%%%%%%%%%%%%%%%%%%%%%%%
\subsubsection{Glancing intersections}

Finally, we construct a Poincar\'e map between two Poincar\'e sections where the billiard flow $\varphi^t$ between these sections has glancing points. Precisely, suppose that $x_0 \in F \cap \{ \pi^*b = 0 \} \cap \{ H_f \pi^*b = 0\} = S \partial \Sigma$. Lemma \ref{lemma: cosphere bundle and boundary can be glanc} says that the hypersurfaces $F$ and $\{\pi^*b = 0 \}$ are glancing at $x_0$, so we define the Poincar\'e sections 
\begin{equation}
    \label{eqn: Poinc secs near glanc point}
    \begin{split}
        &S^{\pm} = \kappa_0^{-1}\left(\left\{ (y,\eta): y_1 =0, \; \eta_1^2 - \eta_n > 0, \;  \eta_1 = \mp \delta \right\} \cap V \right) \subset F \cap \{ \pi^*b > 0 \},\\
        &\text{where } \delta > 0 \quad \text{is small}, \quad y_0^{\pm} = \kappa_0^{-1} \left( e^{\pm \delta H_{y_1}}(0) \right) \in S^{\pm},
    \end{split}
     \quad 
\end{equation}
and the symplectomorphism $\kappa_0$, the open set $V$, and $(y,\eta)$ are defined by Theorem \ref{thm: equiv of glancing hypersurfaces} and Equation \eqref{eqn: friedlander model}. Notice that the hypersurfaces 
$$
\left\{ (y,\eta): y_1 =0, \; \eta_1^2 - \eta_n > 0, \;  \eta_1 = \mp \delta \right\} \cap V 
$$ 
are transversal to the billiard flow $\Phi^t$ defined by Equation \eqref{def: symplectic billiard flow}, so Lemma \ref{lemma: relating billiard flow with flow from fried model} gives that the hypersurfaces $S^{\pm}$ are transversal to the billiard flow $\varphi^t$. We construct a Poincar\'e map from the Poincar\'e section $S^-$ to the Poincar\'e section $S^+$.

\begin{lemm}
    \label{lemma: Poincare map between Poinc secs near glanc point}
    Let $x_0 \in F \cap \{ \pi^*b = 0 \} \cap \{ H_f \pi^*b = 0\} = S \partial \Sigma$. The Poincar\'e map from $S^{-}$ to $S^+$ is given by
    $$
    R_0:=\varphi^{\beta\left(T(\bullet),\bullet\right)}(\bullet): S^{-} \rightarrow S^+,\quad
    \left(T\circ \kappa_0^{-1} \right) (y,\eta) = \begin{cases}
        2 \delta - 2 \sqrt{\eta_n} &\text{if} \quad \eta_n > 0\\
        2 \delta &\text{else} 
    \end{cases}
    $$
    where $\beta$ is  given by Lemma \ref{lemma: relating billiard flow with flow from fried model}. Furthermore, the Poincar\'e map has the form 
    $$ 
    R_0 = \kappa_0^{-1} \circ \widetilde R \circ \kappa_0,
    $$
    where
     $$
     \widetilde R(0,y',y_n,\delta,\eta',\eta_n) = \begin{cases}
        (0, y',y_n + 2 \sqrt{\eta_n},-\delta, \eta', \eta_n) &\text{if} \quad \eta_n > 0 \\
        (0,y',y_n, -\delta, \eta', \eta_n) &\text{else} 
    \end{cases},
    $$
    $y = (y_1, y' , y_n)$, and $\eta = (\eta_1, \eta',\eta_n)$.
\end{lemm}
\begin{proof}
     Consider the flow $\Phi^t$ defined by Equation \eqref{def: symplectic billiard flow}, and from the definition of $\Phi^t$, one can check that the Poincar\'e map from $\kappa_0(S^-)$ to $\kappa_0(S^+)$ is given by 
     $$
     \Phi^{\left(T\circ \kappa_0^{-1}\right)(\bullet)}(\bullet): \kappa_0(S^-) \rightarrow \kappa_0(S^+).
     $$
     Additionally, for any $(y,\eta) \in \kappa_0\left(S^-\right)$, one can check that
     $$
     \Phi^{\left(T\circ \kappa_0^{-1}\right)(y,\eta)}(y,\eta) = \widetilde R(y,\eta),
     $$
     so Lemma \ref{lemma: relating billiard flow with flow from fried model} finishes the proof.
\end{proof}

%%%%%%%%%%%%%%%%%%%%%%%%%%%%%%%%%%%%%%%%%%%%%%%%%%%%%%%%%%%%%%%%%%%%%%%%%%%%%%%%
\subsubsection{Putting it together}

 Now, we would like to study a closed trajectory of the billiard flow $\varphi^t$ that has exactly $m \geq 1$ distinct points $x_0, \dots, x_{m -1} \in F \cap \{\pi^*b = 0\}$ at which $F$ and $\{\pi^*b = 0\}$ are glancing. For each $0 \leq j < m$, let $S^{\pm}_j$ be the Poincar\'e sections defined by Equation \eqref{eqn: Poinc secs near glanc point}; and set $S^-_m = S^-_0$. We choose these Poincar\'e sections such that the billiard flow $\varphi^t$ does not have any glancing points between the sections $S^+_j$ and $S^-_{j+1}$ for all $0 \leq j < m$. %\textcolor{red}{Add a figure of a closed trajectory with $m$ glancing points. Show the Poincar\'e sections $S_j^{\pm}$ on the figure as well.}

\begin{lemm}
    \label{lemma: Poinc return map for closed traj with m glancing pts}
    Suppose that $t_{\per}>0$ and $x_0 \in F \cap \{ \pi^*b = 0 \} = S_{\partial \Sigma} \Sigma$ are such that $\varphi^{t_{\per}}(x_0) = x_0$. Let $m \in \mathbb N$, and assume that there exist $t_0 = 0 < t_1 < \cdots < t_{m-1} < t_{\per}$ such that $F$ and $\{ \pi^*b = 0 \}$ are glancing at $\varphi^{t_j}(x_0) \in F \cap \{ \pi^*b = 0\} $, $0 \leq j < m$, and $\varphi^t(x_0)$ is not a glancing point of $F$ and $\{ \pi^*b = 0\}$ for any $t \in [0,t_{\per}) \setminus \{t_0, t_1, \dots, t_{m-1} \}$. Then there exists an open neighborhood $U$ in $S^-_0 \subset F \cap \{ \pi^*b > 0 \} = S \Sigma^{\circ}$ of the point $y_0^- \in S_0^-$ defined by Equation \eqref{eqn: Poinc secs near glanc point} such that the Poincar\'e return map of the billiard flow $\varphi^t$ 
    $$
    P: U \rightarrow S_0^-, \quad P(y_0^-) = y_0^-
    $$
    is given by 
    $$
    P = \Phi_{m-1} \circ R_{m-1} \circ \cdots \circ \Phi_0 \circ R_0,
    $$
    where the Poincar\'e maps $\Phi_j = \varphi^{T_j(\bullet)}(\bullet): S_j^+ \rightarrow S_{j+1}^-$ are defined by Lemma \ref{lemma: Poinc map with strict reflections between poinc secs} and the Poincar\'e maps $R_j$ by Lemma \ref{lemma: Poincare map between Poinc secs near glanc point}. Furthermore, the map
    \begin{equation*}
        \begin{split}
            &\varphi^{t_{\per}}: U \rightarrow F \cap \{ \pi^*b > 0 \} \quad \text{is given by} \\
            &\varphi^{t_{\per}}(y) = e^{T(y)H_f}P(y),
        \end{split}
    \end{equation*}
    where the function $T: U \rightarrow \mathbb R$ is continuous and $T(y_0^-) = 0$.
\end{lemm}
\begin{proof}
    Notice that 
    \begin{equation*}
        S_0^- \xrightarrow{R_0} S_0^+ \xrightarrow{\varphi^{T_0(\bullet)}(\bullet)} S_1^- \xrightarrow{R_1} \cdots \xrightarrow{\varphi^{T_{m-2}(\bullet)}(\bullet)} S_{m-1}^- \xrightarrow{R_{m-1}} S_{m-1}^+ \xrightarrow{\varphi^{T_{m-1}(\bullet)}(\bullet)} S_0^-,
    \end{equation*}
    where $\varphi^{T_j(\bullet)}(\bullet)$ and $R_j$ are Poincar\'e maps. See Figure \ref{fig:poincare return map}.
    \begin{figure}[ht]
        \centering
        \includegraphics[width=0.95\linewidth]{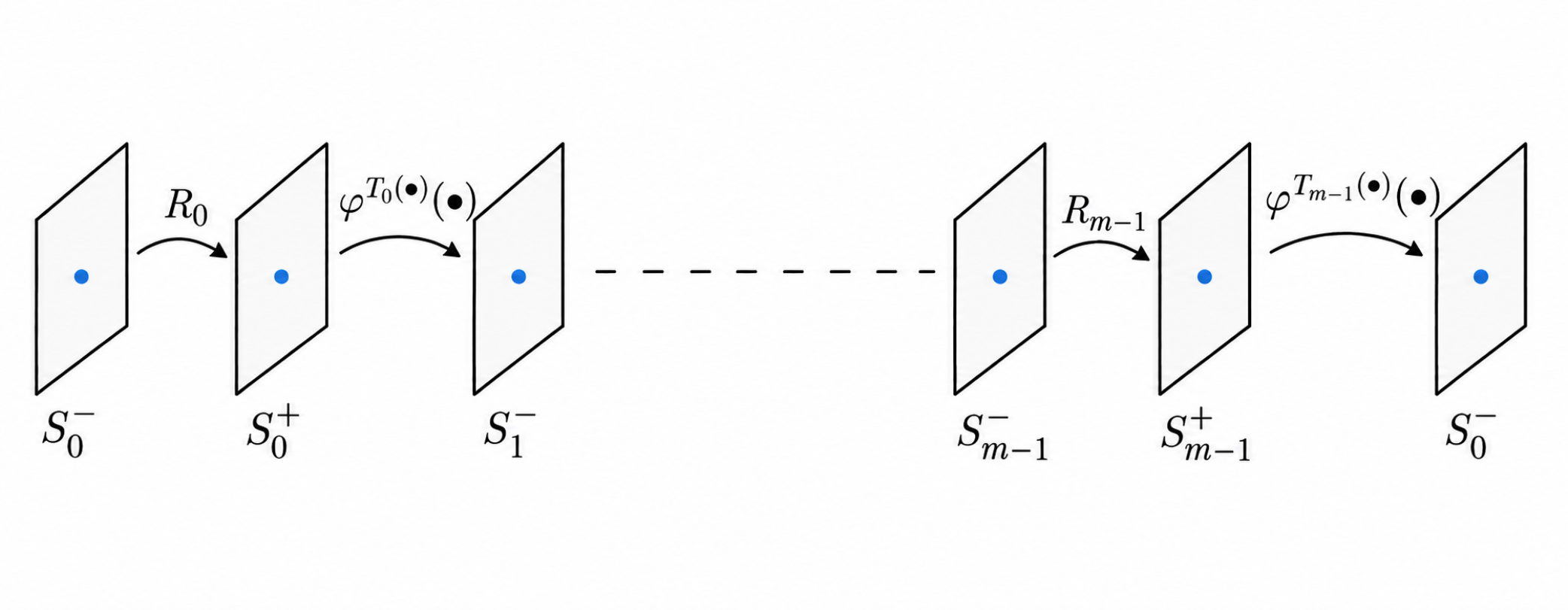}
        \caption{The Poincar\'e return map $P$ as a composition of Poincar\'e maps $R_j$ and $\varphi^{T_j(\bullet)}(\bullet)$ between consecutive Poincar\'e sections.}
        \label{fig:poincare return map}
    \end{figure}
    
    Additionally, one can check that $\varphi^{t_{\per}}(y_0^-) = y_0^-$ by noticing that $\kappa_0(x_0) = 0$ and that Lemma \ref{lemma: relating billiard flow with flow from fried model} gives us $y_0^- = \varphi^{\tilde t}(x_0)$, for some $0 < \tilde t < t_{\per}$. So, it follows that there exists an open neighborhood $U$ in $S^-_0$ of $y_0^-$ such that the Poincar\'e return map $P: U \rightarrow S^-_0$ is given by 
    $$
    P = \varphi^{T_{m-1}(\bullet)}(\bullet) \circ R_{m-1} \circ \cdots \circ \varphi^{T_0(\bullet)}(\bullet) \circ R_0, \quad \text{with} \quad P(y_0^-) = y_0^-.
    $$

    See that from the definition of $R_j$ provided in Lemma \ref{lemma: Poincare map between Poinc secs near glanc point}, we have that 
    $$
    R_j = \varphi^{\alpha_j(\bullet)}(\bullet),
    $$
    with $\alpha_j$ being continuous functions. We define the continuous function $T: S_0^- \rightarrow \mathbb R$ by 
    $$
    T(y) = t_{\per} - \sum_{j=0}^{m-1} \widehat T_j(y),
    $$
    where 
    $$
    \widehat T_j(y) = \alpha_j(Q_j(y)) + T_j(R_j (Q_j(y))), 
    $$
    $Q_0(y) = y$, and $Q_j(y) = \left(\varphi^{T_{j-1}(\bullet)}(\bullet) \circ R_{j-1} \circ \cdots \circ \varphi^{T_0(\bullet)}(\bullet) \circ R_0 \right)(y)$ for $1 \leq j < m$. One can then check that the Poincar\'e return map is given by 
    $$
    P(y) =\varphi^{t_{\per} - T(y)}(y) \quad \text{for any} \quad y \in U.
    $$
    Because $P(y_0^-) = \varphi^{t_{\per}}(y_0^-) = y_0^-$, we deduce that $T(y_0^-) = 0$. Additionally, since $P(y) \in S_0^- \subset F \cap \{ \pi^*b > 0 \}$, we get that $\varphi^{t_{\per}}(y) = e^{T(y) H_f}P(y) \in F \cap \{ \pi^*b > 0 \}$.
\end{proof}

%\textcolor{red}{!! Proof-read from here until end of section !!}

If we conjugate the Poincar\'e return map $P$ by a coordinate function, we can compute the determinant of the differential of this conjugated Poincar\'e return map minus identity, which makes sense outside the glancing set because the differential $dP$ is defined outside the glancing set. In Lemma \ref{lemma: det of Poinc ret map minus ident}, we also give an expansion of the above-mentioned determinant as we approach the glancing set.   

Let us introduce some notation to be used in Lemma \ref{lemma: det of Poinc ret map minus ident}. Suppose that the closed trajectory $\varphi^{t_{\per}}(x_0) = x_0$, with $t_{\per} > 0$ and $x_0 \in F \cap \{ \pi^*b = 0 \}$, has exactly $m$ distinct glancing points $\varphi^{t_j}(x_0)$, $0 \leq j < m$, as made precise in Lemma \ref{lemma: Poinc return map for closed traj with m glancing pts}. If $P: U \subset S^-_0 \rightarrow S^-_0$ is the Poincaré return map $P = \Phi_{m-1} \circ R_{m-1} \circ \cdots \circ \Phi_0 \circ R_0$ given by Lemma \ref{lemma: Poinc return map for closed traj with m glancing pts}, then we define 
\begin{equation}
    \label{eqn: Poinc return map in sympl coord}
    \begin{split}
        \widetilde P &= \kappa_0 \circ P \circ \kappa_0^{-1}\\
        &= \left(\kappa_0 \circ \Phi_{m-1} \circ \kappa_{m-1}^{-1} \right) \circ \left(\kappa_{m-1} \circ R_{m-1} \circ \kappa_{m-1}^{-1} \right) \circ \cdots  \left(\kappa_1 \circ \Phi_0 \circ \kappa_0^{-1} \right) \circ \left(\kappa_0 \circ R_0 \circ \kappa_0^{-1}\right) \\
        &= \widetilde \Phi_{m-1} \circ \widetilde R \circ \cdots \circ \widetilde \Phi_0 \circ \widetilde R,
    \end{split}
\end{equation}
where the coordinate $\kappa_j$ defined near the point $\varphi^{t_j}(x_0)$ is given by Lemma \ref{lemma: relating billiard flow with flow from fried model}, the map $\widetilde R$ by Lemma \ref{lemma: Poincare map between Poinc secs near glanc point}, and $\widetilde \Phi_j = \kappa_{j+1} \circ \Phi_j \circ \kappa_j^{-1}$, with $\kappa_m = \kappa_0$. 

We also define for any $y \in U \subset S_0^-$ the maps 
\begin{equation}
\label{eqn: Poincare map up to the jth Poincare sec before glancing point}
    \begin{split}
        &Q_0(y) = y, \quad Q_j(y) = \Phi_{j-1} \circ R_{j-1} \circ \cdots \circ \Phi_0 \circ R_0 (y) \quad \text{if} \quad 1 \leq j < m, \quad \text{and}\\
        &\widetilde Q_j = \kappa_j \circ Q_j \circ \kappa_0^{-1} \quad \text{for any} \quad 0 \leq j < m. 
    \end{split}
\end{equation}
For any $z \in \kappa_0(U)$, we use the maps $\widetilde Q_j$ to determine if the billiard trajectory starting at the point $\kappa_0^{-1}(z)$ hits the boundary near the glancing point $\varphi^{t_j}(x_0)$ before intersecting the Poincar\'e section $S_0^-$ again. Let $\eta_n$ be defined by Lemma \ref{lemma: relating billiard flow with flow from fried model}. If $\eta_n(\widetilde Q_j(z)) > 0$, then it means that the billiard trajectory starting at $\kappa_0^{-1}(z)$ intersects the boundary near the glancing point $\varphi^{t_j}(x_0)$ before reaching the Poincar\'e section $S_0^-$ again.  On the other hand, if $\eta_n(\widetilde Q_j(z)) < 0$, then the billiard trajectory starting at $\kappa_0^{-1}(z)$ does not intersect the boundary  near the  glancing point $\varphi^{t_j}(x_0)$ before reaching the Poincar\'e section $S_0^-$ again. This observation motivates the definition of the sets 
\begin{equation}
	\label{eqn: sets recording if we hit boundary at j glanc point}
	I_j(z) = \begin{cases}
			\{ 0, 1\} &\text{if} \quad \eta_n(\widetilde Q_j(z)) > 0 \\
			\{ 0\}  &\text{if} \quad \eta_n(\widetilde Q_j (z)) \leq 0
		\end{cases},
	\quad \text{and} \quad \mathcal I(z) = I_0(z) \times \dots \times I_{m-1}(z) \subset \{0,1 \}^m.
\end{equation}

Recall that $\kappa_0(S_0^-) \subset \{ y_1 = 0, \eta_1 = \delta \}$; therefore the conjugated Poincar\'e return map $\widetilde P$ is a function in the variables $(y',y_n, \eta', \eta_n)$, that is $\widetilde P(y',y_n,\eta',\eta_n)$, where $y = (y_1,y',y_n)$ and $\eta = (\eta_1, \eta', \eta_n)$. Similarly, we have that $\widetilde \Phi_j$ are functions in the variables $(y',y_n,\eta',\eta_n)$, where $(y,\eta)$ in this case is given by the coordinate $\kappa_j$ defined near the glancing point $\varphi^{t_j}(x_0)$. For each $z \in \kappa_0(U)$ and $\textbf i \in  \{0, 1 \}^m$, we define the matrices 
\begin{equation}
	\label{eqn: matrices in expansion of det of Poinc return map minus ident}
	B_{\textbf i}(z) = d \widetilde \Phi_{m-1} (\widetilde R (\widetilde Q_{m-1} (z))) N_{i_{m-1}} \cdots d \widetilde \Phi_0 ( \widetilde R( \widetilde Q_0(z)) ) N_{i_0},
\end{equation}
where 
$
N_0 = \text{Id}_{2n-2}
$
is the identity matrix and
$
N_1 = dh(0)
$
is a rank-1 matrix defined in terms of the function $h: \mathbb R^{2n-2} \rightarrow \mathbb R^{2n-2}$ given by $h(y',y_n, \eta',\eta_n) = (0,\eta_n,0,0)$. We also define for each $z \in \kappa_0(U)$ and $\textbf i \in \mathcal I(z) \subset \{0, 1 \}^m$ the functions 
\begin{equation}
	\label{eqn: fns giving assymptotics of determinant of diff of Poinc return map minus ident}
	\lambda_{\textbf i}(z) = \prod_{j: i_j = 1} \eta_n(\widetilde Q_j(z))^{-\frac{1}{2}}.
\end{equation}
Notice that the functions $\lambda_{\mathbf i}(z)$ are well defined because $\eta_n(\widetilde Q_j(z)) > 0$ for all $i_j = 1$, with $\textbf{i} = (i_0,\dots, i_{m-1}) \in \mathcal I(z)$.

Now we state the following:
\begin{lemm}
    \label{lemma: det of Poinc ret map minus ident}
    Suppose that $\Sigma$ has dimension $n=2$ and that the Poincar\'e return map $P: U \rightarrow S_0^-$ is given Lemma \ref{lemma: Poinc return map for closed traj with m glancing pts}. If $\widetilde P$ is the conjugated Poincar\'e return map defined by Equation \eqref{eqn: Poinc return map in sympl coord}, then 
    \begin{equation}
        \label{eqn: det of Poinc ret map minus ident}
        \begin{split}
            &\det\left(d\widetilde P-I\right)(z) = 2 - \tr d\widetilde P(z) \\
            &\text{for any} \quad z \in \kappa_0(U) \quad \text{with} \quad \eta_2 \left( \widetilde Q_j(z) \right) \neq 0 \quad \text{for all} \quad 0 \leq j < m,
        \end{split}
    \end{equation}
    where $\eta_2$ is given by Lemma \ref{lemma: relating billiard flow with flow from fried model} and $\widetilde Q_j(z)$ by Equation \eqref{eqn: Poincare map up to the jth Poincare sec before glancing point}.  Furthermore, we have
    \begin{equation}
    \label{eqn: trace of poinc return map}
        \begin{split}
            &d \widetilde P(z) = \sum_{\textbf i \in \mathcal I(z)}  \lambda_{\textbf i}(z)  B_{\textbf i}(z) \quad \text{and} \quad \tr d\widetilde P(z)  = \sum_{\textbf i \in \mathcal I(z)}  \lambda_{\textbf i}(z) \tr B_{\textbf i}(z), 
        \end{split}
    \end{equation}
    where $\mathcal I(z)$ is defined by Equation \eqref{eqn: sets recording if we hit boundary at j glanc point}, $\lambda_{\textbf i}(z)$ by Equation \eqref{eqn: fns giving assymptotics of determinant of diff of Poinc return map minus ident}, and $B_{\textbf i}(z)$ by Equation \eqref{eqn: matrices in expansion of det of Poinc return map minus ident}.
\end{lemm}
\begin{proof}
   From the assumption that $\Sigma$ is a 2 dimensional manifold,  we have that the differential $d \widetilde P(z)$ is a $2 \times 2$ matrix, so
   $$
   \det\left(d\widetilde P(z) - I\right) = \det\left(d \widetilde P(z)\right) - \tr\left(d \widetilde P(z)\right) + 1.
   $$
   Using the definition of $\widetilde P$ given in Equation \eqref{eqn: Poinc return map in sympl coord} and the definition of $\widetilde R$ provided in Lemma \ref{lemma: Poincare map between Poinc secs near glanc point}, we get that 
   \begin{equation}
   \label{eqn: interm diff of P}
       \begin{split}
        &d \widetilde P(z) = \left(d \widetilde \Phi_{m-1} \circ d \widetilde R \right) \circ \cdots \circ \left(d \widetilde \Phi_0 \circ d\widetilde R\right)(z), \quad \text{which implies that} \\
        &\det(d \widetilde P(z)) = \prod_{j=0}^{m-1} \det\left(d \widetilde \Phi_j\left(\widetilde R \circ \widetilde Q_j(z)\right)\right)
       \end{split}
   \end{equation}
   by noticing that for any $(y_2,\eta_2)$ with $\eta_2 \neq 0$ the differential 
   \begin{equation}
   \label{eqn: diff of tilde R}
       d \widetilde R\left( y_2, \eta_2\right) = \begin{cases}
       \begin{pmatrix}
       1 &  0 \\
       0 &  1
       \end{pmatrix} &\text{if} \quad \eta_2 < 0 \\
       \begin{pmatrix}
       1 &  \frac{1}{\sqrt{\eta_2}}\\
       0 &  1
       \end{pmatrix}  &\text{if} \quad \eta_2 > 0 
       \end{cases}
       \quad \text{has} \quad \det( \widetilde R(y_2, \eta_2) ) = 1,
   \end{equation}
   where we view $\widetilde R$ as a function in the $y_2$ and $\eta_2$ variables. 
   
   Equation \eqref{eqn: Poinc return map in sympl coord} defines $\widetilde Q_j = \kappa_{j+1} \circ Q_j \circ \kappa_j^{-1}$, with the symplectomorphisms $\kappa_j$ being defined by Lemma \eqref{eqn: relating billiard flow with flow from fried model}. Moreover, using Lemma \ref{lemma: Poinc return map for closed traj with m glancing pts} and Lemma \ref{lemma: Poinc map with strict reflections between poinc secs}, we have that the Poincar\'e maps $\Phi_j$ are given by 
   $$
   \Phi_j = e^{T_{j,l}(\bullet)H_f}(\bullet) \circ R \circ \cdots  \circ e^{T_{j,1}(\bullet)H_f}(\bullet) \circ R \circ e^{T_{j,0}(\bullet)H_f}(\bullet) \quad \text{for some} \quad l \geq 0,
   $$
   where $e^{T_{j,k}(\bullet)H_f}(\bullet)$ are Poincar\'e maps of the geodesic flow in $\widetilde M$ and $R$ the reflection operator defined by \eqref{eqn: reflection operator}. Each of these Poincar\'e maps $e^{T_{j,k}(\bullet)H_f}(\bullet)$ preserves the symplectic structure induced on the Poincar\'e sections associated with the Poincar\'e map $e^{T_{j,k}(\bullet)H_f}(\bullet)$ (See \cite[\S 1]{KT72} for more details). Therefore, we get that 
   $$
   \det\left(d \widetilde \Phi_j\left(\widetilde R_j \circ \widetilde Q_j(z)\right)\right)=1.
   $$

   For each $0 \leq j < m$, notice that Equation \eqref{eqn: diff of tilde R} gives that 
   $$
   d \widetilde R\left( \widetilde Q_j(z)\right) = \sum_{i_j \in I_j(z)} \eta_2(\widetilde Q_j(z))^{-\frac{i_j}{2}}N_{i_j},
   $$
   where $I_j(z)$ are defined by Equation \eqref{eqn: sets recording if we hit boundary at j glanc point} and $N_{i_j}$ by Equation \eqref{eqn: matrices in expansion of det of Poinc return map minus ident}.
   therefore the above equality together with Equation \eqref{eqn: interm diff of P} gives Equation \eqref{eqn: trace of poinc return map}. 
\end{proof}

\section{Non-degeneracy assumption}
\label{sec: non-deg assumption}
In this section, we show that the determinant in Lemma \ref{lemma: det of Poinc ret map minus ident} of the identity minus the differential of the Poincar\'e return map of the billiard flow $\varphi^t$  does not vanish under certain assumptions on $(\Sigma,g)$, which will be specified later.

\subsection{Horizontal and vertical decomposition and Jacobi fields}

We first recall that the tangent space of $T \widetilde \Sigma$ admits a horizontal and vertical decomposition as shown in \cite[\S 1.3]{GP99}. Specifically, if $x_0 = (p_0,v_0) \in T \widetilde \Sigma$, $w \in T_{x_0} \left( T \widetilde \Sigma \right)$, and $\sigma(s) = \left(p(s),v(s)\right)$ is a curve in $ T \widetilde\Sigma$ with $\sigma(0) = x_0$ and $\dot \sigma(0) = w$, then the horizontal and vertical decomposition of $w$ is given by 
\begin{equation}
    \label{eqn: hor and vert decomp}
    w = \left( \dot p(0), \nabla_{\dot p} v (0) \right),
\end{equation}
where $\nabla$ is the Levi--Civita connection on $(\widetilde \Sigma,g)$. Let $R$ be the reflection operator given by Equation \eqref{eqn: reflection operator}  and $\sigma(s)$ a curve in $T_{\partial \Sigma} \widetilde \Sigma$. Then the horizontal and vertical decomposition of $dR\left(\dot \sigma(0) \right)$ is given by

\begin{equation}
    \label{eqn: differ of R in horiz and vert decomp}
    dR\left(\dot \sigma(0) \right) = \left(\dot p (0), \nabla_{\dot p} v (0) - 2 \nabla_{\dot p} \big(\left\langle v,  n \right\rangle_g \big)(0) n(0) - 2 \left\langle v,  n \right\rangle_g(0) \nabla_{\dot p} n (0) \right),
\end{equation}
where $n(s) = \frac{\nabla b (p(s))}{\| \nabla b (p(s)) \|_{g(p(s))}}$.

Let $e^{tH_f}$ be the geodesic flow on $S\widetilde \Sigma$. For any $x_0 \in S \widetilde \Sigma$ and $w \in T_{x_0} S\widetilde \Sigma $, if $\sigma(s) = (p(s),v(s))$ is a curve in $S \widetilde \Sigma$ with $\sigma(0) = x_0$ and $\dot \sigma (0) = w$, then \cite[\S 1.5]{GP99} shows that the horizontal and vertical decomposition of $d(e^{tH_f})(w)$ is 
    \begin{equation}
        \label{eqn: jac field and differ of geodesic flow}
        d(e^{tH_f})(w) = \left(J(t),\dot J(t)\right), \quad J(0) = \dot p(0), \quad \dot J(0) = \nabla_{\dot p} v (0),
    \end{equation}
    where $J(t)$ is a Jacobi field along $\gamma(t) = \left(\pi \circ e^{tH_f}\right)(x_0)$ and $\dot J = \nabla_{\dot \gamma}J$. Recall that a Jacobi field $J(t)$ satisfies 
    \begin{equation}
        \label{eqn: jacob field equation}
        \left \langle \nabla_{\dot \gamma} \nabla_{\dot \gamma} J, J \right \rangle + \left( \| J \|^2_g - \left\langle J, \dot \gamma \right\rangle_g^2 \right) K(J, \dot \gamma) = 0,
    \end{equation}
    where $K$ is the sectional curvature of $(\widetilde \Sigma, g)$.

Next, we construct a Jacobi field on a segment of the billiard flow $\varphi^t$ with no glancing points.

 \begin{lemm}
     \label{lemma: billiard jacobi field with transversal reflections}
     Suppose that $\tilde t > 0$,  $x_0, x_1 \in F \cap \{ \pi^*b > 0 \} = S \Sigma^{\circ}$, $t_0 = 0 < t_1 < \dots < t_l < t_{l+1} = \tilde t$, and the Poincar\'e map $\varphi^{T(\bullet)}(\bullet): U_0 \rightarrow U_1$ of the form 
     $$
     \varphi^{T(\bullet)}(\bullet)= e^{T_l(\bullet)H_f}(\bullet) \circ R \circ \cdots \circ e^{T_1(\bullet)H_f}(\bullet) \circ R \circ e^{T_0(\bullet)H_f}(\bullet) 
     $$ 
     are given by Lemma \ref{lemma: Poinc map with strict reflections between poinc secs}. For any $w \in T_{x_0}U_0$, there exists a vector field $W(t) \in T_{\varphi^{t}(x_0)}F$ defined for all $t \in [0,\tilde t\,]$ and with the following properties: 
     \begin{enumerate}
         \item $W(0) = w$ and $W(\tilde t) - c H_f(x_1) = d \left( \varphi^{T(\bullet)}(\bullet) \right)(w) \in T_{x_1}U_1$ for some $c \in \mathbb R$,
         \item $W(t)$ is smooth for all $t \in [0,\tilde t \,] \setminus \{ t_1 , \dots , t_l \}$ and discontinuous for $t \in \{ t_1 , \dots, t_l \}$,
         \item $W(t_j+0) \in T_{\varphi^{t_j}(x_0)} \left( F \cap \{ \pi^*b = 0 \}  \right)$ for all $1 \leq j \leq l$, 
         \item For each $1 \leq j \leq l$ there exist $c_j \in \mathbb R$ such that 
         $$
         W(t_j+ 0) = dR \left( W(t_j - 0) - c_j H_f(\varphi^{t_j-0}(x_0)) \right), \quad \text{and}
         $$
         \item $W(t)$ has a horizontal and vertical decomposition $W(t) = (J(t), \dot J(t))$ for all $t \in [0,\tilde t \,] \setminus \{t_1, \dots , t_l\}$, where $J(t)$ is a Jacobi field along $\left(\pi \circ \varphi^t\right)(x_0)$.
     \end{enumerate}
 \end{lemm}
\begin{proof}
    Let $1 \leq j \leq l$. From the definition of $t_j$ in Lemma \ref{lemma: Poinc map with strict reflections between poinc secs}, we have that $\varphi^{t_j}(x_0) \in \left( F \cap \{ \pi^*b = 0 \} \right) \setminus \{ H_f \pi^*b = 0 \} = S_{\partial \Sigma} \Sigma \setminus S \partial \Sigma$. As a result, we have that $\varphi^{t_j}(x_0) \in F \cap \{ \pi^*b = 0 \} \cap \{ H_f \pi^*b > 0 \}$ and $\varphi^{t_j-0}(x_0) \in F \cap \{ \pi^*b = 0 \} \cap \{ H_f \pi^*b < 0 \}$ with $\varphi^{t_j}(x_0) = R \left(\varphi^{t_j-0}(x_0)\right)$.

    Define the hypersurfaces $V_j$ in $F$ by $V_0 = U_0$ and $V_j = F \cap \{ \pi^*b = 0 \} \cap \{ H_f \pi^*b > 0 \}$ for any $1 \leq j \leq l$. Now take $0 \leq j \leq l$ and $w_j \in T_{\varphi^{t_j}(x_0)} V_j$. We define the vector field
    \begin{equation}
    \label{eqn: recurs def of Jacob field}
        W_j(t,w_j) = d\left(e^{\left(t-t_j\right)H_f}\right)(w_j) \quad \text{for all}  \quad t_j \leq t < t_{j+1}. 
    \end{equation}
    Notice that Lemma \ref{lemma: Poinc map with strict reflections between poinc secs} gives $W_j(t,w_j) \in T_{\varphi^t(x_0)} F$ for all $t_j \leq t < t_{j+1}$; and Equation \eqref{eqn: jac field and differ of geodesic flow} implies that the horizontal and vertical decomposition of $W_j(t,w_j)$ is $W_j(t,w_j) = \left(J_j(t) , \dot J_j(t) \right)$, where $J_j(t)$ is a Jacobi field along $\pi \circ \varphi^t(x_0)$ for all $t \in (t_j, t_{j+1})$.  Additionally, we have
    \begin{equation}
        \label{eqn: diff of interm Phi_jk and diff of geod flow}
        \begin{split}
        d \left( e^{T_j(\bullet)} (\bullet) \right)(w_j) &= d(T_j)(w_j) H_f(\varphi^{t_{j+1}-0}(x_0)) +  d\left(e^{\left(t_{j+1}-t_{j}\right)H_f}\right)(w_{j}) \\
        &= W_j \left(t_{j+1}-0 , w_j\right) - c_j H_f(\varphi^{t_{j+1}-0}(x_0)), 
        \end{split}
    \end{equation}
    where in the first equality we used that $T_j\left(\varphi^{t_j}(x_0)\right) = t_{j+1} - t_j$ (refer to Lemma \ref{lemma: Poinc map with strict reflections between poinc secs}) and $c_j = -dT_j(w_j)$. If $0 \leq j < l$, notice that 
    $$
    W_j\left(t_{j+1}-0 , w_j\right) - c_j H_f \in T_{\varphi^{t_{j+1}-0}(x_0)} \left(F \cap \{ \pi^*b = 0 \} \cap \{ H_f \pi^*b < 0\}\right),
    $$
    so we define  
    \begin{equation}
        \label{eqn: right limit of Jacob field from left lim using R}
        \begin{split}
            W_j(t_{j+1},w_j) &= \begin{cases}
                dR \left( W_j(t_{j+1}-0,w_j) - c_j H_f \right) \in T_{\varphi^{t_{j+1}}(x_0)} V_{j+1} &\text{if} \quad 0 \leq j < l\\
                W_j(t_{j+1} -0, w_j)  &\text{if} \quad j = l.
            \end{cases}
        \end{split}
    \end{equation}
    
    Starting with an initial tangent vector $w \in T_{x_0} U_0$, Equations \eqref{eqn: recurs def of Jacob field} and \eqref{eqn: right limit of Jacob field from left lim using R} allow us to recursively define a vector field $W(t) \in T_{\varphi^t(x_0)}F$ for all $t \in \left[0,\tilde t \right]$ as follows:
    \begin{equation*}
        W(t) = \begin{cases}
            w & \text{if} \quad t=0 \\
            W_j(t, W(t_j)) & \text{if} \quad 0 \leq j \leq l \quad \text{and} \quad t_j \leq t \leq t_{j+1} 
        \end{cases}
    \end{equation*}
    One can check that this vector field $W(t)$ satisfies properties $(1)$ through $(5)$ outlined in Lemma \ref{lemma: billiard jacobi field with transversal reflections}. 
\end{proof}

Finally, we construct a Jacobi field on a segment of the billiard flow $\varphi^t$ that has exactly one glancing point. 
\begin{lemm}
    \label{lemma: Jac field on bill segment with one glanc pt}
    Assume that $\Sigma$ has dimension $n=2$, and let $x_0 \in F \cap \{ \pi^*b = 0 \} \cap \{ H_f \pi^*b = 0 \} = S \partial \Sigma$. Suppose that the map $N_0: T_{\tilde z_0} \left( \kappa_0\left(S^-\right)\right) \rightarrow T_{\widetilde R_0 (\tilde z_0)} \left( \kappa_0\left(S^+\right)\right)$ is defined by Equation \eqref{eqn: matrices in expansion of det of Poinc return map minus ident}, where $\tilde z_0 = e^{-\delta H_{y_1}}(0)$, $\kappa_0$ is defined by Lemma \ref{lemma: relating billiard flow with flow from fried model}, and $S^{\pm}$, $\delta>0$, and $y_1$ by Equation \eqref{eqn: Poinc secs near glanc point}. If $\beta$ is the function given by Lemma \ref{lemma: relating billiard flow with flow from fried model}, $z_0 = \kappa_0^{-1}(\tilde z_0) \in S^-$, and $\tilde t = \beta\left(2\delta,z_0\right)$, then for any $w \in T_{z_0}S^-$ there exists a smooth vector field $W(t) \in T_{\varphi^t\left(z_0\right)}F$ defined for all $t \in [0, \tilde t \,]$ such that:
    \begin{enumerate}
        \item $W(0) = w$ and $W(\tilde t\,) - c H_f(\varphi^{\tilde t}(z_0)) = \left(d\kappa^{-1}_0 \circ N_0 \circ d\kappa_0 \right)(w) \in T_{\varphi^{\tilde t}(z_0)}S^+$ for some $c \in \mathbb R$,
        \item $W(t) = d(e^{tH_f})(w)$ for all $t \in [0,\tilde t \,]$,
        \item The horizontal and vertical decomposition of $W(t)$ is $W(t) = (J(t),\dot J(t))$, where $J(t)$ is a Jacobi field along $\pi \circ \varphi^t(z_0)$.
    \end{enumerate}
\end{lemm}
\begin{proof}
    Let the billiard flow $\Phi^t$ be defined by Equation \eqref{def: symplectic billiard flow}. From the definition of $\Phi^t$, one can see that for all $t \in [0,2\delta]$ we have 
    $$
    \Phi^t(\tilde z_0) = e^{tH_{y_1}}(\tilde z_0) \quad \text{with} \quad \Phi^{2\delta}(\tilde z_0) \in \kappa_0(S^+).
    $$
    For each $x \in S^-$, Lemma \ref{lemma: relating billiard flow with flow from fried model} says that for small time $t$ we have
    $$
    \kappa_0^{-1} \circ e^{tH_{y_1}} \circ \kappa_0(x) = e^{\alpha(t,x)H_f}(x) \quad \text{and} \quad \kappa_0^{-1} \circ \Phi^{t} \circ \kappa_0(x) = \varphi^{\beta(t,x)}(x), 
    $$
    where $\alpha$ is smooth. As a result, we see that for all $t \in [0,2\delta]$ we get
    $$
    \varphi^{\beta(t,z_0)}(z_0) = e^{\alpha(t,z_0)H_f}(z_0) \quad \text{with} \quad \varphi^{\beta(2\delta,z_0)}(z_0) \in S^+ \quad \text{and} \quad \beta(2\delta,z_0) = \alpha(2\delta,z_0).
    $$

    Now, define the vector field $W(t) \in T_{\varphi^t(z_0)}F$ for all $t \in [0,\tilde t = \beta(2 \delta,z_0)]$ by 
    \begin{equation}
        \label{eqn: vector field near glancing point}
        W(t) = d(e^{tH_f})(w).
    \end{equation}
    Notice that 
    $$
    \begin{aligned}
        ( d\kappa_0^{-1} \circ d\left(e^{2 \delta H_{y_1}}|_{\kappa_0(S^-)}\right) \circ d \kappa_0 )(w) &= d \left( e^{\alpha(2 \delta, \bullet)H_f}(\bullet) \right)(w) \\
        &= \left( W(\tilde t \,) + d(\alpha(2 \delta,\bullet))(w) H_f(\varphi^{\tilde t}(z_0)) \right) \in T_{\varphi^{\tilde t}(z_0)} S^+. 
    \end{aligned}
    $$
    Recall that $\kappa_0(S^-) \subset \{y_1 = 0, \eta_1 = \delta \}$, so for $(0,y_2,\delta,\eta_2) \in \kappa_0(S^-)$, one has that 
    $$
    e^{2\delta H_{y_1}}(0,y_2,\delta,\eta_2) = (0,y_2,-\delta,\eta_2) \in \kappa_0(S^+).
    $$
    As a result, if we take the differential of $e^{2\delta H_{y_1}}|_{\kappa_0(S^-)}$ in the $y_2, \eta_2$ variables, we have that 
    $$
    d \left( e^{2\delta H_{y_1}}|_{\kappa_0(S^-)} \right) = N_0.
    $$

    Finally, Equation \eqref{eqn: vector field near glancing point} defining $W(t)$ and Equation \eqref{eqn: jac field and differ of geodesic flow} show that the horizontal and vertical decomposition of $W(t)$ is given by 
    $$
    W(t) = (J(t),\dot J(t)),
    $$
    where $J(t)$ is a Jacobi field along $\pi \circ \varphi^t(z_0)$.
\end{proof}

\subsection{Curvature assumption.}
 
  Let $t_{\per} > 0$ and $x_0 \in F \cap \{ \pi^*b = 0 \} = S_{\partial \Sigma}\Sigma$ be such that $\varphi^{t_{\per}}(x_0) = x_0$. Suppose that $m \in \mathbb N$ and that there exist $t_0 = 0 < t_1 < \cdots < t_{m-1} < t_m = t_{\per}$ such that $\varphi^{t_j}(x_0)$ is a glancing point of $F$ and $\{ \pi^*b = 0 \}$ for any $0 \leq j < m$ and $\varphi^t(x_0)$ is not a glancing point for any $t \in (0,t_{\per}) \setminus \{t_1, \dots, t_{m-1} \}$. We then speficify the following assumption: 
\begin{equation}
    \label{eqn: curvature assumption}
    \begin{split}
        &\text{$(\widetilde \Sigma ,g)$ of dimension $n=2$ has:}\\
        &\quad \bullet \; \text{negative sectional curvature,} \quad \text{or} \\
        &\quad \bullet \; \text{non-positive sectional curvature and for every  $t_j$, $0 \leq j < m$, there exists} \\ 
        &\quad  \quad \; t \in (t_j,t_{j+1}) \quad \text{such that} \quad \varphi^t(y_0) \in F \cap \{\pi^*b = 0\} \cap \{ H_f \pi^*b >0 \}.
    \end{split}
\end{equation}

Before stating the main lemma of this subsection, we show first the following.
\begin{lemm}
    \label{lemma: vec tang to boundary}
    Suppose that $x_0 \in F \cap \{ \pi^*b = 0\} \cap \{ H_f \pi^*b  = 0\} = S \partial \Sigma$ and that the symplectomorphism $\kappa_0: U \subset T \widetilde \Sigma \rightarrow V \subset T^* \mathbb R^n$ and $y_n$ are given by Lemma \ref{lemma: relating billiard flow with flow from fried model}. Let $\partial_{y_n} \in T_0(T^* \mathbb R^n)$. Then $d\kappa_0^{-1}(\partial_{y_n}) \in T_{x_0} (S \partial \Sigma)$. 
\end{lemm}
\begin{proof}
    From the definition of $\kappa_0$, we have that 
    $$
    \begin{aligned}
        & \kappa_0 (F \cap U) = \{ y_1 = 0\} \cap V \quad \text{and} \quad \kappa_0(\{ \pi^*b = 0 \} \cap U) = \{ q  = 0 \} \cap V,
    \end{aligned}
    $$
    where $q=\eta_1^2 - y_1 - \eta_n$. So we know that on $U$ 
    $$
    y_1 \circ \kappa_0 = \lambda f \quad \text{and} \quad q \circ \kappa_0 = \gamma \pi^*b,
    $$
    where $\lambda$ and $\gamma$ are smooth function that do not vanish on $U$. Pick $x \in F \cap \{ \pi^*b = 0 \} \cap \{ H_f \pi^*b = 0 \} \cap U$. Then, we see that
    $$
    -2 \eta_1(\kappa_0(x)) = H_{y_1}q  (\kappa_0(x)) = H_{y_1 \circ \kappa_0} (q \circ \kappa_0) (x) = \lambda(x) \gamma(x) H_f \pi^*b(x) = 0.
    $$
    We deduce that 
    $$
    \kappa_0(F \cap \{ \pi^*b = 0\} \cap \{ H_f \pi^*b = 0 \} \cap U) = \kappa_0(S\partial \Sigma \cap U) = \{y_1 = \eta_1 = \eta_n = 0 \} \cap V.
    $$
    Let $y = (y_1, y',y_n)$ and $\eta = (\eta_1, \eta', \eta_n)$. One can see that  if $(0,y',y_n,0,\eta',0) \in V$ then $(0,y',y_n,0,\eta',0) \in \kappa_0(S \partial \Sigma \cap U)$. As a result, we see that $d\kappa_0^{-1}(\partial_{y_n}) \in T_{x_0} (S \partial \Sigma)$.
\end{proof}

Under the assumptions in Equation \eqref{eqn: curvature assumption}, we show for example that a Poincar\'e map of the billiard flow $\varphi^t$ does not map a tangent vector of the glancing submanifold $S \partial \Sigma$ to a tangent vector of $S \partial \Sigma$. Precisely,
\begin{lemm}
\label{lemma: no tang vec mapped to tang vec}
    Let $N_0$ be given by Equation \eqref{eqn: matrices in expansion of det of Poinc return map minus ident} and the Poincar\'e maps $\widetilde \Phi_j$ by Equation \eqref{eqn: Poinc return map in sympl coord}. Under the assumptions in Equation \eqref{eqn: curvature assumption}, we have that 
\begin{equation}
    \label{eqn: tangent vector not mapped to tangent vector}
    \begin{split}
        &\det \left( d\widetilde \Phi_{m-1}(0)  N_0  \cdots  d\widetilde \Phi_0 (0) N_{0} - I \right) \neq 0 \quad \text{and} \\
        &\left \langle \left(d \widetilde \Phi_r(0)  N_0  d \widetilde \Phi_{r-1}(0) \cdots  N_0  d \widetilde \Phi_p(0) \right) \left(\partial_{y_2}\right) , \partial_{\eta_2}  \right \rangle \neq 0 \quad \text{for all} \quad p \leq r < m,
    \end{split}
\end{equation}
where $y_2$ and $\eta_2$ are defined by Lemma \ref{lemma: relating billiard flow with flow from fried model} and $\langle \cdot, \cdot \rangle$ is the standard metric on $\mathbb R^2$. 
\end{lemm}
\begin{proof}
    We break down the proof in several steps.
    \bigskip
    \newline
    \textbf{Step 1:} Construction of a Jacobi field.
    
    Given any $0 \leq j < m$, recall that Equation \eqref{eqn: Poinc return map in sympl coord} says that $\widetilde \Phi_j = \kappa_{j+1} \circ \Phi_j \circ \kappa_j^{-1}$, where the Poincar\'e maps $\Phi_j = \varphi^{T_j(\bullet)}(\bullet): S_j^+ \rightarrow S_{j+1}^-$ of the form  
    \begin{equation*}
        \Phi_j(\bullet) = e^{T_{j,l_j}(\bullet)H_f} (\bullet) \circ R \circ \cdots \circ R \circ e^{T_{j,0}(\bullet) H_f} (\bullet)
    \end{equation*}
    are defined by Lemma \ref{lemma: Poinc return map for closed traj with m glancing pts} and the symplectomorphisms $\kappa_j$ by Lemma \ref{lemma: relating billiard flow with flow from fried model} (Here we set $\kappa_m = \kappa_0$ and $S^{\pm}_m = S^{\pm}_0$). 
    
    For each $0 \leq j < m$, notice also that $x_j^+ = \varphi^{\beta_j(\delta_j,\varphi^{t_j}(x_0))} (\varphi^{t_j}(x_0)) \in S^+_j$, where $\delta_j$ is defined by Equation \eqref{eqn: Poinc secs near glanc point} and $\beta_j$ by Lemma \ref{lemma: relating billiard flow with flow from fried model}. If we define $x_{j+1}^- = \Phi_j(x^+_j) \in S_{j+1}^-$, then we have that $\varphi^{t_{j+1}}(x_0) = \varphi^{\beta_{j+1}(\delta_{j+1},x_{j+1}^-)}(x_{j+1}^-)$. As a result, we can see that 
    $$
    T_j(x_j^+) = t_{j+1} - t_j - (\tilde \delta_j^+ + \tilde \delta_{j+1}^-),
    $$
    where $\tilde \delta_j^+ = \beta_j(\delta_j,\varphi^{t_j}(x_0))$ and $\tilde \delta_{j+1}^- = \beta_{j+1}(\delta_{j+1},x_{j+1}^-)$. Set $x_0^{\pm} = x_m^{\pm}$.
  
    \begin{figure}[ht]
        \centering
        \includegraphics[width=0.95\linewidth]{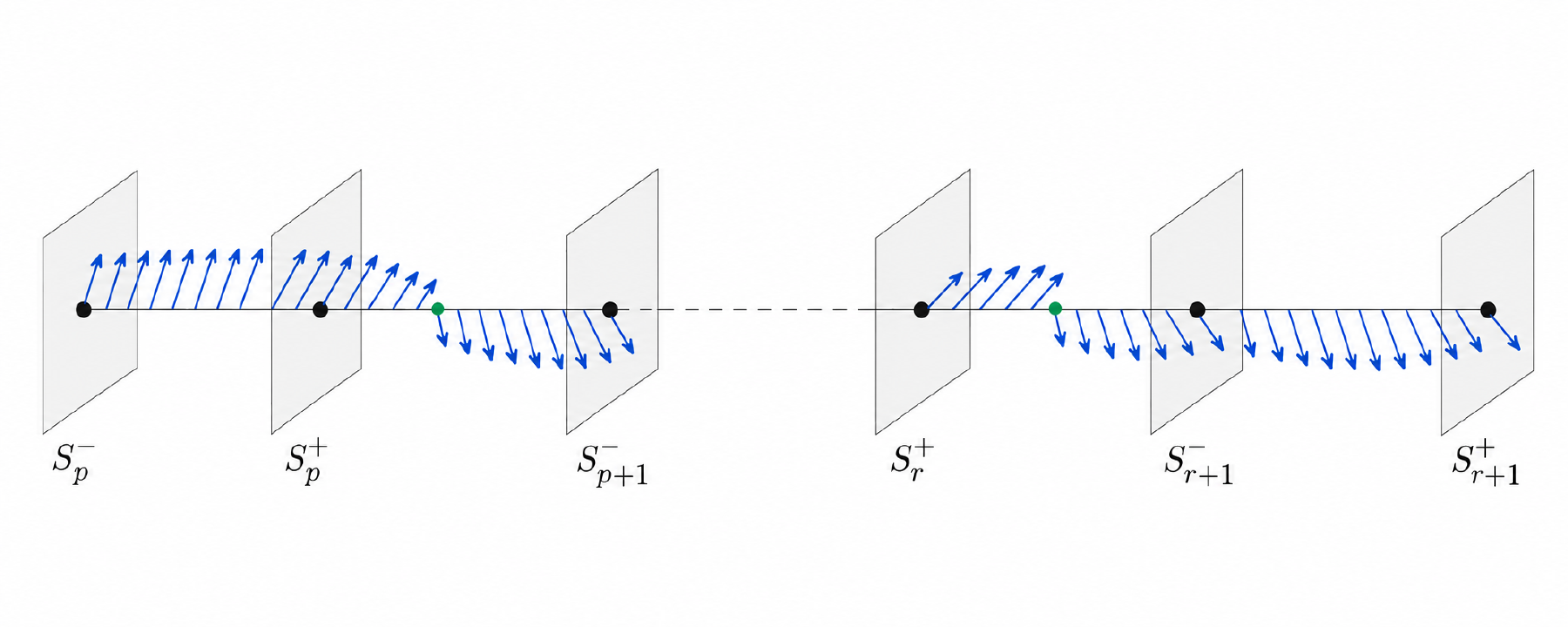}
        \caption{Jacobi field $W(t)$ defined by Equation \eqref{eqn: Jacobi field on segment of closed bill traj with transversal and glanc points} along a segment of a billiard trajectory between two Poincar\'e sections $S_p^-$ and $S_{r+1}^+$. $W(t)$ is not continuous at the black and green dots. The green dots represent points where the billiard flow undergoes transversal reflections.}
        \label{fig:placeholder}
    \end{figure}
    Suppose that $p \leq r < m$, and let $w \in T_{x_p^-}S^-_p$. We use Lemma \ref{lemma: Jac field on bill segment with one glanc pt} to construct a vector field $W_{p,0}(t) \in T_{\varphi^t(x_p^-)}F$ defined for all $t \in [0,\beta_p(2 \delta_j,x_p^-)]$ and satisfying properties (1)--(3) in Lemma \ref{lemma: Jac field on bill segment with one glanc pt}. Notice also that $\beta_p(2\delta_p,x_p^-) = \tilde \delta_p^- + \tilde \delta_p^+$, so the vector field $W_{p,0}(t)$ is defined for all $t \in [0,\tilde \delta_p^- + \tilde \delta_p^+]$. Additionally, property $(1)$ in Lemma \ref{lemma: Jac field on bill segment with one glanc pt} says that there exists $c_{p,0} \in \mathbb R$ such that $W_{p,0}(\tilde \delta_p^- + \tilde \delta_p^+) - c_{p,0} H_f(x_p^+) \in T_{x_p^+} S^+_p$.  Now, assume that for any $p \leq j < r$ we have the tangent vector $W_{j,0}(\tilde \delta_{j}^- + \tilde \delta_{j}^+) - c_{j,0}H_f(x_j^+) \in T_{x_{j}^+}S_{j}^+$. Given this initial tangent vector $W_{j,0}(\tilde \delta_{j}^- + \tilde \delta_{j}^+) - c_{j,0}H_f(x_j^+)$, we use Lemma \ref{lemma: billiard jacobi field with transversal reflections} to get a vector field $W_j(t) \in T_{\varphi^t(x_j^+)}F$ defined for all $t \in [0,T_j(x_j^+)]$ and that satisfies properties (1)--(5) in Lemma \ref{lemma: billiard jacobi field with transversal reflections}. Using property (1) in Lemma \ref{lemma: billiard jacobi field with transversal reflections}, we see that $W_j(T_j(x_j^+)) - c_j H_f(x_{j+1}^-) \in T_{x_{j+1}^-}S_{j+1}^-$, for some $c_j \in \mathbb R$. So we use Lemma \ref{lemma: Jac field on bill segment with one glanc pt} again to construct a vector field $W_{j+1,0}(t) \in T_{\varphi^t(x_j^-)}F$ defined for all $t \in [0,\tilde \delta_{j+1}^- + \tilde \delta_{j+1}^+]$ ($\tilde \delta^{\pm}_m = \tilde \delta^{\pm}_0$) and satisfying properties (1)--(3) in Lemma \ref{lemma: Jac field on bill segment with one glanc pt}.

    Now, we define the vector field $W(t) \in T_{\varphi^t(x_0)}F$ for all $t \in (t_p - \tilde \delta^-_p, t_{r+1} + \tilde \delta^+_{r+1})$ by 
    \begin{equation}
        \label{eqn: Jacobi field on segment of closed bill traj with transversal and glanc points}
        W(t) = \begin{cases}
            W_{j,0}(t-t_j + \tilde \delta_j^-) & \text{if} \quad p \leq j \leq r+1, \quad t_j - \tilde \delta_j^- < t < t_j + \tilde \delta_j^+\\
            W_j(t-t_j - \tilde \delta_j^+) & \text{if} \quad p \leq j \leq r, \quad t_j + \tilde \delta_j^+ \leq t \leq t_{j+1} - \tilde \delta_{j+1}^-.
        \end{cases}
    \end{equation}
    From the definition of $W(t)$, one can check that
    \begin{equation}
    \label{eqn: left and right limits at end points of N_0 match}
    \begin{split}
        &W(t_j - \tilde \delta_j^- -0) - c_{j-1} H_f(x_j^-) = W(t_j - \tilde \delta_j^- +0 ) \in T_{x_j^-}S_j^- \quad \text{for all} \quad p < j \leq r+1 \\
        &\text{and}  \quad W(t_j + \tilde \delta_j^+ -0) -c_{j,0} H_f(x_j^+) = W(t_j + \tilde \delta_j^+ +0) \in T_{x_j^+}S_j^+ \quad \text{for all} \quad p \leq j \leq r.
    \end{split} 
    \end{equation}

    Next we proceed with cases.

    \textbf{Case 1:} Suppose $\left \langle \left(d \widetilde \Phi_r(0)  N_0  d \widetilde \Phi_{r-1}(0) \cdots  N_0  d \widetilde \Phi_p(0) \right) \left(\partial_{y_2}\right) , \partial_{\eta_2}  \right \rangle = 0$ for some $0 \leq p \leq r < m$.
    \noindent
    \newline
    In this case, we set the initial tangent vector $w \in T_{x_p^-}S^-$ used in the definition of $W(t)$ to be $w = d\kappa_p^{-1}(\partial_{y_2})$. The definition of $W(t)$ provided in Equation \eqref{eqn: Jacobi field on segment of closed bill traj with transversal and glanc points}, properties (1)--(2) in Lemma \ref{lemma: Jac field on bill segment with one glanc pt}, and property $(1)$ in Lemma \ref{lemma: billiard jacobi field with transversal reflections} imply that there exist constants $a_p,a_{r+1} \in \mathbb R$ such that
    \begin{equation}
        \label{eqn: W(t) at bound points when we assume glan mapped to glanc}
        \begin{split}
            &W(t_p) -a_p H_f(\varphi^{t_p}(x_0)) = d\kappa_p^{-1}(\partial_{y_2}) \in T_{\varphi^{t_p}(x_0)}F \quad \text{and} \\
            &W(t_{r+1}) - a_{r+1} H_f(\varphi^{t_{r+1}}(x_0)) = a\, d\kappa_{r+1}^{-1}(\partial_{y_2}) \in T_{\varphi^{t_{r+1}}(x_0)}F
        \end{split}
    \end{equation}
    for some $a \in \mathbb R$ because $\left(d \widetilde \Phi_r(0)  N_0  d \widetilde \Phi_{r-1}(0) \cdots  N_0  d \widetilde \Phi_p(0) \right) \left(\partial_{y_2}\right) \in \text{span} \{ \partial_{y_2} \}$ by assumption.

    \textbf{Case 2:} Assume that $\det \left( d\widetilde \Phi_{m-1}(0)  N_0  \cdots  d\widetilde \Phi_0 (0) N_{0} - I \right) = 0$.
    \noindent
    \newline
    Set $p = 0$ and $r=m-1$ and choose $w \in T_{x_0^-}S_0^-\setminus \{ 0 \}$ such that 
    $$
    d\widetilde \Phi_{m-1}(0)  N_0  \cdots  d\widetilde \Phi_0 (0) N_{0} (d\kappa_0(w)) = d \kappa_0(w).
    $$ 
    Using this $w$ in the definition of the vector field $W(t)$ in Equation \eqref{eqn: Jacobi field on segment of closed bill traj with transversal and glanc points}, we see that properties $(1)-(2)$ in Lemma \ref{lemma: Jac field on bill segment with one glanc pt} and property $(1)$ in Lemma \ref{lemma: billiard jacobi field with transversal reflections} imply that there exists $a \in \mathbb R$ such that
    \begin{equation}
        \label{eqn: W(t) at bound points when when dP-I not invert}
        \begin{split}
            &W(-\tilde \delta_0^- +0) = W(t_m - \tilde \delta_0^- -0) - a\, H_f(x_m^-) = w, \quad \text{which implies} \\
            &W(t_0) = W(t_m) - \tilde a H_f(x_0) \quad \text{for some} \quad \tilde a \in \mathbb R.
        \end{split}
    \end{equation}
    \bigskip
    \newline
    \textbf{Step 2: } Applying Jacobi Equation.

    For each $0 \leq j < m$, define the set $B_j = \{ t_{j,0}, \dots, t_{j,l_j} \}$, $t_j < t_{j,0} < \dots < t_{j,l_j} < t_{j+1}$, to be such that $\varphi^{t_{j,k}}(x_0) \in F \cap \{ \pi^*b = 0 \} \cap \{ H_f \pi^*b \neq 0\}$ for all $t_{j,k} \in B_j$ and $\varphi^t(x_0) \in F \cap \{ \pi^*b >  0 \}$ for all $t \in (t_j,t_{j+1}) \setminus B_j$. This set $B_j$ specifies transversal reflections of the billiard flow $\varphi^t$ that happen between times $t_j$ and $t_{j+1}$.   

    Consider the Jacobi field $W(t)$ from case 1 (or case 2) in Step 1 defined for all $t \in [t_i,t_f] = [t_p, t_{r+1}]$ (or $t \in [t_i, t_f] = [t_0, t_m ]$). The horizontal and vertical decomposition of $W(t)$  is given by 
    \begin{equation}
    \label{eqn: W explicitly as jac field}
        W(t) = (J(t),\dot J(t)),
    \end{equation}
    where $J(t)$ is a Jacobi field along $\gamma(t) = \left(\pi \circ \varphi^t\right)(x_0)$. So, Equation \eqref{eqn: jacob field equation} implies 
    
    \begin{equation}
    \label{eqn: integrated Jac equation with end points not handled}
        \begin{split}
            &\int_{t_i}^{t_f} \left(- \left( \| J \|^2_g - \langle J, \dot \gamma \rangle_g^2 \right) K(J, \dot \gamma) + \| \dot J \|_g^2 \right) dt = \langle J, \dot J  \rangle_g(t_f) - \left \langle J, \dot J \right \rangle_g(t_i) \\
            & \quad \quad \quad  + \sum_{t_i \leq t_j < t_f} \left( \sum_{t_{j,k} \in B_j } \left\langle \dot J, J \right\rangle_g \left(t_{j,k}-0\right) - \left\langle \dot J, J \right\rangle_g \left(t_{j,k}+0\right) \right) \\
            & \quad \quad \quad + \sum_{t_i < t_j \leq t_f} \left\langle \dot J, J \right\rangle_g \left(t_j - \tilde \delta_j^- -0\right) - \left\langle \dot J, J \right\rangle_g \left(t_j - \tilde \delta_j^- +0\right) \\
            & \quad \quad \quad + \sum_{t_i \leq t_j < t_f} \left\langle \dot J, J \right\rangle_g \left(t_j + \tilde \delta_j^+ -0\right) - \left\langle \dot J, J \right\rangle_g \left(t_j + \tilde \delta_j^+ +0\right).
        \end{split}
    \end{equation}
    
    Now suppose that $x = (p,v) \in F = S\widetilde \Sigma$ and $w_1,w_2 \in T_{x}F$ are such that $w_1 - w_2 = cH_f(x)$, for some $c \in \mathbb R$. Recall that if $\sigma_i(s) = (p_i(s),v_i(s))$, $i=1,2$, is a curve in $F$ such that $\sigma_i(0) = x$ and $\dot \sigma_i(0) = w_i$, then by Equation \eqref{eqn: hor and vert decomp} the horizontal and vertical decomposition of $w_i$ is given by 
    $$
    w_i = (\dot p_i(0), \nabla_{\dot p_i}v_i (0)).
    $$
    Because $w_1-w_2 = c H_f(x)$, we have that $(\dot p_1(0), \nabla_{\dot p_1}v_1 (0)) - (\dot p_2(0), \nabla_{\dot p_2}v_2 (0)) = c(v,0)$,  where $(v,0)$ is the horizontal and vertical decomposition of the generator of the geodesic flow $H_f(x)$. As a result, we get that 
    \begin{equation}
        \label{eqn: diff of inner prod of vert and hor components for two vect field differing by gen of geod flow}
        \langle \dot p_1 (0), \nabla_{\dot p_1} v_1 (0) \rangle_g - \langle \dot p_2 (0), \nabla_{\dot p_2} v_2 (0) \rangle_g = c \langle v, \nabla_{\dot p_1} v_1 (0) \rangle_g = 0,
    \end{equation}
    where the last equality follows from the observation that $\langle v_1(s), v_1(s) \rangle_g = 1 $, which implies that $2 \langle v, \nabla_{\dot p_1} v_1 (0) \rangle_g = \partial_s|_{s=0}\langle v_1(s), v_1(s) \rangle_g = 0$.

    Set $n = \frac{\nabla b}{\| \nabla b\|_g}$. Then for each $t_{j,k} \in B_j$, the definition of $W(t)$ in Equation \eqref{eqn: Jacobi field on segment of closed bill traj with transversal and glanc points}, Properties $(3)$ and $(4)$ in Lemma \ref{lemma: billiard jacobi field with transversal reflections}, Equation \eqref{eqn: differ of R in horiz and vert decomp}, and Equation \eqref{eqn: diff of inner prod of vert and hor components for two vect field differing by gen of geod flow} imply that 
    \begin{equation}
    \label{eqn: sign from transversal reflections}
        \begin{split}
            &\left\langle \dot J, J \right\rangle_g \left(t_{j,k}-0\right) - \left\langle \dot J, J \right\rangle_g \left(t_{j,k}+0\right) \\
            & \quad \quad \quad = 2 \left(H_f \pi^*b\right)(\varphi^{t_{j,k}-0}(x_0)) \left \langle \nabla^2 b, \frac{1}{\| \nabla b\|^2} J(t_{j,k}+0) \otimes J (t_{j,k}+0)\right \rangle_g \\
            & \quad \quad \leq -c_{j,k} \left\| J\left(t_{j,k} +0\right) \right\|_g^2  \quad (c_{j,k}>0),
        \end{split}
    \end{equation}
    where in the first equality we used the fact that $J(t_{j,k}+0) \in T \partial \Sigma$, which implies that $\left \langle J(t_{j,k}+0), \nabla b \right \rangle = 0$. The last inequality follows from the assumption that the boundary $\partial \Sigma$ is concave as made precise in Equation \eqref{eqn: concave boundary} and from the fact that $H_f \pi^*b\left( \varphi^{t_{j,k}-0}(x_0)\right) < 0$.
    
    In the case when $W(t)$ is from Case 1 in Step 1, Lemma \ref{lemma: vec tang to boundary} and Equation \eqref{eqn: W(t) at bound points when we assume glan mapped to glanc} imply that there exists $c_i,c_f \in \mathbb R$ such that 
    $$
    W(t_i) - c_i H_f(\varphi^{t_i}(x_0)) \in T_{\varphi^{t_i}(x_0)} (S \partial \Sigma) \quad \text{and} \quad W(t_f) - c_f H_f(\varphi^{t_f}(x_0))  \in T_{\varphi^{t_f}(x_0)} (S \partial \Sigma),
    $$
    so Equation \eqref{eqn: diff of inner prod of vert and hor components for two vect field differing by gen of geod flow} give $\left \langle J, \dot J \right \rangle_g(t_f) = \left \langle J, \dot J \right \rangle_g(t_i) = 0 $ since $\Sigma$ has dimension $2$ according to the assumptions made in Equation \eqref{eqn: curvature assumption}.
    
    When $W(t)$ is from Case 2 in Step 1, then Equation \eqref{eqn: W(t) at bound points when when dP-I not invert} and Equation \eqref{eqn: diff of inner prod of vert and hor components for two vect field differing by gen of geod flow} give that $\left \langle J, \dot J \right \rangle_g(t_f) - \left \langle J, \dot J \right \rangle_g(t_i) = 0 $. 

    Furthermore, Equation \eqref{eqn: left and right limits at end points of N_0 match} and Equation \eqref{eqn: diff of inner prod of vert and hor components for two vect field differing by gen of geod flow} gives that 
    $$
    \begin{aligned}
        &\left\langle \dot J, J \right\rangle_g \left(t_j - \tilde \delta_j^- -0\right) - \left\langle \dot J, J \right\rangle_g \left(t_j - \tilde \delta_j^- +0\right) = 0 \quad \text{for all} \quad t_i < t_j \leq t_f \quad \text{and} \\
        & \left\langle \dot J, J \right\rangle_g \left(t_j + \tilde \delta_j^+ -0\right) - \left\langle \dot J, J \right\rangle_g \left(t_j + \tilde \delta_j^+ +0\right) = 0 \quad \text{for all} \quad t_i \leq t_j < t_f.
    \end{aligned}
    $$
    
    We can then use the inequality in Equation \eqref{eqn: sign from transversal reflections} to simplify Equation \eqref{eqn: integrated Jac equation with end points not handled} as follows: 
    \begin{equation*}
        \begin{split}
            \int_{t_i}^{t_f} &\left(- \left( \| J \|^2_g - \left\langle J, \dot \gamma \right\rangle_g^2 \right) K(J, \dot \gamma) + \| \dot J \|_g^2 \right) dt \leq \sum_{t_i \leq t_j < t_f} \left( \sum_{t_{j,k} \in B_j } -c_{j,k} \left\| J\left(t_{j,k} +0\right) \right\|_g^2   \right)  ,
        \end{split}
    \end{equation*}
    where $c_{j,k}>0$. The assumptions made in Equation \eqref{eqn: curvature assumption} imply that $J=0$ and $\dot J = 0$ for all $t \in [t_i,t_f]$. In fact, if the sectional curvature $K$ is negative, then we must have that $\left(\| J(t) \|^2_g - \left\langle J(t), \dot \gamma(t) \right\rangle_g^2\right)K(J(t),\dot \gamma(t)) = 0$ for all $t \in [t_i, t_f]$, which happens when $J = 0$ or when $J = c \dot \gamma $ for some constant $c \neq 0$. But we cannot have $J = c \dot  \gamma$ because otherwise $W(t_i) = (J(t_i), \dot J(t_i))$ would be in the span of $H_f(\varphi^{t_i}(x_0))$, which would contradict Equation \eqref{eqn: W(t) at bound points when we assume glan mapped to glanc} and Equation \eqref{eqn: W(t) at bound points when when dP-I not invert}. On the other hand, if the sectional curvature $K = 0$, then for each $j$ such that $t_i \leq t_j < t_f$ (in fact notice there exists at least one such $j$) the set $B_j$ is non-empty; so we deduce that $J(t_{j,k}) = 0$ for all $t_{j,k} \in B_j$. As a result, we also obtain that $J(t) = 0$ for all $t \in [t_i, t_f]$.
\end{proof}

As an immediate corollary of Lemma \ref{lemma: no tang vec mapped to tang vec} above, we get the following.
\begin{corr}
    \label{cor: lead term in trac of dP non-zero}
    Let $\textbf i=(i_0, \dots, i_{m-1}) \in \{ 0, 1 \}^m$ and the map $B_{\textbf i}$ be given by Equation \eqref{eqn: matrices in expansion of det of Poinc return map minus ident}. If the assumptions in Equation \eqref{eqn: curvature assumption} hold and $i_j = 1$ for some $0 \leq j < m$, then
    \begin{equation*}
        \tr B_{\textbf i}(0) \neq 0.
    \end{equation*}
\end{corr}
\begin{proof}
    From the definition of $B_{\textbf i}$, see that $B_{\textbf i}(0) = d \widetilde \Phi_{m-1}(0) N_{i_{m-1}} \cdots d \widetilde \Phi_0(0) N_{i_0}$. Without loss of generality, we can assume that $i_0 = 1$, since otherwise we have 
    $$
    \tr B_{\textbf i}(0) = \tr\left( d\widetilde \Phi_{j-1}(0) N_{i_{j-1}}  \cdots  d\Phi_0 (0) N_{i_0}  d \widetilde \Phi_{m-1} (0) N_{i_{m-1}}  \cdots d \widetilde \Phi_j (0) N_{i_j} \right),
    $$ 
    where $j$ is such that $i_j =1$ and $i_l = 0$ for all $0 \leq l < j$. In fact, notice that such $j$ exists because we assumed that $i_l = 1$ for some $0 \leq l < m$.

    Using the definition of $N_{i_l}$ provided in Equation \eqref{eqn: matrices in expansion of det of Poinc return map minus ident}, we see that $N_1 = \begin{pmatrix}
        0 & 1 \\
        0 & 0
    \end{pmatrix}$ since we assume that $n=2$. With $i_0 = 1$, we immediately get that 
    \begin{equation*}
        \begin{split}
            \tr B_{\textbf i}(0)= \left \langle \left( d \widetilde \Phi_{m-1}  N_{i_{m-1}} (0)\cdots  d\widetilde \Phi_1(0)N_{i_1}d \widetilde \Phi_0(0) \right)  \left(\partial_{y_2}\right)  , \partial_{\eta_2}\right \rangle,
        \end{split}
    \end{equation*}
    where $y_2$ and $\eta_2$ are defined as in Equation \eqref{eqn: matrices in expansion of det of Poinc return map minus ident}. Let $A_{\textbf i} = \{ j_0 = 0, j_1, \dots, j_r \}$ be such that $0 = j_0 < j_1 < \dots < j_r < m$, $i_{j_0} = i_{j_1} = \dots = i_{j_r} = 1$, and $i_j = 0$ for all $j \in \{ 0, 1, \dots, m-1 \} \setminus A_{\textbf i}$. It follows that 
    
    \begin{equation*}
        \begin{split}
             &\tr B_{\textbf i}(0) = \prod_{p=0}^r \left \langle \left( d \widetilde \Phi_{j_{p+1}-1} (0) N_0 \cdots d \widetilde \Phi_{j_p+1}(0) N_0 d \widetilde \Phi_{j_p}(0) \right)  \left(\partial_{y_2}\right)  , \partial_{\eta_2}\right \rangle, 
        \end{split}
    \end{equation*}
    where we set $j_{r+1} = m$. Then, Lemma \ref{lemma: no tang vec mapped to tang vec} gives that for each $j_p \in A_{\textbf i}$, we have
    $$
    \left \langle \left( d \widetilde \Phi_{j_{p+1}-1} (0) N_0 \cdots  d \widetilde \Phi_{j_p}(0) \right) \left(\partial_{y_2}\right)  , \partial_{\eta_2}\right \rangle \neq 0,
    $$
    which finishes the proof.
\end{proof}

Let $P: U \rightarrow S_0^-$ be the Poincar\'e return map defined by Lemma \ref{lemma: Poinc return map for closed traj with m glancing pts} and $\widetilde Q_j$ the maps given by Equation \eqref{eqn: Poincare map up to the jth Poincare sec before glancing point}. For each $\textbf k = (k_0,\dots, k_{m-1}) \in \{0,1 \}^m$, we define the subset $U_{\textbf k} \subset U$ by 
\begin{equation}
    \label{eqn: partition of U based on how many boundary hitting points}
    U_{\textbf k} = \left\{ x \in U: (-1)^{k_j + 1} \eta_n( \widetilde Q_j(\kappa_0(x))) > 0 \quad \text{for all} \quad 0 \leq j < m \right\}.
\end{equation}
Each of these sets $U_{\textbf k}$ contains points in $U$ such that the billiard flow through those points intersects a subset of the boundary $S_{\partial \Sigma} \Sigma$ close to the glancing set $S \partial \Sigma$ near the same glancing points. For example, $U_{(1, 0 , \dots, 0)}$ contains those points in $U$ where the billiard flow through them intersects the boundary near the glancing point $x_0$ and possibly at other points which are away from any of the other glancing points $x_1, \dots, x_{m-1}$. Points in the set $U_{(0,\dots,0)}$ are such that the billiard flow through them does not interest the boundary near any of the glancing points $x_0, \dots, x_{m-1}$. Notice that 
%\SD{maybe mention that here we use that the maps in question preserve measure 0 sets (note this is not automatic for homeomorphisms); in fact they preserve the Lebesgue measure even}
\begin{equation}
    \label{eqn: the sets U_k form a partition of U up to a set of meas 0}
    \begin{split}
        & x \in U \setminus (\bigcup_{\textbf k \in \{0,1 \}^m} U_{\textbf k}) \iff \eta_n(\widetilde Q_j(\kappa_0(x))) = 0 \quad \text{for some} \quad 0 \leq j < m,  \\
        &\text{so the set} \quad U \setminus (\bigcup_{\textbf k \in \{0,1 \}^m} U_{\textbf k}) \quad \text{has measure zero.} 
    \end{split}
\end{equation}

For each $\textbf k \in (k_0, \dots , k_{m-1}) \in \{0 , 1 \}^m$, define
\begin{equation}
\label{eqn: indices that record bound hit time near glancing same on subset mak partition}
    \begin{split}
        &\mathcal{I}(U_{\textbf k}) = \left\{ \textbf i = (i_0, \dots, i_{m-1}) \in \{0, 1\}^m: 0 \leq i_l \leq k_l \quad \text{for all} \quad 0 \leq l < m \right \} \; \text{and see that}  \\   
        &\mathcal I(z) = \mathcal I(U_{\textbf k}) \quad \text{for all} \quad z \in \kappa_0(U_{\textbf k}), 
    \end{split}
\end{equation}
where $\mathcal I(z)$ is given by Equation \eqref{eqn: sets recording if we hit boundary at j glanc point}.

%\SD{Shouldn't there be an easy formula for $\mathcal I(U_{\mathbf k})$?}

If we consider $\widetilde P$ defined in Equation \eqref{eqn: Poinc return map in sympl coord} as the Poincar\'e return map $P$ in the local coordinate $\kappa_0$, then we get an expansion of $\det( d\widetilde P(z) - I)$ in terms of singular functions where the coefficient corresponding to each singular function in the expansion does not vanish. Precisely, we show the following. 

\begin{prop}
\label{prop: expans of det of Poinc ret map in coord minus ident with non-zero coeff}
Let the functions $\lambda_{\textbf i}$ be defined by Equation \eqref{eqn: fns giving assymptotics of determinant of diff of Poinc return map minus ident} and the matrices $B_{\textbf i}$ by Equation \eqref{eqn: matrices in expansion of det of Poinc return map minus ident}. If we make $U$ small enough and assume that the assumptions in Equation \eqref{eqn: curvature assumption} hold, then there exists a constant $C > 0$ such that 
\begin{equation}
    \label{eqn: upper bound on trace of matrix in expans of det of Poin minus id}
    C^{-1} <|\tr B_{\textbf i}(z) | < C \quad \text{for all} \quad \textbf i \in \{ 0,1 \}^m \setminus (0, \dots, 0) \quad \text{and all} \quad z \in \kappa_0(U).
\end{equation}
Moreover, we have
    \begin{equation}
        \label{eqn: expansion of det of conjug Poinc return map minus Ident with non-zero coeff}
        \det( d \widetilde P(z) - I ) = 2 - \sum_{\textbf i \in \mathcal I(U_{\textbf k})} \lambda_{\textbf i}(z) \tr B_{\textbf i}(z) \neq 0 \quad \text{for all} \quad \textbf k \in \{ 0, 1 \}^m \quad \text{and} \quad z \in \kappa_0(U_{\textbf k}) 
    \end{equation}
    and
    \begin{equation}
        \label{eqn: det of Poinc minus ident when miss bound near glanc set is non-zero and bounded}
        C^{-1} < |\det(d \widetilde P(z) - I )| < C \quad \text{for all} \quad z \in \kappa_0\left(U_{(0,\dots,0)}\right). 
    \end{equation}
\end{prop}
\begin{proof}

    From the definition of the matrices $B_{\textbf i}$, we have that 
    $$
    B_{\textbf i}(z) = d\widetilde \Phi_{m-1}(\widetilde R (\widetilde Q_{m-1}(z))) N_{i_{m-1}} \cdots d\widetilde \Phi_{0}(\widetilde R (\widetilde Q_{0}(z))) N_{i_{0}},
    $$
    where the functions $\widetilde \Phi_j$ are smooth and the functions $\widetilde R$ and $\widetilde Q_j$ continuous. Recall also that $N_0 = \begin{pmatrix}
        1 & 0 \\
        0 & 1
    \end{pmatrix}$ and $N_1 = \begin{pmatrix}
        0 & 1\\
        0 & 0
    \end{pmatrix}$. After possibly making $U$ small, Corollary \ref{cor: lead term in trac of dP non-zero} implies that there exists a constant $C > 0$ such that for all $\textbf i \in \{0, 1 \}^m \setminus \{ (0,\dots,0) \}$ we have $C^{-1}<|\tr B_{\textbf i}(z)| < C$ for all $z \in \kappa_0(U)$ because $\tr B_{\textbf i}(z)$ is continuous in $z$.

    Let $\textbf k \in \{0,1 \}^m$. By Equation \eqref{eqn: indices that record bound hit time near glancing same on subset mak partition}, see that $\mathcal I(z) = \mathcal I(U_{\textbf k})$ for all $z \in \kappa_0(U_{\textbf k})$ and that $\textbf{k} \in \mathcal I(U_{\textbf k})$. Therefore, Lemma \ref{lemma: det of Poinc ret map minus ident} gives that 
    $$
    \det( d \widetilde P(z) - I ) = 2 - \sum_{\textbf i \in \mathcal I(U_{\textbf k})} \lambda_{\textbf i}(z) \tr B_{\textbf i}(z) \quad \text{for all} \quad z \in \kappa_0(U_{\textbf k}).
    $$
    If $\textbf k \neq 0$ and $z \in \kappa_0(U_{\textbf{ k}})$, then the definition of $\lambda_{\textbf{k}}>0$ shows that $\lambda_{\textbf{k}}(z) \rightarrow + \infty$ as $z  \rightarrow 0$. Furthermore, if $\textbf{i} \in \mathcal{I}(U_{\textbf{k}})$ is such that $\textbf i \neq \textbf k$, then we have that 
    $$
    \frac{\lambda_{\textbf i}(z)}{\lambda_{\textbf k}(z)} \rightarrow 0 \quad \text{as} \quad z \rightarrow 0.
    $$
    Therefore, if we make $U$ small enough, we can ensure that 
    $$
    \det( d \widetilde P(z) - I ) \neq 0
    $$
    for all $z \in \kappa_0(U_{\textbf k})$ with $\textbf k \neq 0$.

    Finally, for any $z \in \kappa_0\left(U_{(0,\dots,0)}\right)$, see that 
    $$
    d\widetilde P(z) = d\widetilde \Phi_{m-1}( \widetilde Q_{m-1}(z)) N_{0} \cdots d\widetilde \Phi_{0}(\widetilde Q_{0}(z)) N_{0}
    $$
    because $\widetilde R (\widetilde Q_j(z)) = \widetilde Q_j(z)$ since $\eta_2(\widetilde Q_j(z)) < 0$ for all $0 \leq j < m$. After making $U$ small again if needed and $C$ large enough, Lemma~\ref{lemma: no tang vec mapped to tang vec} implies that 
    $$
    C^{-1} < |\det(d \widetilde P(z) - I )| < C \quad \text{for all} \quad z \in \kappa_0\left(U_{(0,\dots,0)}\right). 
    $$
\end{proof}

As an immediate consequence of Proposition \ref{prop: expans of det of Poinc ret map in coord minus ident with non-zero coeff}, we get the following:
\begin{corr}
    \label{corr: limits diff of Poinc retur map minus id based on reflect pts}
    Let $y_0^-$ be defined by Lemma \ref{lemma: Poinc return map for closed traj with m glancing pts}, $\textbf k \in \{0, 1 \}^m$, and $x_n \in U_{\textbf k}$ be such that $x_n \rightarrow y_0^-$. If the assumptions in Equation \eqref{eqn: curvature assumption} hold, then 
\begin{equation}
    \label{eqn: invers of det of Poinc return map minus ident on each part component}
    \begin{split}
        c_{\textbf k} &:= \lim_{n \rightarrow \infty} \frac{1}{\left|\det \left( d \left( \kappa_0 \circ  P \circ  \kappa_0^{-1} \right)(\kappa_0(x_n)) - I \right) \right|} \\
        &\, = \begin{cases}
            \left|\det \left( d\widetilde \Phi_{m-1}(0)  N_0  \cdots  d\widetilde \Phi_0 (0) N_{0} - I \right) \right|^{-1} \neq 0 & \text{if} \; \textbf{ k} = 0 \\
            0 & \text{else},
        \end{cases}
    \end{split}
\end{equation}
where $\widetilde \Phi_j$ are defined by Equation \eqref{eqn: Poinc return map in sympl coord} and $N_0$ by Equation \eqref{eqn: matrices in expansion of det of Poinc return map minus ident}. Furthermore, if $\psi$ is another coordinate for $S \Sigma$ near $y_0^-$, then 
\begin{equation}
    \label{eqn: coord independence of the constants c_k}
    \lim_{n \rightarrow \infty} \frac{1}{\left|\det \left( d \left( \psi \circ  P \circ  \psi^{-1} \right)(\psi(x_n)) - I \right) \right|} = c_{\textbf k}.
\end{equation}
\end{corr}
\begin{proof}
    From the definition of $\widetilde P = \kappa_0  \circ P \circ \kappa_0^{-1}$ in Equation \eqref{eqn: Poinc return map in sympl coord}, the function $\widetilde P$ is continuous and is given explicitly by $\widetilde P = \widetilde \Phi_{m-1} \circ \widetilde R \circ \cdots \circ \widetilde \Phi_0 \circ \widetilde R$, where $\widetilde R$ is defined by Lemma \ref{lemma: Poincare map between Poinc secs near glanc point}. Suppose that $\textbf k = 0$. Then for any $z \in \kappa_0(U_{\textbf k})$, we have 
    $$
    \widetilde P(z) = \widetilde \Phi_{m-1}\circ \widetilde \Phi_{m-2} \circ \dots \circ \widetilde \Phi_1 \circ \widetilde \Phi_0(z).
    $$
    In fact, from the definition of $U_0$, we have that $\eta_n (\widetilde Q_j(z)) < 0$ for all $0 \leq j < m$ and all $z \in \kappa_0(U_{0})$, where $\widetilde Q_j$ is defined by Equation \eqref{eqn: Poincare map up to the jth Poincare sec before glancing point}. Consequently, we get that $\widetilde R(\widetilde Q_j(z)) = \widetilde Q_j(z)$ for all $0 \leq j < m$ and all $z \in \kappa_0(U_0)$. So, Proposition \ref{prop: expans of det of Poinc ret map in coord minus ident with non-zero coeff} implies that  $|\det( d\widetilde P(\kappa_0(x_n)) - I) | \rightarrow |\det(d\widetilde \Phi_{m-1}(0) N_0 d\widetilde \Phi_0(0)N_0 - I)|$ as $n \rightarrow \infty$ by noticing that $N_0$ is the identity matrix and that $\widetilde Q_j(0) = 0$ for all $0 \leq j < m$.

    Next, suppose that $\textbf k = (k_0,\dots , k_{m-1}) \in \{0,1\}^m \setminus 0 $. Proposition \ref{prop: expans of det of Poinc ret map in coord minus ident with non-zero coeff} says that there exists a constant $C>0$ such that for any $z \in \kappa_0(U_{\textbf k})$, we have 
    $$
    \det(d\widetilde P(z) - I) = 2 - \sum_{\textbf i \in \mathcal I(U_{\textbf k})} \lambda_{\textbf i}(z) \tr B_{\textbf i}(z), 
    $$
    where $\lambda_{\textbf i}$ are defined by Equation \eqref{eqn: fns giving assymptotics of determinant of diff of Poinc return map minus ident} and $ C^{-1} < |\tr B_{\textbf i}(z)| < C $ for all $\textbf i \in \mathcal I(U_{\textbf k})$.
    
    From the definition of $\mathcal I (U_{\textbf k})$ given in Equation \eqref{eqn: indices that record bound hit time near glancing same on subset mak partition}, it is immediate that $\textbf k \in \mathcal I (U_{\textbf k})$. Additionally, for each $\textbf i  = (i_0, \dots, i_{m-1}) \in \mathcal I (U_{\textbf k}) \setminus \{ \textbf k\}$, there exists $0 \leq l < m$ such that $i_l = 0$ and $k_l = 1$. Consequently, for all $z \in \kappa_0(U_{\textbf k})$, we get that 
    $$
    0 < \frac{\lambda_{\textbf i}(z) }{\lambda_{\textbf k}(z) }  \leq \eta_n(\widetilde Q_l(z)),
    $$
    where $\eta_n(\widetilde Q_l(z)) \rightarrow 0$ as $z \rightarrow 0$.
    
    From the definition of the matrix $B_{\textbf k}$ in Equation \eqref{eqn: matrices in expansion of det of Poinc return map minus ident}, we see that $B_{\textbf k}(z)$ is continuous in $z$. Then, Proposition \ref{prop: expans of det of Poinc ret map in coord minus ident with non-zero coeff} implies that there exists a constant $r \neq 0$ such that $\tr B_{\textbf k}(\kappa_0(x_n)) \rightarrow r$ as $n \rightarrow \infty$. But since $\lambda_{\textbf k}(\kappa_0(x_n)) \rightarrow + \infty$ as $n \rightarrow +\infty$, we get that 
    \begin{equation}
        \lim_{n \rightarrow \infty} \frac{1}{|d \widetilde P (\kappa_0(x_n) ) - I|} = 0.
    \end{equation}
    Thus, we have shown Equation \eqref{eqn: invers of det of Poinc return map minus ident on each part component}.

    To prove Equation \eqref{eqn: coord independence of the constants c_k}, notice that $\psi \circ P \circ \psi^{-1} = (\psi \circ \kappa_0^{-1}) \circ (\kappa_0 \circ P \circ \kappa_0^{-1}) \circ (\kappa_0 \circ \psi^{-1})$. As a result, for any $\textbf k \in \{ 0, 1\}^m$ and any $x \in U_{\textbf k}$ we have that 
    $$
    d (\psi \circ P \circ \psi^{-1}) (\psi(x)) = d(\psi \circ \kappa_0^{-1})(\widetilde P (\kappa_0(x)) \; d \widetilde P(\kappa_0(x)) \; d(\kappa_0 \circ \psi^{-1})(\psi(x))
    $$
    and 
    $$
    \det(d (\psi \circ P \circ \psi^{-1}) (\psi(x)) - I) = \det(M(\kappa_0(x))d \widetilde P (\kappa_0(x)) - I), 
    $$
    where $M(z) = d(\kappa_0 \circ \psi^{-1})(\psi \circ \kappa_0^{-1}(z)) \, d (\psi \circ \kappa_0^{-1})(\widetilde P (z) )$ for all $z \in \kappa_0(U)$ and $M(0) = I$. Lemma \ref{lemma: det of Poinc ret map minus ident} and Equation \eqref{eqn: indices that record bound hit time near glancing same on subset mak partition} imply that 
    \begin{equation*}
        \begin{split}
            \det(d (\psi \circ P \circ \psi^{-1}) (\psi(x_n)) - I) & = 1 + \det(M(\kappa_0(x_n))) \\
            &- \sum_{\textbf i \in \mathcal I (U_{\textbf k})} \lambda_{\textbf i}(\kappa_0(x_n)) \tr (M(\kappa_0(x_n)) B_{\textbf i}(\kappa_0(x_n))). 
        \end{split}
    \end{equation*}
    See that $\det(M(\kappa_0(x_n))) \rightarrow 1$ and $\tr (M(\kappa_0(x_n)) B_{\textbf i}(\kappa_0(x_n))) \rightarrow \tr (B_{\textbf i}(0))$ as $n \rightarrow \infty$. Thus, we conclude that 
    $$
    \lim_{n \rightarrow \infty} \frac{1}{|d (\psi \circ P \circ \psi^{-1}) (\psi(x_n)) - I|} = c_{\textbf k}
    $$
    as desired.
\end{proof}
\subsection{Injectivity of the Poincar\'e return map minus identity.}
We consider a closed billiard trajectory with exactly one glancing point and show that the Poincar\'e return map in a local coordinate minus the identity is injective when restricted to the subsets $U_{\textbf k}$ defined by Equation \eqref{eqn: partition of U based on how many boundary hitting points}.  Specifically,

\begin{lemm}
    \label{lemma: injectivity for one glancing}
    Let the Poincar\'e section $U$ and $m \in \mathbb N$ be defined by Lemma \ref{lemma: Poinc return map for closed traj with m glancing pts}. If the assumptions in Equation \eqref{eqn: curvature assumption} hold, $m=1$, and $U$ is made small enough, then 
    \begin{equation}
        \label{eqn: injectivity for one glancing}
        \begin{split}
            &\widetilde P(z) - z \quad \text{restricted on the set $\overline{\kappa_0(U_0)}$ or $\overline{\kappa_0(U_1) }$ is injective}, 
        \end{split}
    \end{equation}
    where $\widetilde P$ is defined by Equation \eqref{eqn: Poinc return map in sympl coord}, $\kappa_0$ by Lemma \ref{lemma: relating billiard flow with flow from fried model}, and the sets $U_{\textbf k}$ ( $\textbf k = 0 , 1$) by Equation \eqref{eqn: partition of U based on how many boundary hitting points}. Here $\overline{\kappa_0(U_{\textbf k})}$ ($\textbf k = 0, 1$) denotes the closure of $\kappa_0(U_{\textbf k})$.
\end{lemm}
\begin{proof}

    Since $m=1$, Equation \eqref{eqn: Poinc return map in sympl coord} gives that 
    $$
    \widetilde P = \widetilde \Phi_0 \circ \widetilde R,
    $$
    where  
    $$
    \widetilde R\left(y_2, \eta_2\right) = \begin{cases}
        \left(y_2, \eta_2\right) &\text{if} \quad \eta_2 \leq 0\\
        \left(y_2 + 2 \sqrt{\eta_2}, \eta_2\right) &\text{if} \quad \eta_2 > 0 
    \end{cases}
    $$ 
    according to Lemma \eqref{lemma: Poincare map between Poinc secs near glanc point}. Let the set $U$ be given by Lemma \ref{lemma: Poinc return map for closed traj with m glancing pts}.
    
    \noindent
    \textbf{Step 1:} Proving the injectivity of $\widetilde P(z) - z$ on the set $\overline{\kappa_0(U_0)}$. 
    
    Notice that by definition $\kappa_0(U_0) = \kappa_0(U) \cap \{ \eta_2 < 0 \}$ and that for any ~$z \in \kappa_0(U) \cap \{ \eta_2 \leq 0 \}$ we have 
    $$
    \widetilde P(z) - z = \widetilde \Phi_0(z) - z.
    $$
    Lemma \ref{lemma: no tang vec mapped to tang vec} gives that $d \widetilde \Phi_0(0) N_0 - I$ is invertible, with $N_0$ being the identity matrix as specified in Equation \eqref{eqn: matrices in expansion of det of Poinc return map minus ident}. So by making $U$ small if needed, the inverse function theorem gives that the restriction of the function $\widetilde P(z) - z $ on the set $\overline{\kappa_0(U_0)}$ is injective.
    
    \bigskip
    \noindent
    \textbf{Step 2:} Proving the injectivity of $\widetilde P(z) - z$ on the set $\overline{\kappa_0(U_1)}$.
    
    By definition, we have that $\kappa_0(U_1) = \kappa_0(U) \cap \{ \eta_2 > 0 \}$. Assume that there exist $w \in \overline{\kappa_0(U_1)}$ and $z \in \overline{\kappa_0(U_1)} \setminus \{ \eta_2 = 0 \}$ such that  $\widetilde P (z) - z = \widetilde P(w) - w$. We define $z(t) = t z + (1-t) w \in \overline{\kappa_0(U_1)}$. In fact we can guarantee that $z(t) \in \overline{\kappa_0(U_1)}$ by making the set $U$ small and convex. See that $z(t) \in  \overline{\kappa_0(U_1)} \setminus \{ \eta_2 = 0\}$ for all $t \in (0,1]$, so it follows that 
    $$
    z-w = \widetilde P(z) - \widetilde P(w) = \int_0^1 d \widetilde P (z(t)) (z-w) dt, 
    $$
    where $d\widetilde P (z(t)) = d\widetilde \Phi_0( \widetilde R (z(t)))  \begin{pmatrix}
        1 & \frac{1}{\sqrt{\eta_2(z(t))}}\\
        0 & 1
    \end{pmatrix}$ for all $t \in (0,1]$. Substituting for $d\widetilde P(z(t))$, we get that 
    \begin{equation}
    \label{eqn: cond that two points maps have same image under P minus ident}
        z-w = \int_0^1 \frac{\langle z-w, e_2 \rangle}{\sqrt{\eta_2(z(t))}}d\widetilde \Phi_0(\widetilde R(z(t))) (e_1) dt + \int_0^1 d\widetilde \Phi_0(\widetilde R(z(t))) (z-w) dt,
    \end{equation}
    where $e_1 = \partial_{y_2}$, $e_2 = \partial_{\eta_2}$,
    and $\langle \cdot, \cdot \rangle$ is the standard inner product on $\mathbb R^2$.

    Set $z-w = r v$, where $r \geq 0$ and $v \in \mathbb R^2$ is such that $\langle v , v \rangle = 1$. Additionally, choose $\gamma > 0$ small such that
    \begin{equation}
    \label{eqn: upper bound on norm of dPhi_0 of e_j}
       \gamma^{\frac{1}{4}} < |\langle d \widetilde \Phi_0(0)(e_1) , e_2 \rangle|, \quad  \gamma^{\frac{1}{4}} < \left| d \widetilde \Phi_0(0)(e_1) \right| < \gamma^{-
    \frac{1}{4}} \quad \text{and} \quad \gamma^{\frac{1}{4}} < \left| d \widetilde \Phi_0(0)(e_2) \right| < \gamma^{-
    \frac{1}{4}}.
    \end{equation}
    Indeed, notice that Lemma \ref{lemma: no tang vec mapped to tang vec} gives the lower bound $\gamma^{\frac{1}{4}} < |\langle d \widetilde \Phi_0(0)(e_1) , e_2 \rangle|$.
    
    \noindent
    \textbf{Case 1:} Suppose that $| \langle v, e_2 \rangle| \geq \gamma$.
    
    Equation \eqref{eqn: cond that two points maps have same image under P minus ident} gives that 
    \begin{equation}
    \label{eqn: proper of vector joining points with same displ of Poinc return map}
        \begin{split}
            &r\left( 1 - \int_0^1 \frac{1}{\sqrt{\eta_2(z(t))}} \left\langle d\widetilde \Phi_0(\widetilde R(z(t))) (e_1), e_2 \right \rangle dt \right) \langle v , e_2 \rangle \\
            & \quad \quad \quad = r \int_0^1 \left \langle d\widetilde \Phi_0(\widetilde R(z(t))) (v) , e_2 \right \rangle dt.
        \end{split}
    \end{equation}
    If we make $U$ small enough, then Equation \eqref{eqn: upper bound on norm of dPhi_0 of e_j} implies that 
    $$
    \left| \left\langle d\widetilde \Phi_0(\widetilde R(\tilde z)) (e_1), e_2 \right \rangle \right| > \gamma^{\frac{1}{4}} > 0
    $$
    for all $\tilde z \in \overline{\kappa_0(U)}$. We can also arrange for
    $$
    \left| 1 - \min_{\tilde z \in \overline{\kappa_0(U_1)}} \frac{\gamma^{\frac{1}{4}}}{\sqrt{\eta_2(\tilde z)}} \right| \gamma \gg \max_{\tilde z \in \overline{\kappa_0(U)}} \left( \max_{u \in \mathbb R^2: \langle u, u \rangle = 1} \left| \left \langle d\widetilde \Phi_0(\widetilde R(\tilde z)) (u) , e_2 \right \rangle \right| \right)
    $$ 
    by choosing $U$ small enough again. In fact, notice that the term $\min_{\tilde z \in \overline{\kappa_0(U_1)}}\frac{\gamma^{\frac{1}{4}}}{\sqrt{\eta_2(\tilde z)}}$ increases as $U$ becomes small. Consequently, because $|\langle v , e_2 \rangle| \geq \gamma$, we get that the term
    $$
    \left( 1 - \int_0^1 \frac{1}{\sqrt{\eta_2(z(t))}} \left\langle d\widetilde \Phi_0(\widetilde R(z(t))) (e_1), e_2 \right \rangle dt \right) \langle v , e_2 \rangle - \int_0^1 \left \langle d\widetilde \Phi_0(\widetilde R(z(t))) (v) , e_2 \right \rangle dt
    $$
    does not vanish. So, Equation \eqref{eqn: proper of vector joining points with same displ of Poinc return map} implies that $r = 0$, which shows that $z-w = 0$. 

    \noindent
    \textbf{Case 2:} Suppose that $|\langle v,e_2\rangle | < \gamma$.

    We can write $v = \left(a\sqrt{1 - \sigma^2} \right) e_1 + \sigma e_2$, where $|\sigma | \leq \gamma $ and $a = 1 \; \text{or} \; -1$. So Equation \eqref{eqn: cond that two points maps have same image under P minus ident} becomes 
    \begin{equation}
    \label{eqn: two points same image under P- I and along horiz dir}
        \begin{split}
            r v &= r \int_0^1 \left( a \sqrt{1 - \sigma^2} + \frac{\sigma}{\sqrt{\eta_2(z(t))}} \right) d \widetilde \Phi_0( \widetilde R(z(t))) (e_1) dt \\
            & +r \sigma \int_0^1 d \widetilde \Phi_0( \widetilde R(z(t))) (e_2) dt. 
        \end{split}
    \end{equation}

    Assume that $r \neq 0$. It follows that 
    \begin{equation}
        \label{eqn: decomp of v}
        v = C_1 u_0 + C_2 u_0 + u_1 + u_2 + u_3,
    \end{equation} 
    where 
    $$
    \begin{aligned}
        &C_1 = a \sqrt{1 - \sigma^2}, \quad C_2 = \int_0^1 \frac{\sigma}{\sqrt{\eta_2(z(t))}} dt, \quad u_0 = d \widetilde \Phi_0(0)(e_1) \\
        & u(z(t)) = d\widetilde \Phi_0(\widetilde R(z(t)))(e_1) - u_0, \quad u_1 = C_1 \int_0^1 u(z(t)) dt, \quad  u_2 = \int_0^1 \frac{\sigma u(z(t))}{\sqrt{\eta_2(z(t))}} dt,  \\
        & \text{and} \quad u_3 = \sigma \int_0^1 d \widetilde \Phi_0( \widetilde R(z(t))) (e_2) dt .
    \end{aligned}
    $$
    
    Let $ 0 < \epsilon < \frac{\gamma}{30}$. If we make $U$ small enough if needed and by using Equation \eqref{eqn: upper bound on norm of dPhi_0 of e_j}, we deduce that 
    \begin{equation}
    \label{eqn: upper bound on dPhi_0 of e_2 on U}
        \gamma^{\frac{1}{4}} < |u_0| < \gamma^{-\frac{1}{4}}, \quad \gamma^{\frac{1}{4}} < |\langle u_0, e_2 \rangle|,  \quad \sup_{\tilde z \in \overline{\kappa_0(U)}}|u(z)| < \epsilon, \quad \text{and} \quad |u_3| < \gamma^{\frac{3}{4}}.
    \end{equation}
    Recalling that $|v| = 1$, Equation \eqref{eqn: decomp of v} implies the following upper bound on $C_2$: 
    \begin{equation*}
        \begin{split}
            1 + |C_1 u_0| + |u_1| + |u_3| &\geq |C_2 u_0 + u_2|  \geq |C_2| (\gamma^{\frac{1}{4}} - \epsilon)  \geq \frac{1}{2}|C_2| \gamma^{\frac{1}{4}},
        \end{split}
    \end{equation*}
    where the second inequality follows from Equation \eqref{eqn: upper bound on dPhi_0 of e_2 on U}. Using Equation \eqref{eqn: upper bound on dPhi_0 of e_2 on U} again, we see that 
    $$
    |C_2| \leq (1 + \gamma^{-\frac{1}{4}} + \epsilon + \gamma^{\frac{3}{4}}) 2 \gamma^{-\frac{1}{4}} \leq 3 \gamma^{-\frac{1}{2}}.
    $$

    See also that Equation \eqref{eqn: decomp of v} gives
    $$
    |C_1 + C_2| \gamma^{-\frac{1}{4}} \geq |(C_1 + C_2)u_0| \geq 1 - (|u_1| + |u_2| + |u_3|) \geq 1 -(\epsilon + 3 \gamma^{-\frac{1}{2}} \epsilon + \gamma^{\frac{3}{4}}),
    $$
    where in the last inequality we used the bounds found in Equation \eqref{eqn: upper bound on dPhi_0 of e_2 on U}. Thus, we see that 
    $$
    |C_1 + C_2| \geq \frac{3 \gamma^{\frac{1}{4}}}{4}. 
    $$

    Finally, we use Equation \eqref{eqn: decomp of v} to get 
    $$
    \begin{aligned}
        |\langle v, e_2 \rangle| &\geq |C_1 + C_2| |\langle u_0, e_2 \rangle| - (|u_1|+|u_2| + |u_3|) \\
        & \geq \frac{3 \gamma^{\frac{1}{2}}}{4} - (\epsilon + |C_2| \epsilon + \gamma^{\frac{3}{4}}) \\
        & \geq \frac{ \gamma^{\frac{1}{2}}}{2},
    \end{aligned}
    $$
    where in the second inequality we use the bounds from Equation \eqref{eqn: upper bound on dPhi_0 of e_2 on U}. Because we assumed that $|\langle v,e_2 \rangle| < \gamma$, we get a contradiction. Thus, we have $r=0$, which implies that $z-w = 0$. 

    In both Case 1 and Case 2, we deduce that $z-w =0$. If $z,w \in \overline{\kappa_0(U_1)} \setminus \{ \eta_2 = 0 \}$, then we get the desired result that $z=w$. On the other hand, if $w \in \overline{\kappa_0(U_1)} \cap \{ \eta_2 = 0 \}$, then we get a contradiction because $z \in \overline{\kappa_0(U_1)} \setminus \{ \eta_2 = 0 \}$ would imply that $z \neq w$. As desired, we must have that $\widetilde P(z) - z \neq \widetilde P(w) - w$ if $w \in \overline{\kappa_0(U_1)} \cap \{ \eta_2 = 0 \}$ and $z \in \overline{\kappa_0(U_1)} \setminus \{ \eta_2 = 0 \}$.

\end{proof}

We get the following corollary of Lemma \ref{lemma: injectivity for one glancing}.

\begin{corr}
    \label{corr: almost closed implies close to z_0}
    Let $m \in \mathbb N$ and $U$ be defined by Lemma \ref{lemma: Poinc return map for closed traj with m glancing pts} and $\gamma > 0$. If the assumptions in Equation \eqref{eqn: curvature assumption} hold, $m=1$, and $U$ is made small as in Lemma \ref{lemma: injectivity for one glancing}, then there exists $\epsilon(\gamma) > 0$ such that
    $$
    \text{for all} \quad z \in \kappa_0(U),\quad
    |\widetilde P(z) - z| < \epsilon(\gamma) \ \Longrightarrow\   |z| < \gamma.
    $$
\end{corr}
\begin{proof}
    Define the compact set 
    $$
    V = \{ z \in \overline{\kappa_0(U)}: |z| \geq \gamma  \},
    $$
    where $\overline{\kappa_0(U)}$ is the closure of the open set $\kappa_0(U)$. It follows that there exists $\tilde z \in V$ such that 
    $$
    |\widetilde P(\tilde z) - \tilde z| = \inf_{z \in V} |\widetilde P (z) - z|.
    $$
    As a result, Lemma \ref{lemma: injectivity for one glancing} implies that $|\widetilde P(\tilde z) - \tilde z| > 0$. Take $\epsilon(\gamma) > 0$ such that $|\widetilde P(\tilde z) - \tilde z| > \epsilon(\gamma) > 0$. Then it follows that for all $z \in \kappa_0(U)$, 
    $$
    |\widetilde P(z) - z| < \epsilon(\gamma) \Longrightarrow |z| < \gamma.
    $$
\end{proof}

We finish this section by giving a description of the image of the subsets $\kappa_0(U_{\textbf k})$, where $U_{\textbf k}$ ($\textbf k = 0,1$) are defined by Equation \eqref{eqn: partition of U based on how many boundary hitting points}, under the map $\widetilde P(z) - z$ and when the closed trajectory $\varphi^{t_{\per}}(x_0) = x_0$ has exactly one glancing point.  

\begin{lemm}
    \label{lemma: image of missing or hitting subset under returm map minus identity}
    Let $m \in \mathbb N$ and $U$ be defined by Lemma \ref{lemma: Poinc return map for closed traj with m glancing pts} and $R > 0$. Suppose that the assumptions in Equation \eqref{eqn: curvature assumption} hold, $m=1$, and $U$ is made small as in Lemma \ref{lemma: injectivity for one glancing}. Then there exist a constant $C>0$ and a small constant $\epsilon_0 > 0$ such that for any $0 < \epsilon < \epsilon_0$ and any $\textbf k \in \{ 0, 1 \}$ we have that 
    \begin{equation}
        \label{eqn: swandwiching of image of missing or hitting subset under return map minus ident}
        M_{\textbf k} (\mathbb R \times (C \epsilon^2,\infty))\cap B_{\epsilon R}(0) \; \subset \; A_{\textbf k, \epsilon} \; \subset \; M_{\textbf k} (\mathbb R \times (-C \epsilon^2,\infty)) \cap B_{\epsilon R}(0),
    \end{equation}
    where $B_{\epsilon R}(0) :=\left\{ u \in \mathbb R^2: |u| < \epsilon R \right\}$, $A_{\textbf k, \epsilon} := \left\{\widetilde P(z) - z: z \in \kappa_0(U_{\textbf k}) \right\} \cap B_{\epsilon R}(0)$, and $M_{\textbf k} \in \SO(2)$ is a rotation matrix.
    \begin{comment}
    a smooth function $f: \mathbb R^2 \rightarrow \mathbb R$ and a small constant $\epsilon_0 > 0$ such that  
    $$
    \left\{u \in \mathbb R^2: f(u) = 0 \right\} \cap B_{\epsilon_0 R}(0) = \left\{\widetilde P(z) - z: z \in \kappa_0(U) \cap \{ \eta_2 = 0 \} \right\} \cap B_{\epsilon_0 R}(0). 
    $$
    Moreover, for any $0 < \epsilon < \epsilon_0$ and any $\textbf k \in \{0,1\}$, we have
    $$
    \begin{aligned}
        &A_{\textbf k, \epsilon} := \left\{u \in \mathbb R^2: (-1)^{\textbf k}f(u) < 0 \right\} \cap B_{\epsilon R}(0) = \left\{\widetilde P(z) - z: z \in \kappa_0(U_{\textbf k}) \right\} \cap B_{\epsilon R}(0),\\
        &A_{0, \epsilon} \cup A_{1, \epsilon} = B_{\epsilon R}(0) \setminus \{ f =0 \}, \quad \text{and}\\
        &M (\mathbb R \times (C \epsilon^2,\infty))\cap B_{\epsilon R}(0) \; \subset \; A_{\textbf k, \epsilon} \; \subset \; M (\mathbb R \times (-C \epsilon^2,\infty)) \cap B_{\epsilon R}(0), 
    \end{aligned}
    $$
    \end{comment}
\end{lemm}
\begin{proof}
    Consider the function $f: \kappa_0(U) \rightarrow \mathbb R^2$ given by $f(z) = \widetilde \Phi_0(z) - z$, where $\widetilde \Phi_0$ is defined by Equation \eqref{eqn: Poinc return map in sympl coord}. Lemma \ref{lemma: no tang vec mapped to tang vec} says that $d \widetilde \Phi_0(0) N_0 - I$ is invertible, where $N_0$ is the identity matrix according to Equation \eqref{eqn: matrices in expansion of det of Poinc return map minus ident}. So, we get that $df(0)$ is also invertible; and the inverse function theorem implies that there exists $\gamma > 0$ and an open neighborhood $V \subset \mathbb R^2$ of $0$ such that 
    $$
    f: B_{\gamma}(0) \rightarrow V \quad \text{is a diffeomorphism.}
    $$

    Take $\epsilon_0 > 0$ small such that $\overline{B_{\epsilon_0 R}(0)} \subset V$, where $\overline{B_{\epsilon_0 R}(0)}$ is the closure of the open ball $B_{\epsilon_0 R}(0)$. Furthermore, by making $\epsilon_0 > 0$ small again, Corollary \ref{corr: almost closed implies close to z_0} ensures that 
    $$
    \left\{z \in \kappa_0(U): | \widetilde P(z)-z | \leq \epsilon_0 R \right\} \subset B_{\gamma}(0).
    $$
    
    Because $f$ is a diffeomorphism, we deduce that there exists a constant $L$ such that for all $u \in B_{\epsilon_0 R}(0)$ we have 
    $$
    |f^{-1}(u)| \leq L |u|.
    $$

    Let $(y_2, \eta_2)$ be the coordinates for $\kappa_0(U)$ given by Lemma \ref{lemma: relating billiard flow with flow from fried model} and $0 < \epsilon < \epsilon_0$. 
    Notice that $\left\{ u \in B_{\epsilon R}(0): \eta_2 \circ f^{-1}(u) = 0 \right\} \subset \left\{ f(y_2,0): y_2 \in \mathbb R \quad \text{and} \quad |y_2| < \epsilon R L \right\}$. So because for any $(y_2,0) \in B_{\gamma}(0)$, we have $f(y_2,0) = y_2 \, df(0)(\partial_{y_2}) + O(|y_2|^2)$, we deduce that there exists a constant $C> 0$ and a rotation $M \in \SO(2)$ such that 
    \begin{equation*}
        \label{eqn: glancing set mapped inside a small strip}
        \left\{ u \in B_{\epsilon R}(0): \eta_2 \circ f^{-1}(u) = 0 \right\} \subset M (\mathbb R \times (-C \epsilon^2, C \epsilon^2)) \cap B_{\epsilon R}(0).
    \end{equation*}
    
    For each $\textbf k \in \{0, 1 \}$, define the set 
    $$
    X_{\textbf k, \epsilon} = \left\{u \in B_{\epsilon R}(0): (-1)^{\textbf k} (\eta_2 \circ f^{-1})(u) < 0\right\},
    $$
    which is illustrated in Figure \ref{fig: half space under return map minus identity}. Notice that there exists a rotation $M_{\textbf k} \in \SO(2)$ such that 
    \begin{equation}
        \label{eqn: each half space from slice by image of glancing is close to half disk}
        M_{\textbf k} (\mathbb R \times (C \epsilon^2, \infty)) \cap B_{\epsilon R}(0) \; \subset \; X_{\textbf k, \epsilon} \; \subset \; M_{\textbf k} (\mathbb R \times (-C \epsilon^2, \infty)) \cap B_{\epsilon R}(0).
    \end{equation}

    \begin{figure}[ht]
        \centering
        \includegraphics[width=0.95\linewidth]{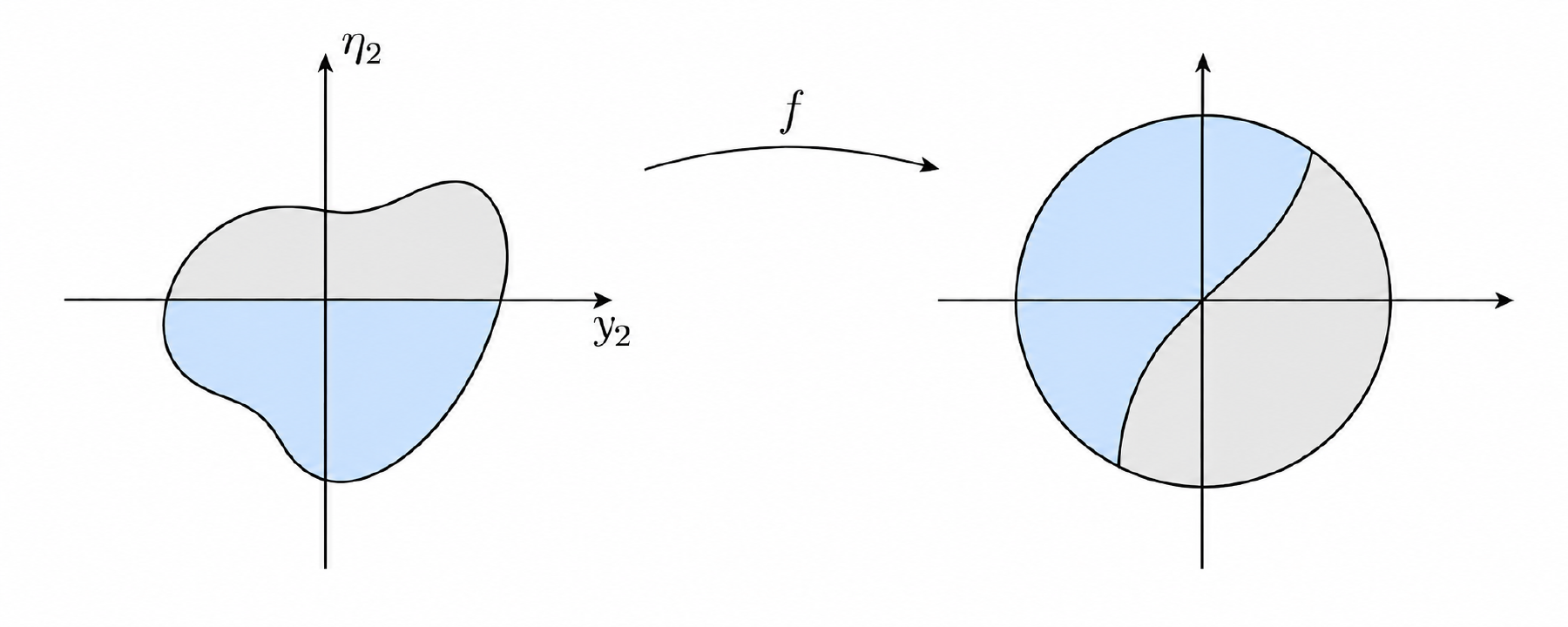}
        \caption{The light blue and light gray regions on the figure on the right represent the sets $X_{0,\epsilon}$ and $X_{1,\epsilon}$ respectively.}
        \label{fig: half space under return map minus identity}
    \end{figure}

    See that when $\textbf k = 0$, we have that $\widetilde P(z) - z = f(z)$ for any $z \in \kappa_0(U_0) \cap B_{\gamma}(0)$, so one can check that $A_{0,\epsilon} = X_{0,\epsilon}$. Then Equation \eqref{eqn: each half space from slice by image of glancing is close to half disk} finishes the proof in the case when $\textbf k = 0$. 
 
    Next we show Equation \eqref{eqn: swandwiching of image of missing or hitting subset under return map minus ident} in the case when $\textbf k = 1$. Lemma \ref{lemma: injectivity for one glancing} says that the map $\widetilde P(z) - z$ restricted on $\overline{\kappa_0(U_1)}$ is injective, so we deduce that the image of $\kappa_0(U_1)$ under the map $\widetilde P(z) - z$ does not interect the set $\left\{ u \in B_{\epsilon R}(0): \eta_2 \circ f^{-1}(u) = 0 \right\}$. 
    
    Now suppose that there exist $\tilde z, z' \in \kappa_0(U_1)$ such that 
    $$
    \widetilde P(\tilde z) - \tilde z \in X_{0, \epsilon} \quad \text{and} \quad \widetilde P(z') - z' \in X_{1,\epsilon}.
    $$
    Define $z(t) = tz' + (1-t) \tilde z \in \kappa_0(U_1)$. It follows that there exists $t_0 \in (0,1)$ such that 
    $$
    \eta_2\circ f^{-1}(\widetilde P(z(t_0)) - z(t_0)) = 0
    $$
    because $\widetilde P(z) - z$ is continuous and 
    $$
    \eta_2\circ f^{-1}( \widetilde P(z(0)) - z(0) ) < 0 \quad \text{and} \quad \eta_2\circ f^{-1}( \widetilde P(z(1)) - z(1) ) > 0
    $$
    by definition. Since for any $z \in \kappa_0(U) \cap \{ \eta_2 = 0 \}$, the map $\widetilde P(z) - z$ is given by $f(z)$, we deduce that $z(t_0) \in \kappa_0(U) \cap \{ \eta_2 = 0 \}$ because Lemma \ref{lemma: injectivity for one glancing} says that $\widetilde P(z) - z$ is injective on the set $\overline{\kappa_0(U_0)}$ or $\overline{\kappa_0(U_1)}$. Since $z(t_0) \in \kappa_0(U_1) = \kappa_0(U) \cap \{ \eta_2 > 0 \}$, we get a contradiction.  

    So without loss of generality we can assume that 
    $$
    \left\{ \widetilde P(z) - z: z \in \kappa_0(U_1) \right \} \cap B_{\epsilon R}(0) \subset X_{0,\epsilon}.
    $$
    Indeed this is the case because from the definition of $X_{0,\epsilon}$ and $X_{1,\epsilon}$, one can check that
    $$
    B_{\epsilon R}(0) = X_{0,\epsilon} \cup X_{1,\epsilon} \cup \left\{ u \in B_{\epsilon R}(0): \eta_2 \circ f^{-1}(u) = 0\right\}.
    $$

    Next we are going to show that 
    $$
    X_{0,\epsilon} \subset \left\{ \widetilde P(z) - z: z \in \kappa_0(U_1) \right \} \cap B_{\epsilon R}(0). 
    $$
    First from the continuity of $\widetilde P(z) - z$, we deduce that there exists $z_0 \in \kappa_0(U_1)$ such that $u_0 = \widetilde P(z_0) - z_0 \in B_{\epsilon R}(0)$ because $\widetilde P(0) = 0 \in B_{\epsilon R}(0)$. Consequently, we must have that $u_0 \in X_{0,\epsilon}$. Now, choose $u \in X_{0,\epsilon}$, and notice that for all $t \in [0,1]$
    $$
    u(t) = f\left(t f^{-1}(u) + (1-t)f^{-1}(u_0) \right)
    $$
    is a smooth curve in $X_{0,\epsilon}$ with $u(0) = u_0$ and $u(1) = u$. 

    Define 
    $$
    \Omega = \left\{ t \in [0,1]: u(t) \in \left\{ \widetilde P(z) - z: z \in \kappa_0(U_1) \right\} \right\} \quad \text{and} \quad t_1 = \sup \Omega.
    $$
    Then, we proceed with cases.

    \noindent
    \textbf{Case 1:} Suppose $t_1 = 1$. 
    
    It follows that there exist sequences $\{t_n\}_{n\in \mathbb N} \subset [0,1]$ and $z_n \in \kappa_0(U_1) \cap B_{\gamma}(0)$ such that $t_n \rightarrow t_1$ and $\widetilde P(z_n) - z_n = u(t_n)$. Because the set $\overline{ \kappa_0(U_1) \cap B_{\gamma}(0)}$ is compact, there exists $z_{\infty} \in \overline{ \kappa_0(U_1) \cap B_{\gamma}(0)}$ such that after taking a subsequence, we have that $z_n \rightarrow z_{\infty}$. Consequently, we get that 
    $$
    \widetilde P(z_{\infty}) - z_{\infty} = \lim_{n \rightarrow \infty} \widetilde P(z_n) - z_n = \lim_{n \rightarrow \infty} u(t_n) = u(t_1) = u.
    $$ 
    Because $u \in X_{0,\epsilon}$, it follows that $z_{\infty} \in \kappa_0(U_1)$. 

    \noindent
    \textbf{Case 2:} Suppose that $t_1 < 1$ and $t_1 \in \Omega$. 

    There exists $z_1 \in \kappa_0(U_1) \cap B_{\gamma}(0)$ such that $\widetilde P(z_1) - z_1 = u(t_1)$. Proposition \ref{prop: expans of det of Poinc ret map in coord minus ident with non-zero coeff} gives that $d\widetilde P(z_1) - I$ is invertible, so the inverse function theorem shows that $\widetilde P(z) - z$ is a local diffeomorphism at $z_1$. As a result, there exists a small constant $\sigma > 0$ such that $[t_1, t_1 + \sigma) \in \Omega$. We get a contradiction since $t_1 = \sup \Omega$. 

    \noindent
    \textbf{Case 3:} Suppose that $t_1 < 0$ and $t_1 \notin \Omega$. 

    By definition, there exist sequences $\{ t_n \}_{n \in \mathbb N} \subset [0,1]$ and $z_n \in \kappa_0(U_1) \cap B_{\gamma}(0)$ such that $t_n \rightarrow t_1$ and $\widetilde P(z_n) - z_n = u(t_n)$. After passing to a subsequence, we deduce that there exists $z_{\infty} \in \overline{\kappa_0(U_1) \cap B_{\gamma}(0)}$ such that $z_n \rightarrow z_{\infty}$. So, we get that 
    $$
    \widetilde P(z_{\infty}) - z_{\infty} = \lim_{n \rightarrow \infty} \widetilde P(z_n) - z_n = \lim_{n \rightarrow \infty} u(t_n) = u(t_1),
    $$
    which implies that $t_1 \in \Omega$. We obtain a contradiction since by assumption $t_1 \notin \Omega$. 

    Thus, we have shown that $X_{0,\epsilon} = \left\{ \widetilde P(z) - z: z \in \kappa_0(U_1) \right \} \cap B_{\epsilon R}(0) = A_{1,\epsilon}$ as desired. So Equation \eqref{eqn: each half space from slice by image of glancing is close to half disk} finishes the proof. 
\end{proof}

\section{Regularized Flat trace}
\label{sec: regularized flat trace}

Let $M = S \Sigma$ and $\widetilde M = S \widetilde \Sigma$, and notice that the billiard flow $\varphi^t: M \rightarrow M$ defined by Definition \eqref{def: sympl billiard flow for all positive times} is such that the pullback $(\varphi^t)^*$ has the following mapping properties: 
\begin{equation}
    \label{eqn: map prop of pullback by billiard flow}
    (\varphi^t)^*: C^{\infty}(M) \rightarrow \mathcal D' (M),
\end{equation}
where $C^{\infty}(M)$ is the space of smooth functions on $M$ and $\mathcal D'(M)$  the space of distributions on $M$. 

Suppose that we fix a metric $h$ on $\widetilde M$, which induces a distance function $d$ on $\widetilde M$. Let $\epsilon > 0$ and $n =  \dim \Sigma$, and define the smoothing operator
\begin{equation}
    \label{eqn: regularizer}
    \begin{split}
        &E_{\epsilon}: \mathcal D'(M) \rightarrow C_c^{\infty}(M^{\circ}) \quad \text{by}\\
        &\left(E_{\epsilon}u \right)(x) = \int_M K_{E_{\epsilon}}(x,y) u(y) dy =  \int_M \chi \left( \frac{d(x,y)}{\epsilon}\right) \left(1-w \right)\left( \frac{d\left(y,\partial M\right)}{4 \epsilon} \right) u(y)  \frac{dy}{\epsilon^{2n-1}},
     \end{split}
\end{equation}
where $M^{\circ}$ denotes the interior of $M$, $C^{\infty}_c(M^{\circ})$ is the space of smooth and compactly supported functions on $M^{\circ}$, $dy$ is the volume form $\dvol_h$, and $d(y, \partial M)$ is the distance of a point $y \in M$ to the boundary $\partial M$. The functions $\chi \,, w \in C_c^{\infty}([0,1), [0,\infty))$, where $C_c^{\infty}([0,1), [0,\infty))$ is the space of smooth and compactly supported functions in $[0,1)$ with values in $[0,\infty)$, have the properties that 
\begin{equation}
    \label{eqn: prop of regularizer}
    \begin{split}
        & \int_0^{\infty} \chi(r) r^{2n-2} dr = \frac{1}{\text{Vol}(\mathbb S^{2n-2})}, \quad w \leq 1,  \quad  \text{and} \quad w(r) = 1 \quad \text{for all} \quad r \in \left[0,\frac{1}{2}\right].\\
    \end{split}
\end{equation}
Here $\text{Vol}(\mathbb S^{2n-2})$ denotes the volume of the $2n-2$ dimensional sphere $\mathbb S^{2n-2}$. This choice of functions $\chi$ and $w$ satisfying properties in Equation \eqref{eqn: prop of regularizer} imply that for any $u \in C^{\infty}(M^{\circ})$, we have that 
\begin{equation}
    \label{eqn: regularizer approx a delta function}
    E_{\epsilon}u \rightarrow u \quad \text{in} \quad C^{\infty}(M^{\circ}) \quad \text{as} \quad \epsilon \rightarrow 0. 
\end{equation}
One can also see that
\begin{equation}
    \label{eqn: support away from bound after apply regularizer}
    \supp E_{\epsilon}u \subset \{y \in M: d(y,\partial M) > \epsilon \} \quad \text{for any} \quad u \in \mathcal D'(M).
\end{equation}

Using the smoothing operator $E_{\epsilon}$, we define a regularized pullback operator 
\begin{equation}
    \label{eqn: regularized pullback}
    \begin{split}
        &A_{\epsilon}(t) = E_{\epsilon} \left( \varphi^t \right)^* E_{\epsilon}: \mathcal D'(M) \rightarrow C^{\infty}_c(M^{\circ}),
    \end{split}
\end{equation}
which has Schwartz kernel 
\begin{equation}
    \label{eqn: kernel of regul pullback}
    \begin{split}
        &K_{A_{\epsilon}(t)}(x,y) = \int_M K_{E_{\epsilon}}(x,z) K_{E_{\epsilon}}(\varphi^t(z),y) dz,
    \end{split}
\end{equation}
where $K_{E_{\epsilon}}$ is defined by Equation \eqref{eqn: regularizer}.

\begin{lemm}
    \label{lemma: int of kernel of regularized pull back in x}
    Let $\rho \in C^{\infty}_c(M^{\circ})$. Then there exists $\epsilon_0 > 0$ and $R> 0$ such that for all $y,z \in M$ and all  $0 < \epsilon < \epsilon_0$ we have 
    \begin{equation}
    \label{eqn: int of kernel of regularized pull back in x}
        \begin{split}
             p(y,z,\epsilon) &:=\epsilon^{2n-1} \int_M \rho(x) K_{E_{\epsilon}}(x,y) K_{E_{\epsilon}}(z,x) dx \\ &= \begin{cases}
                \rho(y) \alpha_y(\epsilon^{-1} \exp_y^{-1}(z)) + O(\epsilon) & \text{if} \quad \| \epsilon^{-1} \exp^{-1}_y(z) \|_{h(y)} < R \\
                0 & \text{else}
            \end{cases} , \\
            \text{where} \quad &\alpha_y(w) = \int_{T_yM} \chi(\|v\|_{h(y)}) \chi(\|v- w \|_{h(y)}) dv \quad \text{for any} \quad w \in T_yM .
        \end{split}
    \end{equation}
    Moreover, we have that 
    \begin{equation}
        \label{eqn: psi int to 1}
        \alpha_y(w) = 0 \quad \text{for any} \quad w \in T_yM \quad \text{with} \quad \|w \|_{h(y)} \geq R \quad \text{and} \quad \int_{T_y M} \alpha_y(w) dw = 1.
    \end{equation}
\end{lemm}
\begin{proof}
    From the definition of $K_{E_{\epsilon}}$ provided in Equation \eqref{eqn: regularizer}, we have that 
    $$
    \epsilon^{2n-1}\int_M \rho(x) K_{E_{\epsilon}}(x,y) K_{E_{\epsilon}}(z,x) dx = \int_M \rho(x) \chi\left(\frac{d(x,y)}{\epsilon}\right) \chi\left(\frac{d(z,x)}{\epsilon}\right) \frac{dx}{\epsilon^{2n-1}}
    $$
    by choosing $\epsilon_0 > 0$ small enough such that for all $0 < \epsilon < \epsilon_0$ we have $(1-w)\left( \frac{d\left(x, \partial M\right)}{4\epsilon} \right) = 1$ for all $x \in \supp \rho$ and $(1-w)\left( \frac{d\left(y, \partial M\right)}{4\epsilon} \right) = 1$ for all $y \in M$ such that $d(y,x) < \epsilon$ for some $x \in \supp \rho$.

    Making the change of variables $x = \exp_y(\epsilon v)$, where $\exp_y$ is the exponential map and $v \in T_y M$, we get that 
    \begin{equation*}
        \begin{split}
            &\epsilon^{2n-1}\int_M \rho(x) K_{E_{\epsilon}}(x,y) K_{E_{\epsilon}}(z,x) dx \\
            &= \int_{T_yM} \rho(\exp_y(\epsilon v)) \chi\left(\| v \|_{h(y)}\right) \chi\left(\| v-\epsilon^{-1} \exp_y^{-1}(z) \|_{h(y)} + O(\epsilon)\right) |\det(d\exp_y(\epsilon v))|dv.
        \end{split}
    \end{equation*}
    We can then pick $R > 0$ big enough such that for any $v \in T_y M$, we have that 
    $$
    \chi\left(\| v \|_{h(y)}\right) \chi\left(\| v-\epsilon^{-1} \exp_y^{-1}(z) \|_{h(y)} + O(\epsilon)\right) = 0 \quad \text{if} \quad \| \epsilon^{-1} \exp_y^{-1}(z) \|_{h(y)} \geq R.
    $$ 
    As a result, we get that 
    $$
    \epsilon^{2n-1}\int_M \rho(x) K_{E_{\epsilon}}(x,y) K_{E_{\epsilon}}(z,x) dx = 0 \quad \text{if} \quad \| \epsilon^{-1} \exp_y^{-1}(z) \|_{h(y)} \geq R.
    $$
    In the case when $\| \epsilon^{-1} \exp_y^{-1}(z) \|_{h(y)} < R$, we use the identity $$
    d\left(\exp_y(\epsilon v), \exp_y(\epsilon w)\right) = \epsilon \| v- w \|_h(y) + O(\epsilon^3) \quad \text{for any} \quad v,w \in T_yM
    $$
    and the fact that $d \exp_y(0)$ is the identity to show that 
    $$
    \begin{aligned}
        \epsilon^{2n-1}\int_M \rho(x) K_{E_{\epsilon}}(x,y) &K_{E_{\epsilon}}(z,x) dx  \\
        &= \rho(y) \int_{T_yM} \chi\left(\| v \|_{h(y)}\right) \chi\left(\| v-\epsilon^{-1} \exp_y^{-1}(z) \|_{h(y)}\right) dv + O(\epsilon).
    \end{aligned}
    $$

    Notice that if $R> 0$ is big enough, we can ensure that 
    $$
    \alpha_y(w) = 0 \quad \text{for any} \quad w \in T_yM \quad \text{with} \quad \| w\|_{h(y)} \geq R \,.
    $$
    Finaly, using Equation \eqref{eqn: prop of regularizer}, one can show that 
    $$
    \int_{T_yM} \alpha_y(w) dw = \left( \int_{T_y M} \chi(\| w \|_{h(y)}) dw \right)^2 = 1.
    $$
\end{proof}

Using a similar argument to the proof of Lemma B.1 in Appendix B of \cite{DZ16}, we show the following lemma: 
\begin{lemm}
    \label{lemma: local trace formula for regularized billiard flow on fns}
    Suppose that $x_0 \in F \cap \{\pi^*b = 0\} \cap \{H_f \pi^*b = 0 \}$ and $t_{\per} >0 $ are such that $\varphi^{t_{\per}}(x_0) = x_0$. Let $m \in \mathbb N$ be given by Lemma \ref{lemma: Poinc return map for closed traj with m glancing pts} and $ \tilde t \in (0, t_{\per})$ be such that $y_0 = \varphi^{\tilde t}(x_0) \in M^{\circ}$. If the assumptions in Equation \eqref{eqn: curvature assumption} hold and the closed trajectory $\varphi^{t_{\per}}(x_0) = x_0$ has exactly one glancing point, that is $m=1$, then there exists $\sigma> 0$ and an open neighborhood $W$ of $y_0$ in $M^{\circ}$ such that $e^{tH_f}(y_0) \in W$ for all $|t| < \sigma$ and 
    \begin{equation}
        \label{eqn: loc trace form for regu billiard flow on functions}
        \lim_{\epsilon \rightarrow 0^+}\int_{\mathbb R} \int_W f(t)\rho(x)K_{A_{\epsilon}(t)}(x,x)\, dx dt = \frac{ c_0 \, f(t_{\per}) }{2} \int_{-\sigma}^{\sigma} \rho \left( e^{tH_f}(y_0)\right)\, dt
    \end{equation}
    for any $f \in C_c^{\infty}\left((t_{\per}-\sigma, t_{\per} + \sigma)\right)$ and any $\rho \in C_c^{\infty}(W)$, where $c_0$ is defined by Corollary \ref{corr: limits diff of Poinc retur map minus id based on reflect pts}.
\end{lemm}
\begin{proof}
     We break down the proof in multiple steps. 

     \noindent 
     \textbf{Step 1: } A coordinate adapted to the billiard flow and construction of $W$.
     \newline
     Let $U \subset M^{\circ}$ be the Poincar\'e section transversal to the billiard flow $\varphi^t$ given by Lemma \ref{lemma: Poinc return map for closed traj with m glancing pts} and suppose $U_{y_0} \subset M^{\circ}$ is a Poincar\'e section through $y_0 \in M^{\circ}$. After making the Poincar\'e sections $U$ and $U_{y_0}$ small if needed, Lemma \ref{lemma: Poinc map with strict reflections between poinc secs} says that the Poincar\'e map $Q$ between $U$ and $U_{y_0}$ is a diffeomorphism; and without loss of generality we can assume that $Q: U \rightarrow U_{y_0}$ (otherwise $Q^{-1}: U \rightarrow U_{y_0}$). 
     
     Because the billiard flow $\varphi^t$ on $M^{\circ}$ is given by the Hamiltonian flow $e^{tH_f}$, it follows that the Hamilonian flow $e^{tH_f}$ is also transversal to the Poincar\'e section $U_{y_0}$. So, there exists $\sigma_0 > 0$ such that $V = \left\{ e^{tH_f}(y): y \in U_{y_0} \; \text{and} \;\; |t| < \sigma_0 \right\} \subset M^{\circ}$.  

    Let $\kappa_0$ be the coordinate for $T \widetilde \Sigma$ near $x_0$ given by Lemma \ref{lemma: relating billiard flow with flow from fried model} and $\delta_0 > 0$ be given by Equation \eqref{eqn: Poinc secs near glanc point}. Define the set 
    $$
    S = \{ \xi' = (\xi_2, \xi_3) \in \mathbb R^2: \kappa_0^{-1}(0,\xi_2, \delta_0, \xi_3) = y \quad \text{for some} \quad y \in U \}
    $$
    and the smooth coordinate function $\psi:V \rightarrow (-\sigma_0, \sigma_0 ) \times S \subset  \mathbb R^3$ by 
    \begin{equation}
    \label{eqn: def of coord adapted to flow}
        x = \psi^{-1}(\xi_1, \xi_2, \xi_3) =  e^{\xi_1 H_f}(Q \circ \kappa_0^{-1}(0,\xi_2,\delta_0,\xi_3)).
    \end{equation}
    We see that the coordinate function $\psi$ satisfies the following properties:
    \begin{itemize}
        \item $\psi(y_0) = 0$, and
        \item $\psi(U_{y_0}) = \left\{ \xi = (\xi_1, \xi') \in \mathbb R^3: \xi_1 = 0 \right\} \cap \psi(V)$.
    \end{itemize}

    Let $0 < \sigma < \frac{\sigma_0}{4}$ and $B_{\sigma}(0) = \{ \xi \in \mathbb R^3: | \xi | < \sigma \}$. If $\sigma> 0$ is small enough, then $\psi^{-1}(B_{\sigma}(0)) \subset V$ and $\varphi^{t_{\per}}(x) \in V$ for all $x \in \psi^{-1}(B_{\sigma}(0))$. Let 
    $$
    P: U \rightarrow U \quad \text{and} \quad P_{y_0}: U_{y_0} \rightarrow U_{y_0}
    $$
    be the Poincar\'e return maps corresponding to the Poincar\'e sections $U$ and $U_{y_0}$, respectively. Then one can see that 
    $$
    P_{y_0} = Q \circ P \circ Q^{-1}. 
    $$
    So Lemma \ref{lemma: Poinc return map for closed traj with m glancing pts} implies that
    $$
    P_{y_0} = Q \circ \varphi^{t_{\per} - T(\bullet)}(\bullet) \circ Q^{-1}, 
    $$
    where $T$ is continuous with $T \circ Q^{-1}(y_0) = 0$. Recalling that Lemma \ref{lemma: Poinc map with strict reflections between poinc secs} says $Q(\bullet) = \varphi^{R(\bullet)}(\bullet)$, for some smooth function $R: U \rightarrow \mathbb R_+$, we deduce that 
    $$
    P_{y_0}(y) = \varphi^{t_{\per} - T_{y_0}(y)}(y), \quad \text{where} \quad T_{y_0}(y) = T \circ Q^{-1}(y) - R \circ P \circ Q^{-1}(y) + R\circ Q^{-1}(y).
    $$
    We can then check that $T_{y_0}(y_0) = 0$. 
    
    We use the identity $P_{y_0}(y) = \varphi^{t_{\per}-T_{y_0}(y)}(y)$ and the fact that $P_{y_0} = Q \circ P \circ Q^{-1}$ to deduce that for any $(0,\xi') \in B_{\sigma}(0)$ we have
    $$
    \Phi^{t_{\per}}(0,\xi') := \psi \circ \varphi^{t_{\per}} \circ \psi^{-1}(0,\xi') = (\widetilde T(\xi'), \widetilde P(\xi')), 
    $$
    where  $\widetilde P$ is defined by Equation \eqref{eqn: Poinc return map in sympl coord} and $\widetilde T(\xi') = T_{y_0}(\psi^{-1}(0,\xi'))$ is continuous with $\widetilde T(0) = 0$.     
    
    By making $\sigma > 0$ small again if needed, we get for any $\xi = (\xi_1, \xi') \in B_{\sigma}(0)$ and any $|t| \leq \sigma$ that
    \begin{equation}
        \label{eqn: billiard flow in special coordinates}
        \Phi^{t_{\per}+t}(\xi) = \left(\psi \circ \varphi^{t_{\per}+t} \circ \psi^{-1} \right)(\xi) =   (\xi_1 + t + \widetilde T(\xi') , \widetilde P(\xi') ).
    \end{equation}
    Define $W = \psi^{-1}(B_{\sigma}(0)) \subset M^{\circ}$, and see that $W$ is a neighborhood of $y_0$ in $M^{\circ}$ such that $e^{tH_f}(y_0) \in W$ for all $|t| < \sigma$.

    \noindent
    \textbf{Step 2: } Computing a local regularized flat trace
    
     From the definition of $K_{A_{\epsilon}(t)}$ in Equation \eqref{eqn: kernel of regul pullback}, we have that 
    \begin{equation*}
        \begin{split}
            \int_{\mathbb R} \int_W f(t) \rho(x) K_{A_{\epsilon}(t)} (x,x) dx dt &=  \int_{\mathbb R}\int_W \int_M f(t)\rho(x)K_{E_{\epsilon}}(x,y) K_{E_{\epsilon}}(\varphi^t(y),x) dy dx dt. 
        \end{split}
    \end{equation*}
    See that $x \in \supp \rho \subset W$. So, if $\epsilon$ is small enough, the definition of $K_{E_{\epsilon}}$ provided in Equation \eqref{eqn: regularizer} implies that $K_{E_{\epsilon}}(x,y) = 0$ for any $x \in \supp \rho$ and any $y \notin W$. As a result, we make the change of variables $\psi(y) = \xi = (\xi_1, \xi')$ for $y \in W$ and obtain that 
    \begin{equation*}
        \begin{split}
            \int_{\mathbb R} \int_W &f(t) \rho(x) K_{A_{\epsilon}(t)} (x,x) dx dt \\
            &= \int_{\mathbb R}\int_{B_{\sigma}(0)} \int_W f(t)\rho(x)K_{E_{\epsilon}}(x,\psi^{-1}(\xi)) K_{E_{\epsilon}}(\varphi^t(\psi^{-1}(\xi)),x) dx |\det d\psi^{-1}(\xi)|d\xi  dt \\
            & =\int_{\mathbb R}\int_{B_{\sigma}(0)} f(t) p(\psi^{-1}(\xi),\varphi^t(\psi^{-1}(\xi)),\epsilon) |\det d\psi^{-1}(\xi)| \, \epsilon^{-3} d\xi  dt, 
        \end{split}
    \end{equation*}
    where the function $p$ is defined by Lemma \ref{lemma: int of kernel of regularized pull back in x}.
    
    \begin{comment}
        \begin{equation*}
        \begin{split}
            &\int_W \rho(x)K_{A_{\epsilon}(t)} (x,x) dx =  \int_{W_{\epsilon}} \int_W \rho(x)\chi \left( \frac{d(x,y)}{\epsilon}\right) \chi\left( \frac{d\left(\varphi^t(y),x\right)}{\epsilon}\right)  \frac{dx dy}{\epsilon^{4n-2}}\\
            &= \int_{W_{\epsilon}} \int_{\frac{1}{\epsilon} \exp_y^{-1}(W)} \left(\rho  \circ \exp_y\right)(\epsilon v)\chi(\|v\|_{h(y)}) \chi \left( \left\| \frac{1}{\epsilon} \exp_y^{-1}\left(\varphi^t(y)\right)-v \right\|_{h(y)} + O(\epsilon)\right)\\
            & \quad \quad \quad \quad \left| d \exp_y(\epsilon v) \right|dv \frac{dy}{\epsilon^{2n-1}},
        \end{split}
    \end{equation*}
    where in the last equality we used the change of coordinate $x = \exp_y(\epsilon v)$ (here $\exp$ is the exponential map) and the fact that $d\left(\exp_y(\epsilon v), \exp_y(\epsilon w)\right) = \epsilon \| v- w \|_h + O(\epsilon^3)$ for any $v,w \in T_yM$.
    \end{comment}
     
     Notice that Lemma \ref{lemma: int of kernel of regularized pull back in x} also gives that 
     $$
     \begin{aligned}
         &p(\psi^{-1}(\xi),\varphi^t(\psi^{-1}(\xi)),\epsilon) = \rho(\psi^{-1}(\xi)) \, \alpha_{\psi^{-1}(\xi)}\left(\epsilon^{-1} \exp^{-1}_{\psi^{-1}(\xi)}\left(\varphi^t(\psi^{-1}(\xi))\right)\right) + O(\epsilon)\\
         &\text{if} \quad \left\| \epsilon^{-1} \exp^{-1}_{\psi^{-1}(\xi)}\left(\varphi^t(\psi^{-1}(\xi))\right) \right\|_{h(\psi^{-1}(\xi))} < R
     \end{aligned}
     $$
     and that
     $$
     p(\psi^{-1}(\xi),\varphi^t(\psi^{-1}(\xi)),\epsilon) = 0 \quad \text{otherwise},
     $$
     where $\alpha_{\psi^{-1}(\xi)}$ and $R$ are given by Lemma \ref{lemma: int of kernel of regularized pull back in x}. Let ~$(t,\xi) \in (t_{\per}-\sigma, t_{\per}+\sigma) \times B_{\sigma}(0)$ and suppose that $\varphi^t(\psi^{-1}(\xi)) = \exp_{\psi^{-1}(\xi)}(\epsilon w)$, for some $w \in T_yM$. Then, we see that 
     $$
     \Phi^t(\xi) = \left(\psi \circ \varphi^t \circ \psi^{-1}\right)(\xi) = \left(\psi \circ \exp_{\psi^{-1}(\xi)} \right) (\epsilon w).
     $$
     An expansion in the parameter $\epsilon$ shows that $d\psi^{-1}(\xi)(\Phi^t(\xi) -\xi) = \epsilon w + O(\epsilon^2)$, so if $\epsilon$ is small enough we have that
     \begin{equation}
         \label{eqn: value of inner integral}
         \begin{split}
             &p(\psi^{-1}(\xi),\varphi^t(\psi^{-1}(\xi)),\epsilon) = \rho(\psi^{-1}(\xi)) \, \alpha_{\psi^{-1}(\xi)}\left(\epsilon^{-1} d\psi^{-1}(\xi)(\Phi^t(\xi) -\xi) + O(\epsilon)\right) + O(\epsilon)\\
             &\text{if} \; \left\| \epsilon^{-1} d\psi^{-1}(\xi)(\Phi^t(\xi) -\xi) \right\|_{h(\psi^{-1}(\xi))} < R \quad \text{and} \quad p(\psi^{-1}(\xi),\varphi^t(\psi^{-1}(\xi)),\epsilon) = 0 \; \text{otherwise}.
         \end{split}
     \end{equation} 
     
     See that Equation \eqref{eqn: billiard flow in special coordinates} implies that
     $$
     \Phi^t(\xi) - \xi = (t-t_{\per} + \widetilde T(\xi'), \widetilde P (\xi') - \xi' ).
     $$
     For any $\xi = (\xi_1, \xi') \in \mathbb R^3$, define the projection map $\pi_{\xi'}$ by $\pi_{\xi'}(\xi) = \xi'$. For each $\textbf k \in \{ 0, 1\}$, the definition of the set $U_{\textbf k}$ in Equation \eqref{eqn: partition of U based on how many boundary hitting points} and the definition of $\psi$ provided by Equation \eqref{eqn: def of coord adapted to flow} imply that
     \begin{equation}
         \label{eqn: notation for subsets missing or hitting boundary}
         \widehat W_{\textbf k} := \pi_{\xi'} \left( B_{\sigma}(0) \cap \psi(U_{\textbf k}) \right) = \{ \xi' = (\xi_2, \xi_3) \in \mathbb R^2: |\xi'| < \sigma \quad \text{and} \quad (-1)^{\textbf k} \xi_3 < 0 \}.
     \end{equation} 
     Then, by Lemma \eqref{lemma: injectivity for one glancing}, we have that for any $\textbf k \in \{ 0, 1\}$ the map $\widetilde P(\xi') -  \xi'$ when restricted to the set $\widehat W_{\textbf k}$ is injective.   
   
    So, we make the change of variables 
     $$
     \epsilon(\tau,u') = \Phi^t(\xi) -  \xi = (t-t_{\per}+\widetilde T(\xi'), \widetilde P(\xi') - \xi'),
     $$
    and for each $\textbf k \in \{0 , 1 \}$, we define the sets 
    \begin{equation}
        \label{eqn: notation subsets missing or hitting boundary after eps change of coord}
        \widehat W_{\epsilon, \textbf k} = \{ u' \in \mathbb R^{2}: \quad \epsilon u' =  \widetilde P(\xi') - \xi' \quad \text{for some} \quad \xi' \in \widehat W_{\textbf k} \}
    \end{equation}
    and set 
    $$
    \widehat W_{\epsilon} = \widehat W_{\epsilon, 0} \cup \widehat W_{\epsilon, 1}.
    $$
    As a result, with 
    \begin{equation}
        \label{eqn: norm of diff of psi on ball of radius sigma}
        C = \sup_{\xi \in \overline{B_{\sigma}(0)}} \, \sup_{w \in S_{\psi^{-1}(\xi)} M}| d\psi(\psi^{-1}(\xi)) w | < \infty,
    \end{equation}
    where $\overline{B_{\sigma}(0)}$ denotes the closure of $B_{\sigma}(0)$, we get that 
    \begin{equation*}
        \begin{split}
            &\int_{\mathbb R} \int_W f(t) \rho(x) K_{A_{\epsilon}(t)} (x,x) dx dt \\
            &= \int_{-\sigma}^{\sigma} \int_{\widehat W_{\epsilon} \cap \{|u'| < CR \}} \int_{-CR}^{CR} f(\epsilon \tau + t_{\per} - \widetilde T(\xi'(\epsilon u'))) \rho(\psi^{-1}(\xi_1, \xi'(\epsilon u'))) \\
            &\quad \quad \quad \quad \quad \quad \left[ \alpha_{\psi^{-1}(\xi_1, \xi'(\epsilon u'))}\left( d\psi^{-1}(\xi_1,\xi'(\epsilon u'))(\tau, u') + O(\epsilon)\right) + O(\epsilon) \right] \\
            &\quad \quad \quad \quad \quad \quad |\det d\psi^{-1}(\xi_1, \xi'(\epsilon u'))| |\det(d \widetilde P(\xi'(\epsilon u')) - I) |^{-1} d\tau du' d\xi_1,
        \end{split}
    \end{equation*}
    where the second equality follows from the change of variables $\epsilon (\tau, u') = \Phi^t(\xi) - \xi$ and Equation \eqref{eqn: value of inner integral}.
        
    Let $\delta > 0$. There exists $\gamma > 0$ such that for all $|\xi '| < \gamma$, $|s| \leq \gamma$, $|\xi_1| \leq \sigma$, and $(\tau,u') \in \mathbb R^3 \cap B_{CR}(0)$, we have that 
    \begin{equation}
        \label{eqn: delta error for small xi'}
        \begin{split}
            &|f(t_{\per} + s - \widetilde T(\xi')) - f(t_{\per})| \leq \delta, \quad|\rho(\psi^{-1}(\xi_1,\xi')) - \rho(\psi^{-1}(\xi_1,0))| \leq \delta, \\
            &|\alpha_{\psi^{-1}(\xi_1, \xi')}(d\psi^{-1}(\xi_1,\xi')(\tau,u') + O(\gamma) ) - \alpha_{\psi^{-1}(\xi_1,0)}(d\psi^{-1}(\xi_1,0)(\tau,u'))| \leq \delta ,\\
            &|\det(d\psi^{-1}(\xi_1,\xi')) - \det(d\psi^{-1}(\xi_1,0))| \leq \delta, \quad \text{and}\\
            &\left| \frac{1}{\det(d\widetilde P(\xi')-I)} - c_{\textbf k} \right| \leq \delta \quad \text{for} \quad \textbf k \in {0,1} \quad  \text{and}  \quad \psi^{-1}(0,\xi') \in U_{\textbf k},
        \end{split}
    \end{equation}
    where $c_{\textbf k}$ are defined by Corollary \ref{corr: limits diff of Poinc retur map minus id based on reflect pts}. With the change of variables $\epsilon u' = \widetilde P(\xi') - \xi'$, where $|u'| \leq CR$, Corollary \ref{corr: almost closed implies close to z_0} implies that if $\epsilon > 0$ is small enough, then we must have that $|\xi'(\epsilon u)| < \gamma$. As a result, Equation \eqref{eqn: delta error for small xi'} implies that
    \begin{equation*}
        \begin{split}
            &\int_{\mathbb R}\int_W f(t)\rho(x) K_{A_{\epsilon}(t)} (x,x) dx dt = \\
            & = f(t_{\per}) \sum_{\textbf k=0}^1 c_{\textbf k} \int_{-\sigma}^{\sigma} \int_{ \widehat W_{\epsilon, \textbf k} \cap \{ |u'| < CR \}} \int_{-CR}^{CR} \rho (\psi^{-1}(\xi_1, 0)) \\
            & \quad \quad \quad    \alpha_{\psi^{-1}(\xi_1,0)}(d\psi^{-1}(\xi_1, 0)(\tau, u'))  \left|\det( d \psi^{-1}(\xi_1, 0
            ))\right| d\tau du' d\xi_1 + O(\delta).
        \end{split}
    \end{equation*}

    For each $\textbf k \in \{ 0, 1 \}$, Lemma \ref{lemma: image of missing or hitting subset under returm map minus identity} and Equations \eqref{eqn: notation for subsets missing or hitting boundary} and \eqref{eqn: notation subsets missing or hitting boundary after eps change of coord} defining the sets $\widehat W_{\textbf k}$ and $\widehat W_{\epsilon, \textbf k}$, respectively, imply that 
    $$
    M_{\textbf k}(\mathbb R \times(\widetilde C \epsilon, \infty)) \cap \{ |u'| < CR \} \subset \widehat W_{\epsilon,\textbf k} \cap \{ |u'| < CR \} \subset M_{\textbf k}(\mathbb R \times(-\widetilde C \epsilon, \infty)) \cap \{ |u'| < CR \}
    $$
    for some rotation matrix $M_{\textbf k} \in \SO(2)$ and some constant $\widetilde C > 0$. Then, it follows that 
    \begin{equation*}
        \begin{split}
            &\int_{-\sigma}^{\sigma}\int_{\widehat W_{\epsilon, 0} \cap \{|u'| < CR \}} \int_{-CR}^{CR} \rho (\psi^{-1}(\xi_1,\xi' = 0)) \\
            & \quad \quad \quad    \alpha_{\psi^{-1}(\xi_1,0)}(d\psi^{-1}(\xi_1, 0)(\tau, u'))  \left|\det( d \psi^{-1}(\xi_1,0
            ))\right| d\tau du' d\xi_1\\
            &\quad \quad \quad =\int_{-\sigma}^{\sigma}\int_{Y_{\epsilon , \textbf k}} \int_{-CR}^{CR} \rho (\psi^{-1}(\xi_1,\xi' = 0)) \\
            & \quad \quad \quad \quad    \alpha_{\psi^{-1}(\xi_1,0)}(d\psi^{-1}(\xi_1, 0)(\tau, u'))  \left|\det( d \psi^{-1}(\xi_1,0
            ))\right| d\tau du' d\xi_1 + O(\epsilon),
        \end{split}
    \end{equation*}
    where $Y_{\epsilon, \textbf k} = M_{\textbf k}(\mathbb R \times [0, \infty)) \cap \{ |u'| < CR \}$. Notice that 
    $$
    Y_{\epsilon, \textbf k} \cup (-Y_{\epsilon, \textbf k}) = B_{CR}(0) \quad \text{and} \quad Y_{\epsilon, \textbf k} \cap (- Y_{\epsilon, \textbf k}) = M_{\textbf k} (\mathbb R \times \{ 0\}) \cap \{ |u'| < CR \}.
    $$
    Additionally, for any $|\xi_1| < \sigma$, Equation \eqref{eqn: norm of diff of psi on ball of radius sigma} implies that the set
    $$
    d\psi^{-1}(\xi_1,0) \left( \left\{ (\tau, u') \in \mathbb R^3: |\tau| < CR \quad \text{and} \quad u' \in Y_{\epsilon} \cup (-Y_{\epsilon}) \right\} \right) 
    $$
    constains the set 
    $$
    \left\{ w \in T_{\psi^{-1}(\xi_1,0)} M : \| w \|_{h(\psi^{-1}(\xi_1,0))} < R \right\}.
    $$
    From the definition of $\alpha_{\psi^{-1}(\xi_1,0)}$ provided in Lemma \ref{lemma: int of kernel of regularized pull back in x}, we see that the function $\alpha_{\psi^{-1}(\xi_1,0)}(\bullet)$ is symmetric with respect to rotations in $T_{\psi^{-1}(\xi_1,0)}M$. Thus, we get that 
    \begin{equation*}
        \begin{split}
            &\int_{-\sigma}^{\sigma}\int_{Y_{\epsilon, \textbf k}} \int_{-CR}^{CR} \rho (\psi^{-1}(\xi_1,\xi' = 0)) \\
            & \quad \quad \quad \quad    \alpha_{\psi^{-1}(\xi_1,0)}(d\psi^{-1}(\xi_1, 0)(\tau, u'))  \left|\det( d \psi^{-1}(\xi_1,0
            ))\right| d\tau du' d\xi_1 \\
            &= \frac{1}{2}\int_{-\sigma}^{\sigma}\int_{Y_{\epsilon,\textbf k} \cup (-Y_{\epsilon, \textbf k})} \int_{-CR}^{CR} \rho (\psi^{-1}(\xi_1,\xi' = 0)) \\
            & \quad \quad \quad \quad    \alpha_{\psi^{-1}(\xi_1,0)}(d\psi^{-1}(\xi_1, 0)(\tau, u'))  \left|\det( d \psi^{-1}(\xi_1,0
            ))\right| d\tau du' d\xi_1 \\
            &= \frac{1}{2} \int_{-\sigma}^{\sigma} \rho(\psi^{-1}(\xi_1, 0))\int_{T_{\psi^{-1}(\xi_1,0)}M} \alpha_{\psi^{-1}(\xi_1,0)}(w) dw \, d\xi_1\\
            &= \frac{1}{2} \int_{-\sigma}^{\sigma} \rho(\psi^{-1}(\xi_1, 0)) du_1 = \frac{1}{2} \int_{-\sigma}^{\sigma} \rho(e^{sH_f}(y_0)) ds,
        \end{split}
    \end{equation*}
    where in the second equality we used Lemma \ref{lemma: int of kernel of regularized pull back in x} which says that $\alpha_{\psi^{-1}(\xi_1,0)}(w) = 0$ for any $w \in T_{\psi^{-1}(\xi_1,0)} M$ with $\| w \|_{h(\psi^{-1}(\xi_1,0))} \geq R$. In the third equality, we also used the same Lemma \ref{lemma: int of kernel of regularized pull back in x} that says  $\int_{T_{\psi^{-1}(\xi_1,0)}M} \alpha_{\psi^{-1}(\xi_1,0)}(w) dw = 1$. 
    
    Putting everything together, we obtain that 
    \begin{equation*}
        \begin{split}
            &\lim_{\epsilon \rightarrow 0^+} \int_{\mathbb R}\int_W f(t)\rho(x) K_{A_{\epsilon}(t)} (x,x) dx dt = \frac{c_0 \, f(t_{\per})}{2} \int_{-\sigma}^{\sigma} \rho \left( e^{sH_f}(y_0)\right) ds
        \end{split}
    \end{equation*}
    because $c_1 = 0$.
\end{proof}

In the previous Lemma \ref{lemma: local trace formula for regularized billiard flow on fns}, we take a cut off function $\rho$ supported in the interior of $M$ away from the glancing point. In the next lemma below, we study what happens when the support of $\rho$ contains a boundary point, which may be glancing or not. Specifically, we show the following.

\begin{lemm}
      \label{lemma: local trace formula for reg billiard flow near glanc pt}
    Suppose that $x_0 \in F \cap \{\pi^*b = 0\} \cap \{ H_f \pi^*b = 0 \}$ and $t_{\per} >0 $ are such that $\varphi^{t_{\per}}(x_0) = x_0$. Let $m \in \mathbb N$ be given by Lemma \ref{lemma: Poinc return map for closed traj with m glancing pts}, $\tilde t \in (0, t_{\per})$ be such that $\varphi^{\tilde t}(x_0) \in F \cap \{ \pi^*b = 0 \} \setminus \{ H_f \pi^*b = 0 \}$, and $\sigma > 0$ be small. Define $W_0 = \{x \in M: d(x,x_0) < \sigma \}$, $W_1 = \{x \in M: d(x,\varphi^{\tilde t}(x_0)) < \sigma \}$, and $W_2 = \{x \in M: d(x,\varphi^{\tilde t-0}(x_0)) < \sigma  \}$. If the assumptions in Equation \eqref{eqn: curvature assumption} hold and the closed trajectory $\varphi^{t_{\per}}(x_0) = x_0$ has exactly one glancing point, that is $m=1$, then for any $f \in C_c^{\infty}\left((t_{\per}-\sigma, t_{\per} + \sigma)\right)$ and any $\rho \in C_c^{\infty}(W_j)$ ($j \in \{0, 1, 2 \}$) there exists a constant $C > 0$ such that 
    \begin{equation}
        \label{eqn: loc trace form for regu billiard flow near glanc pt}
        \lim_{\epsilon \rightarrow 0^+}\int_{\mathbb R} \int_{W_j} f(t)\rho(x)K_{A_{\epsilon}(t)}(x,x)\, dx dt < C \, \sigma.
    \end{equation}
\end{lemm}
\begin{proof}
    Let $\epsilon > 0$ and $j \in \{0, 1, 2\}$. We have that 
    $$
    \int_{\mathbb R} \int_{W_j} f(t)\rho(x)K_{A_{\epsilon}(t)}(x,x)\, dx dt < C_1 \int_{t_{\per}-\sigma}^{t_{\per}+\sigma} \int_{W_j} K_{A_{\epsilon}(t)}(x,x) dx dt 
    $$
    for some constant $C_1 > 0$. Using the definition of $K_{A_{\epsilon}(t)}$ provided by Equation \eqref{eqn: kernel of regul pullback} and the definition of $K_{E_{\epsilon}}$ in Equation \eqref{eqn: regularizer}, we get that 
    \begin{equation}
    \label{eqn: initial upper bound for local regularized trace}
        \begin{split}
            \int_{t_{\per}-\sigma}^{t_{\per}+\sigma} \int_{W_j} &K_{A_{\epsilon}(t)}(x,x) dx dt \\
            &\leq \int_{t_{\per}-\sigma}^{t_{\per}+\sigma} \int_{W_j} \int_M \chi \left( \frac{d(x,z)}{\epsilon} \right) \chi \left( \frac{d(\varphi^t(z),x)}{\epsilon} \right) \epsilon^{-6} dz dx dt \\
            &\leq \sup \chi \int_{t_{\per}-\sigma}^{t_{\per}+\sigma} \int_{W_j} \int_M \chi \left( \frac{d(x,z)}{\epsilon} \right) \mathbbm 1_{[0,1]} \left( \frac{d(\varphi^t(z),z)}{2\epsilon} \right) \epsilon^{-6} dz dx dt \\
            & \leq C_2 \int_{t_{\per}-\sigma}^{t_{\per}+\sigma} \int_{W_{j,\epsilon}} \mathbbm 1_{[0,1]} \left( \frac{d(\varphi^t(z),z)}{2  \epsilon} \right) \epsilon^{-3} dz dt, 
        \end{split}
    \end{equation}
    where $W_{j,\epsilon} = \{z \in M : d(z,x) \leq \epsilon \quad \text{for some} \quad x \in W_j \}$ and $\mathbbm 1_{[0,1]}$ is the indicator function on the set $[0,1]$. 
    
    If $\psi$ is any smooth coordinate function for $\widetilde M$ near $x_j$, where $x_1 = \varphi^{\tilde t}(x_0)$ and $x_2 = \varphi^{\tilde t - 0}(x_0)$, then we would like to rewrite the upper bound in Equation \eqref{eqn: initial upper bound for local regularized trace} using the coordinate function $\psi$. So, let $(t,z) \in (t_{\per}-\sigma, t_{\per}+\sigma) \times W_{j,\epsilon}$, $d(z,\varphi^t(z)) \leq 2\epsilon$, and $\exp_z^{-1}(\varphi^t (z)) = 2\epsilon w \in T_z M$. Set $z = \psi^{-1}(\xi)$ and see that 
    $$
    \xi(t) = (\psi \circ \varphi^t \circ \psi^{-1})(\xi) = \left(\psi \circ \exp_{\psi^{-1}(\xi)}\right)(2\epsilon w),
    $$
    which implies that 
    $$
    2\epsilon w = d\psi^{-1}(\xi) (\xi(t) - \xi) + O(\epsilon ^2).
    $$
    We deduce the inequality
    \begin{equation}
    \label{eqn: initial upper bound for local regularized trace in loc coords}
    \begin{split}
        \int_{\mathbb R} &\int_{W_j} f(t)\rho(x) K_{A_{\epsilon}(t)}(x,x)\, dx dt \\
        &< C_1 C_2 \int_{t_{\per}-\sigma}^{t_{\per}+\sigma} \int_{\psi(W_{j,\epsilon})} \mathbbm 1_{[0,1]} \left( \left \|(2\epsilon)^{-1} d\psi^{-1}(\xi) (\xi(t) - \xi) + O(\epsilon) \right\|_{h(\psi^{-1}(\xi))} \right) \epsilon^{-3} d\xi dt. 
    \end{split}
    \end{equation}

    Now we proceed with cases. 

    \noindent
    \textbf{Case 1:} Assume that $j = 0$. 
    
    The set $W_{0,\epsilon}$ contains the glancing point $x_0$. We classify $(t,z) \in (t_{\per} - \sigma, t_{\per}+\sigma) \times W_{0,\epsilon}$ based on whether the Hamiltonian flow $e^{tH_f}$ starting at $z$ and $\varphi^t(z)$ miss or hit the boundary of $M$ after some small positive or negative time. To make this classification, we introduce the coordinate $\kappa_0$ and the variables $(y_2, \eta_1, \eta_2)$ defined in Lemma \ref{lemma: relating billiard flow with flow from fried model}. 
    
    The parameter $\sigma > 0$ is small enough such that $W_{0,\epsilon}$ and $\varphi^{t}(W_{0,\epsilon})$ are in the domain of $\kappa_0$ for all $t \in (t_{\per}-\sigma, t_{\per}+ \sigma)$; so for some $L > 0$ we have  
    \begin{equation}
        \label{eqn: image of W_0 by kappa inside ball of rad prop to sigma }
        \kappa_0(W_{0,\epsilon}) \subset B_{L\sigma}(0) \cap \left\{ \eta_1^2 \geq \eta_2 \right\},
    \end{equation}
    where $B_{L\sigma}(0) = \{(y_2, \eta_1, \eta_2) \in \mathbb R^3: |(y_2, \eta_1, \eta_2)| < L \sigma \}$. Define the set 
    $$
    \widetilde V = (t_{\per} - \sigma, t_{\per} + \sigma) \times \kappa_0(W_{0,\epsilon}),
    $$
    and for each $(t, y_2,\eta_1, \eta_2) \in \widetilde V$ let
    $$
    (y_2(t),\eta_1(t),\eta_2(t)) = \kappa_0 \circ \varphi^t \circ \kappa_0^{-1}(y_2, \eta_1, \eta_2). 
    $$
    
    Moreover, define
    $$
    \widetilde V_1 = \{ (t,y_2, \eta_1, \eta_2) \in \widetilde V: \eta_2 > 0 \quad \text{and} \quad \eta_1(t) \, \eta_1 < 0 \} \quad \text{and} \quad \widetilde V_2 = \widetilde V \setminus \widetilde V_1.
    $$
    If $(t,y_2,\eta_1, \eta_2) \in \widetilde V_1$, then the Hamiltonian flow $e^{tH_f}$ starting at $z = \kappa_0^{-1}(y_2,\eta_1,\eta_2)$ intersects the boundary of $M$ after some small positive (negative) time while the Hamiltonian flow $e^{tH_f}$ starting at $\varphi^t(z)$ may intersect the boundary of $M$ after some small negative (positive) time.

    \noindent
    \textbf{Subcase 1.1:} Suppose that $(t, y_2,\eta_1,\eta_2) \in \widetilde V_1$.

    If we assume that for some constant $K>0$ we have 
    $$
    |(y_2, \eta_1, \eta_2) - (y_2(t), \eta_1(t), \eta_2(t))| \leq K \epsilon,
    $$
    then we get that 
    $$
    |\eta_1 - \eta_1(t)| \leq K \epsilon, \quad \text{which gives} \quad |\eta_1| \leq K \epsilon \quad \text{and} \quad |\eta_1(t)| \leq K \epsilon
    $$
    because the inequality $\eta_1(t) \; \eta_1 < 0$ implies that $\eta_1(t)$ and $\eta_1$ have opposite signs. Since $\eta_1^2 \geq \eta_2$ and $\eta_2 > 0$ for any $(t,y_2,\eta_1, \eta_2) \in \widetilde V_1$, we deduce that 
    $$
    |\eta_2| \leq K^2 \epsilon^2.
    $$
    Set $\psi = \kappa_0$, so we conclude that 
    $$
    \int_{\widetilde V_1} \mathbbm 1_{[0,1]} \left( \left \|(2\epsilon)^{-1} d\kappa_0^{-1}(\xi) (\xi(t) - \xi) + O(\epsilon) \right\|_{h(\kappa_0^{-1}(\xi))} \right) \epsilon^{-3} d\xi dt \leq C_3 \sigma^2
    $$
    for some constant $C_3 > 0$.
    
    \noindent
    \textbf{Subcase 1.2:} Suppose that $(t, y_2, \eta_1, \eta_2) \in \widetilde V_2$.
    
    Define the following subsets 
    $$
    \begin{aligned}
        \widetilde V_2^{\pm} = &\{(t,y_2,\eta_1, \eta_2) \in \widetilde V_2: \eta_2 > 0, \; \pm \eta_1 > 0 \; \text{and} \; \pm \eta_1(t) > 0 \} \cup \\
        &\{ (t,y_2,\eta_1, \eta_2) \in \widetilde V_2: \eta_2 < 0 \; \text{and} \; \pm \eta_1(t) > 0 \}
    \end{aligned}
    $$
    of $\widetilde V_2$, and see that $\widetilde V_2 \setminus \{ \eta_2 = 0 \} = \widetilde V_2^+ \cup \widetilde V_2^-$. If $(t,y_2, \eta_1, \eta_2) \in \widetilde V_2^+$ and $z = \kappa_0^{-1}(y_2, \eta_1, \eta_2)$, then the Hamiltonian flow $e^{tH_f}$ starting at $z$ or $\varphi^t(z)$ does not intersect the boundary of $M$ for all negative times $\tau$ with $|\tau| \leq s$, where $s> 0$ is small but much bigger than $\sigma$.
    
    \noindent
    \textbf{Subsubcase 1.2.1:} Suppose that $(t,y_2, \eta_1, \eta_2) \in \widetilde V^+_2$.
    
    See that for any $s>0$ small,  we have that 
    $$
    (y_2 , \eta_1 + s, \eta_2) \in \{ \eta_1^2 > \eta_2 \} \quad \text{and} \quad (y_2(t), \eta_1(t) + s, \eta_2(t)) \in \{ \eta_1^2 > \eta_2 \}.
    $$ 
    Since
    $$
    \begin{aligned}
        &e^{sH_{y_1}}(y_2,\eta_1 + s, \eta_2) = (y_2, \eta_1, \eta_2) \quad \text{and} \\
        &e^{sH_{y_1}}(y_2(t),\eta_1(t) + s, \eta_2(t)) = (y_2(t), \eta_1(t), \eta_2(t)),
    \end{aligned}
    $$
    we use Lemma \ref{lemma: relating billiard flow with flow from fried model} to deduce that for any $s > 0$ small
    $$
    e^{-sH_f}(x) \in M^{\circ} \quad \text{and} \quad e^{-sH_f}(\varphi^t(x)) \in M^{\circ} \quad \text{for all} \quad (t,x) \in V_2^{+},
    $$
    where $V_2^+ = \{(t,x) \in M: (t,\kappa_0(x)) \in \widetilde V_2^+ \}$.
    
    Moreover, for each $s>0$ small there exists a constant $C > 0$ such that for any $(t,x) \in V_2^+$ with $d(x,\varphi^t(x)) \leq 2 \epsilon$ we have 
    $$
    d(e^{-sH_f}(x),e^{-sH_f}(\varphi^t(x))) \leq C d(x,\varphi^t(x)) \leq C \epsilon.
    $$
    But 
    $$
    e^{-sH_f}(\varphi^t(x)) = \varphi^t(e^{-sH_f}(x))
    $$
    because $\varphi^t(x) \in M^{\circ}$, $e^{-sH_f}(\varphi^t(x)) \in M^{\circ}$ for any $s>0$ small, and the billiard flow on $M^{\circ}$ is given by the Hamiltonian flow generated by $H_f$. Consequently, we obtain
    $$
    d(e^{-sH_f}(x),\varphi^t(e^{-sH_f}(x))) \leq C \epsilon,
    $$ 
    and Equation \eqref{eqn: initial upper bound for local regularized trace} implies
    $$
    \int_{V_2^+} K_{A_{\epsilon}(t)}(x,x) dx dt \leq C_2 \int_{V_{2,s}^+} \mathbbm 1_{[0,1]} \left( \frac{d(\varphi^t(z),z)}{C  \epsilon} \right) \epsilon^{-3} dz dt,
    $$
    where $V_{2,s}^+ = \{(t,e^{-sH_f}(x)): (t,x) \in V_2^+ \}$.
    
    Let $s \gg \sigma$ be small. Then there exists an open neighborhood $W_s \subset M^{\circ}$ of $e^{-sH_f} (x_0)$ such that for all $(t,x) \in V^+_2$, we have  
    $$
    e^{-sH_f}(x), e^{-sH_f}(\varphi^t(x)) \in W_s \; \text{(if $\sigma$ is small enough)}. 
    $$
    Consider the smooth coordinate function $\psi: W_s \rightarrow \mathbb R^3$ defined by 
    $$
    x = \psi^{-1}(\xi_1,\xi_2,\xi_3) = e^{\xi_1H_f }(\kappa_0^{-1}(\xi_2, s, \xi_3)).
    $$
    Notice that we obtain $x = \psi^{-1}(\xi_1, \xi_2, \xi_3) \in W_s$ by starting from $\kappa_0^{-1}(\xi_2,s,\xi_3)$ on the Poincar\'e section $S_s = \kappa_0^{-1} (\{ \eta_1 = s \} )$ and following the Hamiltonian flow $e^{tH_f}$ for time $\xi_1$.  
    
    Define 
    $$
    \widetilde V_{2,s}^+ = \left\{ (t, \psi(z)): (t,z) \in V_{2,s}^+ \right\}
    $$
    and set $\xi = (\xi_1, \xi') \in \mathbb R^3$. For each $(t,\xi) \in \widetilde V_{2,s}^+$, Lemma \ref{lemma: Poinc return map for closed traj with m glancing pts} gives that 
    $$
    \xi(t) = \psi \circ \varphi^t \circ \psi^{-1}(\xi_1,\xi') = (\xi_1 + t-t_{\per}+\widetilde T(\xi'),\widetilde P(\xi')), 
    $$
    where $\widetilde T$ is continuous with $\widetilde T(0) = 0$ and $\widetilde P$ is defined by Equation \eqref{eqn: Poinc return map in sympl coord}.
    
    The change of variables $C \epsilon \, (\tau , u') = \xi(t) - \xi = (t-t_{\per}+ T(\xi'), \widetilde P(\xi') -\xi')$ and a similar argument to the derivation of Equation \eqref{eqn: initial upper bound for local regularized trace in loc coords} give
    $$
    \begin{aligned}
        &\int_{V_2^+} f(t)\rho(x) K_{A_{\epsilon}(t)}(x,x)\, dx dt \\
        &< C_1 C_2 C_3 \int_{-K \sigma}^{+ K \sigma} \int_{\mathbb R^3} \mathbbm 1_{[0,1]} \left( \left \| d\psi^{-1}(\xi_1, \xi'(\epsilon u')) (\tau, u') + O(\epsilon) \right\|_{h(\psi^{-1}(\xi_1,\xi'(\epsilon u')))} \right)  d\tau du'  d \xi_1 
    \end{aligned}
    $$
    where $K, C_3 > 0$ are constants. In fact there exists a constant $K>0$ such that if $(t,\xi_1, \xi') \in \widetilde V_{2,s}^+$, then $|\xi_1| \leq K \sigma$ because  $e^{s H_f}(\psi^{-1}(\xi_1, \xi')) \in W_{0,\epsilon}$. As for the constant $C_3$, notice that Corollary \ref{corr: limits diff of Poinc retur map minus id based on reflect pts} implies that there exists a constant $C_3 > 0$ such that for all 
    $\xi' \in \{ \xi': |\xi'| \leq \widetilde D \sigma \quad \text{and} \quad \xi_3 \neq 0 \}$, where $\widetilde D > 0$ is some constant, we have $|\det(d \widetilde P(\xi') - I) |^{-1} < C_3$. Thus, we get the desired result that 
    $$
    \int_{V_2^+} f(t)\rho(x) K_{A_{\epsilon}(t)}(x,x)\, dx dt  \leq \widetilde C \sigma,
    $$
    for some constant $\widetilde C>0$.

    \noindent
    \textbf{Subsubcase 1.2.2:} Suppose that $(t,y_2, \eta_1, \eta_2) \in \widetilde V^-_2$.

    The argument is similar to the argument in Subcase 1.2.1 except that we choose instead a small negative time $s$. 

    \bigskip
    \noindent
    \textbf{Case 2:} Assume that $j=1$ or $j=2$.
    
    We give a proof when $j=1$ which can be adapted with minor modifications when $j=2$. 
    
    Consider the set $W_{1,\epsilon}$ and see that for $s>0$ small we have $W_{1,\epsilon, s}:= e^{sH_f}(W_{1,\epsilon}) \subset M^{\circ}$. Similar to Subsubcase 1.2.1, Equation \eqref{eqn: initial upper bound for local regularized trace} implies 
    $$
    \int_{W_{1}} K_{A_{\epsilon}(t)}(x,x) dx dt \leq C_2 \int_{W_{1,\epsilon, s}} \mathbbm 1_{[0,1]} \left( \frac{d(\varphi^t(z),z)}{C  \epsilon} \right) \epsilon^{-3} dz dt,
    $$
    where $C>0$ is some constant.
    
    Take a Poincar\'e section $U_s$ of the billiard flow $\varphi^t$ through the point $e^{sH_f}(x_1) \in W_{1,\epsilon,s}$. Let $\psi: W_{1,\epsilon,s} \rightarrow \mathbb R^3$ be the coordinate chosen in Step 1 in the proof of Lemma \ref{lemma: local trace formula for regularized billiard flow on fns} and define 
    $$
    \widetilde W_{1,\epsilon,s} = \{(t,\xi) \in (t_{\per}-\sigma, t_{\per} + \sigma) \times \mathbb R^3: z = \psi^{-1}(\xi) \in W_{1, \epsilon,s} \quad \text{and} \quad d(z,\varphi^t(z)) \leq C \epsilon \}.
    $$
    For any $(t,\xi) \in \widetilde W_{1, \epsilon, s}$, we have 
    $$
    \xi(t):= \Phi^t(\xi) = \psi \circ \varphi^t \circ \psi^{-1} (\xi) = (\xi_1 + t-t_{\per} + \widetilde T(\xi'), \widetilde P(\xi')),
    $$
    where $\xi = (\xi_1, \xi')$, $\widetilde P$ is defined by Equation \eqref{eqn: Poinc return map in sympl coord}, and $\widetilde T$ is a continuous function with $\widetilde T(0) = 0$.

    An argument similar to the argument presented in the last paragraph of Subsubcase 1.2.1 shows that there exists a constant $\widetilde C$ such that 
    $$
    \int_{W_1} f(t)\rho(x) K_{A_{\epsilon}(t)}(x,x)\, dx dt  \leq \widetilde C \sigma.
    $$ 

\end{proof}

We finally give a proof of Theorem \ref{thm: trace formula near fixed closed traj with one grazing point}, which is the main theorem of this paper.
\begin{proof}[Proof of Theorem \ref{thm: trace formula near fixed closed traj with one grazing point}]
    The proof is a direct application of Lemma \ref{lemma: local trace formula for regularized billiard flow on fns} and Lemma \ref{lemma: local trace formula for reg billiard flow near glanc pt}. 

    Let $\delta > 0$ be small and $(t_0,x_0) \in \supp f \times M$. We proceed with cases. 

    \noindent
    \textbf{Case 1:} Suppose that $x_0 \in M^{\circ}$ and $\varphi^{t_0}(x_0) = x_0$.
    
    We take $\sigma > 0$ and $W$ provided by Lemma \ref{lemma: local trace formula for regularized billiard flow on fns}, and this same lemma gives
    $$
    \lim_{\epsilon \rightarrow 0^+}\int_{\mathbb R} \int_M \lambda(t)\rho(x)K_{A_{\epsilon}(t)}(x,x)\, dx dt = \frac{ c_0 \, \lambda(t_{0}) }{2} \int_{-\sigma}^{\sigma} \rho \left( e^{tH_f}(y_0)\right)\, dt
    $$
    for any $\lambda \in C_c^{\infty}(t_0 - \sigma, t_0 + \sigma)$ and $\rho \in C_c^{\infty}(W)$. 

    \noindent
    \textbf{Case 2:} Suppose that $x_0 \in \partial M$ and $\varphi^{t_0}(x_0) = x_0$ or $\varphi^{t_0-0}(x_0) = x_0$.
    
    If $x_0 \in \partial M\setminus S\partial \Sigma$ and $\varphi^{t_0}(x_0) = x_0$, then we take $0 <\sigma < \delta$ and $W = W_1$, where $W_1$ is defined by Lemma \ref{lemma: local trace formula for reg billiard flow near glanc pt}. And if $x_0 \in \partial M\setminus S\partial \Sigma$ and $\varphi^{t_0-0}(x_0) = x_0$, then we take $0 <\sigma < \delta$ and $W = W_2$, where $W_2$ is defined by Lemma \ref{lemma: local trace formula for reg billiard flow near glanc pt}. 
    
    On the other hand, if $x_0 \in S\partial \Sigma$, then we choose $0 < \sigma <\delta$ and $W = W_0$, where $W_0$ is defined by Lemma \ref{lemma: local trace formula for reg billiard flow near glanc pt}. 

    As a result, Lemma \ref{lemma: local trace formula for reg billiard flow near glanc pt} shows that for any $\lambda \in C_c(t_0-\sigma, t_0 + \sigma)$ and $\rho \in C_c^{\infty}(W)$ there exists $C$ such that 
    $$
    \lim_{\epsilon \rightarrow 0^+}\int_{\mathbb R} \int_{M} \lambda(t)\rho(x)K_{A_{\epsilon}(t)}(x,x)\, dx dt < C \, \delta.
    $$

    \noindent
    \textbf{Case 3:} Suppose that $\varphi^{t_0}(x_0) \neq x_0$ and $\varphi^{t_0-0}(x_0) \neq x_0$.

    There exists $r > 0$ such that $d(\varphi^{t_0}(x_0), x_0) > r$ and $d(\varphi^{t_0-0}(x_0),x_0) > r$. Moreover, there exists $\sigma > 0$ and a neighborhood $W$ of $x_0$ in $M$ such that $d(\varphi^t(x),x) > \frac{r}{2}$ for all $t \in (t_0-\sigma, t_0 + \sigma) \times W$. If $\epsilon > 0$ is small enough, we can ensure that 
    $$
    d(\varphi^t(z),z) > \frac{r}{4}
    $$
    for all $(t,z) \in (t_0-\sigma, t_0 + \sigma) \times W_{\epsilon}$, where $W_{\epsilon} = \{z \in M: d(z,x) < \epsilon \quad \text{for some} \quad x \in W \}$. As a result, the definition of $K_{A_{\epsilon}(t)}$ provided by Equation \eqref{eqn: kernel of regul pullback} implies that 
    $$
    \lim_{\epsilon \rightarrow 0}\int_{\mathbb R} \int_M \rho(t,x) K_{A_{\epsilon}(t)}(x,x) dx dt = 0
    $$
    for any $\rho \in C_c((t_0-\sigma, t_0 + \sigma) \times W)$.

    \bigskip
    The compactness of $\supp f \times M$ and a partition of unity finishes the proof.
\end{proof}

\section{Difficulties for multiple glancing points.}
\label{sec: difficulties for mult glancing points}

Here we discuss briefly the challenges in proving Theorem \ref{thm: trace formula near fixed closed traj with one grazing point} for closed billiard trajectories with multiple grazing points and which part of our argument in the case of a single grazing point fails to generalize to multiple grazing points. 

Let $\gamma$ be a closed billiard trajectory with $m$ grazing points in the sense of Lemma \ref{lemma: Poinc return map for closed traj with m glancing pts} and suppose that 
$$
P: U \rightarrow U, \quad P(y_0^-) = y_0^-
$$
is the associated Poincar\'e return map given by the same lemma. Furthermore, let $\kappa_0$ be the coordinate defined in Lemma \ref{lemma: relating billiard flow with flow from fried model} and $\widetilde P = \kappa_0 \circ P \circ \kappa_0^{-1}$ the conjugated Poincar\'e return map given by Equation \eqref{eqn: Poinc return map in sympl coord}.  

The first challenge we need to overcome to prove Theorem \ref{thm: trace formula near fixed closed traj with one grazing point} for closed trajectories with $m$ grazing points is show that for each $\textbf k \in \{ 0, 1\}^m$, the map $\widetilde P - I$ is injective on the closure of the set $\kappa_0(U_{\textbf{k}})$, where $U_{\textbf k} \subset U$ is defined by Equation \eqref{eqn: partition of U based on how many boundary hitting points}. When $m=1$, we manage to prove the injectivity of $\widetilde P - I$ in Lemma \ref{lemma: injectivity for one glancing}, but our proof does not easily generalize to $m > 1$. 

Once the issue of injectivity of $\widetilde P - I$ is resolved, still one needs to understand for all $\epsilon > 0$ small and some fixed $R>0$ the sets 
$$
V_{\textbf{k},\epsilon} = \left\{u' \in \mathbb R^2: |u'| < R \quad \text{and} \quad \epsilon u' = \widetilde P(z) - z \quad \text{for some} \quad z \in \kappa_0(U_{\textbf k}) \right\}. 
$$
This is done for $m=1$ in Lemma \ref{lemma: image of missing or hitting subset under returm map minus identity}; so we need a similar lemma when $m>1$.

After the two technical difficulties mentioned above are solved, then with minor modifications one should be able to prove a lemma analogous to Lemma \ref{lemma: local trace formula for regularized billiard flow on fns} but for $m>1$, which would give Theorem \ref{thm: trace formula near fixed closed traj with one grazing point} for closed trajectories with multiple grazing points.

%%%%%%%%%%%%%%%%%%%%%%%%%%%%%%%%%%%%%%%%%%%%%%%%%%%%%%%%%%%%%%%%%%%%%%%%%%%%%%%%
% BIBLIOGRAPHY
%%%%%%%%%%%%%%%%%%%%%%%%%%%%%%%%%%%%%%%%%%%%%%%%%%%%%%%%%%%%%%%%%%%%%%%%%%%%%%%%

\end{document}